\documentclass[11pt]{amsart}
\usepackage{amssymb,amsmath,amsthm,amsrefs}
\usepackage{appendix}
\usepackage{amssymb,amsmath}
\usepackage[mathscr]{eucal}
\usepackage{color}
\usepackage{enumerate}
\usepackage[utf8]{inputenc}
\usepackage[leftcaption]{sidecap}
\usepackage{tikz, pgfplots}
 \usepackage{tikz-cd}
 \usepackage{cancel}
 \usepackage{stmaryrd}
\usetikzlibrary{matrix,arrows,decorations.pathmorphing}
\usetikzlibrary{positioning}
\usetikzlibrary{matrix,arrows,decorations.pathmorphing, shapes.geometric}
\tikzset{>=latex}

\usepackage[T1, T2A]{fontenc}
\usepackage[utf8]{inputenc}

\usepackage{psfrag}
\usepackage{hyperref}

\DeclareFontFamily{U}{mathx}{\hyphenchar\font45}
\DeclareFontShape{U}{mathx}{m}{n}{
      <5> <6> <7> <8> <9> <10>
      <10.95> <12> <14.4> <17.28> <20.74> <24.88>
      mathx10
      }{}
\DeclareSymbolFont{mathx}{U}{mathx}{m}{n}
\DeclareFontSubstitution{U}{mathx}{m}{n}
\DeclareMathAccent{\widecheck}{0}{mathx}{"71}

\makeatletter
\newcounter{savesection}
\newcounter{apdxsection}
\renewcommand\appendix{\par
  \setcounter{savesection}{\value{section}}%
  \setcounter{section}{\value{apdxsection}}%
  \setcounter{subsection}{0}%
  \gdef\thesection{\@Alph\c@section}}
\newcommand\unappendix{\par
  \setcounter{apdxsection}{\value{section}}%
  \setcounter{section}{\value{savesection}}%
  \setcounter{subsection}{0}%
  \gdef\thesection{\@arabic\c@section}}
\makeatother

\usepackage{epstopdf}
\usepackage{subcaption}
\usepackage{graphicx}
\usepackage{verbatim}
\usepackage{amsmath,amsthm}
\usepackage{graphics}
\usepackage{color}
\usepackage{epsfig}
\usepackage{fullpage}
\usepackage{amssymb,amsmath}
\usepackage[mathscr]{eucal}
\usepackage{ upgreek }
\usepackage{tensor}

\newtheorem{theorem}[equation]{Theorem}
\newtheorem{lemma}[equation]{Lemma}

\newtheorem{proposition}[equation]{Proposition}

\newtheorem{corollary}[equation]{Corollary}

\newtheorem{definition}[equation]{Definition}

\newtheorem{question}[equation]{Question}

\newtheorem{remark}[equation]{Remark}
\newtheorem{notation}[equation]{Notation}
\newtheorem{convention}[equation]{Convention}

\numberwithin{equation}{section}

\newcommand{\R}{\mathbb{R}}
\newcommand{\N}{\mathbb{N}}

\newcommand{\K}{\mathbb{K}}
\newcommand{\T}{\mathbb{T}}

\newcommand{\Z}{\mathbb{Z}}

\newcommand{\Sph}{\mathbb{S}}
\newcommand{\eq}{{\mathrm{eq}}}
\newcommand{\shr}{{\mathrm{shr}}}
\newcommand{\Euc}{{\mathrm{Euc}}}

\newcommand{\Spheqtres}{\mathbb{S}_{\eq}^3}

\newcommand{\Thetaeq}{\Theta_{\eq}}

\newcommand{\Grp}{\mathscr{G}} 
\newcommand{\GrpT}{\mathscr{G}_{\T}}

\newcommand{\Hgrp}{\mathscr{H}}

\newcommand{\Fcal}{\mathcal{F}}
\newcommand{\Ccal}{\mathcal{C}}

\newcommand{\Lcal}{\mathcal{L}}

\newcommand{\Rcal}{\mathcal{R}}

\newcommand{\uC}{{\mathscr{C}}}

\DeclareFontFamily{U}{mathx}{}
\DeclareFontShape{U}{mathx}{m}{n}{<-> mathx10}{}
\DeclareSymbolFont{mathx}{U}{mathx}{m}{n}
\DeclareMathAccent{\widehat}{0}{mathx}{"70}
\DeclareMathAccent{\widecheck}{0}{mathx}{"71}

\newcommand{\xx}{\ensuremath{\mathrm{x}}}
\newcommand{\yy}{\ensuremath{\mathrm{y}}}
\newcommand{\zz}{\ensuremath{\mathrm{z}}}
\newcommand{\rr}{\ensuremath{\mathrm{r}}}
\newcommand{\zcheck}{\widecheck{z}}
\newcommand{\zzcheck}{\widecheck{\ensuremath{\mathrm{z}}}}

\newcommand{\abs}[1]{\left\lvert#1\right\rvert}
\newcommand{\norm}[1]{\left\|#1\right\|}
\newcommand{\cutoff}[3]{\mathbf{\Psi}\left[ #1,#2;#3 \right]}

\newcommand{\osc}{{{osc}}}

\newcommand{\SSS}{\mathsf{S}}

\newcommand{\XXX}{\mathsf{X}}
\newcommand{\YYY}{\mathsf{Y}}

\newcommand{\XXXu}{\underline{\mathsf{X}}}
\newcommand{\YYYu}{\underline{\mathsf{Y}}}
\newcommand{\ZZZu}{\underline{\mathsf{Z}}}
\newcommand{\XXXh}{\hat{\mathsf{X}}}
\newcommand{\YYYh}{\hat{\mathsf{Y}}}

\newcommand{\XXXuh}{\hat{\underline{\mathsf{X}}}}
\newcommand{\YYYuh}{\hat{\underline{\mathsf{Y}}}}

\newcommand{\skernel}{\mathscr{K}}
\newcommand{\skernelv}{\widehat{\mathscr{K}}}
\newcommand{\skernele}{\bar{\mathscr{K}}}
\newcommand{\skernelev}{\widehat{\bar{\mathscr{K}}}}

\newcommand{\Area}{\operatorname{Area}}

\newcommand{\Ric}{\operatorname{Ric}}

\newcommand{\sech}{\operatorname{sech}}

\newcommand{\Lmer}{{L_{mer}}}

\newcommand{\psicut}{{\psi_{cut}}}
\newcommand{\Psibold}{{\boldsymbol{\Psi}}}

\newcommand{\avg}{\mathrm{avg}}

\newcommand{\dist}{\mathbf{d}}

\newcommand{\Graph}{\mathrm{Graph}}
\newcommand{\inj}{\mathrm{inj}}

\newcommand{\sym}{\mathrm{sym}}
\newcommand{\asym}{\mathrm{asym}}
\newcommand{\Cylinder}{\mathrm{Cyl}}

\newcommand{\Ecal}{\mathcal{E}}
\newcommand{\Ecalu}{\underline{\mathcal{E}}}
\newcommand{\Vcal}{\mathcal{V}}

\newcommand{\Mcal}{\mathcal{M}}

\newcommand{\Pcal}{\mathcal{P}}
\newcommand{\Jcal}{\mathcal{J}}
\newcommand{\BPcal}{B_{\Pcal}}

\newcommand{\kappau}{\underline{\kappa}}
\newcommand{\zetabold}{\boldsymbol{\zeta}}

\newcommand{\zetaboldu}{\underline{\zetabold}}

\newcommand{\varphihat}{\hat{\varphi}}

\newcommand{\Ku}{\underline{K}}
\newcommand{\Kcech}{\widecheck{K}}

\newcommand{\bbracket}[1]{[\![#1]\!]}

\newcommand{\radius}{\varepsilon}

\newcommand{\Mbreve}{\breve{M}}
\newcommand{\ubreve}{\breve{u}}
\newcommand{\Mbrevem}{\breve{M}_m}

\newcommand{\Mshrm}{\breve{M}_{\shr,m}}
\newcommand{\Sigmabreve}{\breve{\Sigma}}

\newcommand{\Spheq}{\mathbb{S}^2_{\eq}}
\newcommand{\Spheqn}{\mathbb{S}^n_{\eq}}
\newcommand{\Spheqt}{\mathbb{S}^3_{\eq}}

\newcommand{\Sphshrt}{\mathbb{S}^3_{\shr}}
\newcommand{\sgr}[2]{\mathrm{Isom}_{#1}^{#2} } 
\newcommand{\xL}{\boldsymbol{\xi}}

\newcommand{\xbrevering}{\boldsymbol{\breve{\mathring{\xi}}}}         
\newcommand{\ximm}{\xbrevering_{m\times m}}
\newcommand{\gshr}{g_{\shr}}

\newcommand{\Sphshrtres}{\Sph^3_{\shr}}
\newcommand{\Sphst}{\Sph^3_{\shr}}
\newcommand{\Sphsn}{\Sph^n_{\shr}}
\newcommand{\Hess}{\operatorname{Hess}}

\title[Minimal hypersurfaces in $\mathbb{S}^{4}(1)$]{Minimal hypersurfaces in $\mathbb{S}^{4}(1)$ \\ by doubling the equatorial $\mathbb{S}^{3}$}  

\author[N.~Kapouleas]{Nikolaos~Kapouleas}
\address{Department of Mathematics, Brown University, Providence, RI 02912} 
\email{nicolaos\_kapouleas@brown.edu}
	   
\author[J.~Zou]{Jiahua~Zou} 
\address{Department of Mathematics, Rutgers University, Pistacaway, NJ 08854} 
\email{jiahua.zou@rutgers.edu}

\begin{document}

\date{\today}


\begin{abstract}
For each large enough $m\in\mathbb{N}$ we construct by PDE gluing methods a closed embedded smooth minimal hypersurface ${\breve{M}_m}$ 
doubling the equatorial three-sphere $\mathbb{S}_{\mathrm{eq}}^3$ in $\mathbb{S}^4(1)$, 
with ${\breve{M}_m}$ containing $m^2$ bridges modelled after the three-dimensional catenoid and 
centered at the points of a square $m\times m$ lattice $L$ contained in the Clifford torus $\mathbb{T}^2\subset \mathbb{S}_{\mathrm{eq}}^3$. 
This answers a long-standing question of Yau in the case of $\Sph^4(1)$ and long-standing questions of Hsiang. 
Similarly we construct a self-shrinker ${\breve{M}_{\mathrm{shr},m}}$ of the Mean Curvature Flow in $\mathbb{R}^4$ 
doubling the three-dimensional spherical self-shrinker $\mathbb{S}_{\mathrm{shr}}^3\subset\R^4$ with the bridges 
centered at the points of a square $m\times m$ lattice $L$ contained in a Clifford torus $\mathbb{T}^2\subset \mathbb{S}_{\mathrm{shr}}^3$. 
Both constructions respect the symmetries of the lattice $L$ as a subset of $\mathbb{S}^4(1)$ or $\R^4$ and are based on the Linearized Doubling (LD) methodology 
which was first introduced in the construction of minimal surface doublings of $\mathbb{S}_{\mathrm{eq}}^2$ in $\mathbb{S}^3(1)$.
Furthermore $\breve{M}_m$ converges as $m \to\infty$ in the varifold sense to $2\mathbb{S}_{\mathrm{eq}}^3$, and its volume $|\breve{M}_m| < 2|\mathbb{S}_{\mathrm{eq}}^3|$. 
\end{abstract}

\maketitle



\section{Introduction}
\label{S:intro} 

\subsection*{The historic framework in high dimensions} 
$\vspace{.162cm}$ 
\nopagebreak

Although closed embedded minimal hypersurfaces in $\Sph^{n+1}(1)\subset\R^{n+2}$ (or surfaces for $n=2$) have been studied extensively by geometers, many fundamental questions remain open. 
In particular in the high dimensional case $n\ge3$ little is known. 
In this case the simplest examples are the equatorial $\Spheqn$ and Clifford products of spheres \cite{choe:hoppe:2018:products}. 
The remaining known examples are briefly discussed below; 
and although infinitely many noncongruent examples are known for each $n>2$, 
their topological types are very few. 

An important class of minimal examples in $\Sph^{n+1}(1)$ are the ones constructed by the methodology introduced in the early eighties by Wu-Yi Hsiang 
\cite{hsiang1982,hsiangI,hsiangII,hsiang1984,hsiangIII,hsiang1987duke,hsiang1993,CS}.  
Hsiang's approach is based on using the framework of equivariant differential geometry \cite{hsiang:lawson1971} to reduce the construcion of the minimal hypersurfaces 
to the construction of curves in the orbit space of the action of a Lie group on $\Sph^{n+1}(1)$. 
The minimality condition for hypersurfaces invariant under this action can be understood then by ODE theory. 
Hsiang's approach provides infinitely many non-congruent examples for any $n>2$ but of only few simple topological types, namely $\Sph^n$, $\Sph^1\times\Sph^{n-1}$, 
$\Sph^2\times\Sph^{n-2}$, and a few more products of spheres. 

Another class of examples are the minimal isoparametric hypersurfaces 
which have been studied extensively and classified 
(see for example \cite{chi:survey,cecil:chi:jensen-2007annals,chi:2020jdg,solomon,ki:nakagawa} and references there). 
In $\Sph^4(1)$ there is only one, a Cartan isoparametric hypersurface diffeomorphic to $SO(3)/\Z_2\times \Z_2$ 
\cite{solomon,ki:nakagawa}.  

Finally min-max methods have proved the existence of infinitely many examples in general Riemannian manifolds of dimension $\le8$, 
with important information on aspects of their asymptotic behavior  
\cite{song2023,irie:marques:neves2018,liokumovich:marques:neves2018,marques:neves2017,marques:neves:song2019,zhou2020}.  
The topological type of these hypersurfaces however is not known in general. 

Our article is partially motivated by 
a fundamental and insightful question of Yau included 
in his celebrated list of problems from the early nineties 
to which we give a satisfactory answer in the case of $\Sph^4$.  
The full question is as follows.  

\begin{question}[Yau (1993) {{\cite[Problem 32]{yau-problems1993}}} ]  
\label{Q:yau} 
Find an effective way to construct complete minimal hypersurfaces in $\R^n$ or $\Sph^n$ with finite topology and without continuous groups of isometries. 
(We interpret $\R^n$ and $\Sph^n$ to mean $\R^{n+1}$ and $\Sph^{n+1}(1)$ in our notation). 
\end{question} 

Note that until this article no progress has been made in answering this question. 
In our main theorem \ref{thm} we answer this question 
in the case of $\Sph^4$ is by doubling $\Spheqt$ using PDE gluing methods, 
and in particular the Linearized Doubling (LD) methodology introduced by NK in \cite{SdI}. 
Our construction is simplified greatly by the high (but discrete) symmetry imposed.  
We expect that by gradually reducing the amount of symmetry required we will answer fully Yau's question in all cases. 

We also recall below two questions of Hsiang on which no progress has been made until now. 

\begin{question}[Hsiang (1993) {{\cite[Problem 3]{hsiang1993}}} ]  
\label{Q:hsiang3} 
Construct embedded closed minimal hypersurfaces in $\Sph^4(1)$ which are of topological types different than 
$\Sph^3$, $\Sph^1\times\Sph^2$, $\Sph^1\times \Sph^1\times \Sph^1$, and $SO(3)/\Z_2\times \Z_2$.  
\end{question} 

\begin{question}[Hsiang (1993) {{\cite[Problem 6]{hsiang1993}}} ] 
\label{Q:hsiang6} 
Can $\quad\! \Sph^{n+1}(1)$ accommodate embedded closed minimal hypersurfaces of the diffeomorphism types of arbitrary connected sums of the products of spheres? 
\end{question}

This article answers fully the first question, 
and 
since the hypersurfaces we construct are diffeomorphic to $\#_{m^2-1} \Sph^2\times\Sph^1$ for any large enough $m\in\N$,       
it answers also a substantial part of the second question. 
	   
Embedded self-shrinkers of the mean curvature flow in $\R^{n+1}$ are also very important in geometry. 
In this case only finitely many closed examples are known for each $n>2$. 
The simplest closed examples for $n\ge2$ are the spherical self-shrinker $\Sphsn\subset\R^{n+1}$ 
and the Angenent doughnut which is diffeomorphic to $\Sph^1\times\Sph^{n-1}$ and $O(n)$-invariant \cite{angenent}. 
The remaining known closed examples are induced from isoparametric surfaces by a construction of Riedler \cite{riedler2023shrinkers} based on extensions of the Hsiang approach as in 
\cite{FK,PT,Wang1}. 
In this article we modify the earlier construction to construct closed examples in $\R^4$ 
doubling the three-dimensional spherical self-shrinker $\Sphshrt \subset\R^4$.  
These examples are of infinitely many topological types $\#_{m^2-1} \Sph^2\times\Sph^1$ for any large enough $m\in\N$ 
similarly to the minimal hypersyrfaces in $\Sph^4(1)$ in the earlier construction.      

\subsection*{The historic framework for $n=2$ and minimal surface doublings} 
$\vspace{.162cm}$ 
\nopagebreak

We discuss now for comparison the case $n=2$ including the history of the doubling methodology we will be using. 
For a long time the only examples of closed embedded minimal surfaces in $\Sph^3(1)$ known 
were the equatorial two-sphere $\Spheq$ and the Clifford torus $\T=\T^2 := \Sph^1(1/\sqrt2 ) \times \Sph^1(1/\sqrt2 )$. 
These two simple examples play an important role in the theory of minimal surfaces in $\Sph^3(1)$ and 
celebrated proofs of long-standing conjectures on the Clifford torus have been achieved by Brendle \cite{brendle:2013:lawson} and Marques-Neves \cite{marques:neves}.  

Lawson in 1970 discovered the first new examples of closed embedded minimal surfaces in $\Sph^3(1)$ which he named $\xL_{k,m}$ and have genus $km$ ($k,m\in \N$) \cite{Lawson}. 
Many more examples have been found since by expanding the original Lawson approach \cite{KPS,choe:soret}, 
by PDE gluing methods \cite{kapouleas:yang,wiygul:jdg2020,wiygul:t,SdI,SdII,kapouleas:wiygul:toridesingularization,gLD,IIgLD,douT},     
by min-max methods \cite{PRu,ketoverMN:2020,ketover:2022:flipping}, 
and by equivariant maximization of the normalized first eigenvalue \cite{KKMS}.    
Most of these examples are (or expected to be) \emph{desingularizations} in the sense of \cite{alm20}*{Definition 1.3} and the further discussion there, 
or \emph{doublings} in the sense of \cite[Definition 1.1]{gLD}, of great two-spheres or Clifford tori.   
For example the Lawson surfaces $\xL_{k,m}$ are desingularizations of $k+1$ great two-spheres intersecting symmetrically along a common great circle for $k\ge1$, $m\ge2$,  
and five of the nine minimal surfaces constructed in \cite{KPS} are doublings of $\Spheq$ in $\Sph^3(1)$. 

Of particular interest for this article are the 
doubling constructions by PDE gluing methods.  
These were first proposed and discussed in \cite{kapouleas:survey,kapouleas:yang,alm20}.
PDE gluing methods have been applied extensively and with great success in Gauge Theories by Donaldson, Taubes, and others.
The particular kind of gluing methods used relates most
closely to the methods developed in \cite{schoen} and \cite{kapouleas:annals},
especially as they evolved and were systematized in
\cite{kapouleas:wente:announce,kapouleas:wente,kapouleas:imc}.
We refer to \cite{kapouleas:survey} for a general discussion of this gluing methodology 
and to \cite{alm20} for an early general discussion of doubling by PDE gluing methods.

The first doubling constructions by PDE gluing methods produced minimal surface doublings of the Clifford torus \cite{kapouleas:yang}.  
Subsequently NK introduced in \cite{SdI} a refinement of the general methodology which he called Linearized Doubling (LD); 
using the LD methodology he constructed the first high genus minimal surface doublings of the equatorial sphere $\Spheq$ in $\Sph^3(1)$.  
In these doublings the catenoidal bridges are equidistributed either along two parallel circles of $\Spheq$, or along the equatorial circle with two more bridges at the poles. 
Since then the LD methodology has led to many new results \cite{SdII,gLD,IIgLD} and has great further potential.

\subsection*{Brief discussion of the results and the construction}   
$\vspace{.162cm}$ 
\nopagebreak

The general definition of a doubling \cite[Definition 1.1]{gLD} can be trivially extended to higher dimensions as follows. 

\begin{definition}[Hypersurface doublings {\cite[Definition 1.1]{gLD}} ]  
\label{D:dou} 
We define 
a \emph{(hyper)surface doubling $\Mbreve$ over/of a base hyper(surface) $\Sigma$ in a Riemannian manifold $(N,g)$} 
to be a smooth (hyper)surface which is the union of two smooth graphs over a smooth closed subset $\Sigmabreve\subset\Sigma$.  
(The two graphs join smoothly with vertical tangent planes along $\partial \Sigmabreve$ which we assume smooth and of codimension one in $\Sigma$).  
We call $(\Sigma, N, g)$ the \emph{background} of the doubling and the connected components of $\Sigma\setminus \Sigmabreve$ its \emph{doubling holes}. 
We call $\Mbreve$ a \emph{minimal (hyper)surface doubling} of $\Sigma$ in $N$, or a \emph{minimal doubling} for short, if both $\Mbreve$ and $\Sigma$ are minimal. 
Finally we will assume $\Mbreve$ and $\Sigma$ are connected and embedded unless otherwise stated. 
\end{definition}

In this article we construct by PDE gluing methods and for each large enough $m\in\N$ 
a closed embedded smooth minimal hypersurface $\Mbrevem$ doubling the equatorial three-sphere $\mathbb{S}_{\eq}^3$ in $\mathbb{S}^4(1)$ (see Theorem \ref{thm}).   
The hypesurface $\Mbrevem$ contains $m^2$ bridges modelled after the three-dimensional catenoid in $\R^4$,   
with each bridge centered at a point of $L$ and with its waist lying on $\Spheqt$, 
where $L=L[m]\subset\T^2 \subset \Spheqt \subset \mathbb{S}^4(1)$ is a square $m\times m$ lattice (see \ref{def:L}).  
The symmetries of $\Mbrevem$ are the same with the symmetries of $L$, 
that is in the notation of \ref{not:manifold}\ref{item:sym}  $\sgr L{\Sph^4(1)} = \sgr {\Mbrevem}{\Sph^4(1)} $. 
We also prove an asymptotic formula \eqref{E:vol} for the volume of the doublings we construct. 
\eqref{E:vol} is analogous to \cite[Theorem A(v) and (5.9)]{gLD} 
and implies in particular that our doublings have volume strictly smaller than the volume of two great three-spheres. 

Clearly the topology type of $\Mbrevem$ is (equivalently $\Mbrevem$ is diffeomorphic to) $\#_{m^2-1} \Sph^2\times\Sph^1$,   
and moreover $\Mbrevem$ is an orientable closed three-manifold with torsion--free first and second homology, both equal to $\Z^{m^2-1}$.  
This is the first construction which provides closed embedded minimal hypersurfaces in $\Sph^4(1)$ of infinitely many topological types. 
Since our minimal hypersurfaces are clearly not rotationally symmetric it answers \cite[Problem 32]{yau-problems1993} for $\Sph^4(1)$ as already mentioned (see \ref{Q:yau}), 
and also \cite[Problem 3]{hsiang1993}  (see \ref{Q:hsiang3}) and partially \cite[Problem 6]{hsiang1993} (see \ref{Q:hsiang6}). 
Note that the question of classifying all the topological types realized by closed embedded minimal hypersurfaces is much harder, 
and in fact it is not clear at the moment which orientable closed three-manifolds can be smoothly embedded into $\Sph^4$ (see \cite{Budney_2020} for a survey of known results). 

Our constructions in this article have strong similarities and differences with the constructions in both \cite{kapouleas:yang} and \cite{SdI}.  
We first recall that in \cite{kapouleas:yang},  
for each large enough $m\in\N$,  
a minimal surface doubling of the Clifford torus $\T$ in $\Sph^3(1)\simeq\Spheqt$ is constructed 
containing $m^2$ catenoidal bridges centered at the points of the square lattice $L=L[m]\subset\T^2$; we will refer to it as $\ximm$.   
The symmetries of $\ximm$ in $\Sph^3(1)$ are the same with the symmetries of $L$, 
that is in the notation of \ref{not:manifold}\ref{item:sym}  $\sgr L{\Sph^3(1)} = \sgr {\ximm   }{\Sph^3(1)} $. 
The crucial difference is of course that the bridges in $\Mbrevem$ are three-dimensional with waists on $\Spheqt$ 
while in $\ximm$ they are two-dimensional with waists (approximately) on $\T$. 

Moreover although in both cases the bridges are uniformly distributed on the Clifford torus $\T$, 
$\T$ has codimension one in the base surface $\Spheqt$ in the case of $\Mbrevem$,  
while it is the whole base surface in the case of $\ximm$. 
In this respect $\Mbrevem$ is more similar to the doublings of $\Spheq$ in $\Sph^3(1)$ constructed in \cite{SdI,SdII}, where the bridges are distributed along parallel circles.  
Actually an even closer analogy would be with doublings of $\Spheq$ in $\Sph^3(1)$ with bridges distributed along a single equatorial circle. 
Such doublings do not exist however \cite[Remark 6.31]{SdI}. 
The reason they do not exist is that the equatorial circle subdivides $\Spheq$ into two hemispheres on which the first Dirichlet eigenvalue of the Jacobi operator vanishes. 
In our case $\T$ subdivides $\Sph^3(1)$ into two congruent domains which are tubular neighborhoods of radius $\pi/4$ 
of two circles which we will denote by $\uC$ and $\uC^\perp$ (see \eqref{eq:clifford}). 
The Jacobi operator on these domains has positive first Dirichlet eigenvalue  
and this is what makes our constructions possible.

By appropriately modifying the construction of $\Mbrevem$ we also construct a three-dimensional self-shrinker $\Mshrm\subset \R^4$ 
doubling the spherical self-shrinker $\Sphst(1)\subset\R^4$ 
with three-catenoidal bridges with centers in an $m\times m$ lattice in the Clifford torus $\T \subset \Sphst(1)\subset\R^4$ (see Theorem \ref{thmshr}).  
In this case, since the two sides of the spherical shrinker are not symmetric, the three-catenoidal bridges have to be elevated to achieve balancing, 
similarly to the case of doublings of the spherical self-shrinker in $\R^3$ in \cite[Section 10]{gLD}; no tilting is needed though because of the symmetries. 

We can also generalize the construction in this article by using instead of a square lattice a $k\times m$ lattice $L=L[k,m]\subset\T$ with $k,m\in \N$ large enough; 
we can further generalize by prescribing $k_\circ \in \N$ parallel tori to the Clifford torus along which to have $km$ three-catenoidal bridges with the option to have 
$k$ bridges along $\uC$ and/or $m$ bridges along $\uC^\perp$ \cite{dS3II}. 
This generalization is similar to the generalization from \cite{SdI} to \cite{SdII} and involves horizontal balancing which 
complicates the construction and has been avoided in this article by the extra symmetries. 
In work in progress we expoit again the unusual symmetries of the Clifford torus to construct with a similar setup minimal hypersurface doublings of 
$\Sigma= \Sph^1(1/2) \times \Sph^3(\sqrt{3\,} /2 )$ (which is minimal \cite{choe:hoppe:2018:products}) in $\Sph^5(1)$ 
with four-dimensional catenoidal bridges along 
$\Sph^1(1/2) \times \frac {\sqrt{3\,}}{2} \T^2$, 
where $\frac {\sqrt{3\,}}{2} \T^2$ is the Clifford torus inside the second factor of $\Sigma$. 
Finally note that using a different setup NK constructs minimal hypersurface doublings of the equatorial $n$-dimensional sphere $\Spheqn$ in $\Sph^{n+1}(1)$ 
for any $n>2$ in \cite{dSn} as announced in \cite{SdI}.

\subsection*{Outline of strategy and main ideas}
$\vspace{.162cm}$ 
\nopagebreak

The basic approach is the same as in \cite{SdI,SdII} with some necessary modifications. 
The \emph{rotationally invariant linearized doubling (RLD) solutions} 
are now constant on the tori $\T_c$ which are parallel to the Clifford torus $\T=\T_0$ at distance $|c|$ (see \ref{eq:clifford}).  
These tori degenerate to the circles $\uC$ and $\uC^\perp$ at $c=\pm\pi/4$ and they are the orbits of the action of the group $\Hgrp$ (see \ref{not:hgrp}). 
Our RLD solutions (see \ref{def:phihat}) are invariant under this action and have a single derivative jump at the Clifford torus $\T$. 
They form a one-dimensional vector space and can be expressed in terms of hypergeometric functions as in \ref{lem:phiC}, 
although this description is not necessary for the results in this article. 

The RLD solutions are converted to Linearized Doubling (LD) solutions in 
Section \ref{section:LD} by an approach which parallels the one in dimension two \cite{SdI,SdII,gLD} .  
Our LD solutions $\varphi=\varphi\bbracket{\zeta}$ are defined in \ref{def:varphi} and they are singular solutions of the Jacobi equation with singularities 
modelled after the Green's function $1/r$ on $\R^3$ (see \ref{def:LD}), 
so the singularities are not logarithmic as in dimension two. 
They depend on a single parameter $\zeta$ which adjusts by a scalar factor $\varphi$ and the strength of its singularities.  

Each LD solution $\varphi$ is modified then to a ``matched'' LD solution  
${\varphi+\underline{v}+\bar{\underline{v}}}$ 
which depends on one more parameter $\bar{\zeta}$ and satisfies the Jacobi equation on $\Spheqt\setminus L$ modulo a two-dimensional \emph{extended substitute kernel} 
or \emph{obstruction space} $ \skernel[L] \oplus \skernele[L]$ defined in \ref{def:VW}.   
The \emph{``obstruction space content''} of 
${\varphi+\underline{v}+\bar{\underline{v}}}$ is prescribed by the parameters as in Lemma \ref{lem:miss} and in particular Equation \ref{eq:zetamu}. 
Note that the creation of ``obstruction content'' by the $\zeta$ parameter corresponds to the creation of mismatch by scaling the bridges, and by the $\bar{\zeta}$ 
parameter corresponds to replacing the connected components of $\Spheqt\setminus\T$ 
with $\Hgrp$-invariant minimal graphs at the linearized level,  
which is consistent with a generalized interpretation of the \emph{geometric principle} where the dislocation uses exceptional Jabobi fields not induced by Killing fields. 

In \ref{lem:varphinl}
we ``correct'' ${\varphi+\underline{v}+\bar{\underline{v}}}$ with a small modification to obtain 
$\varphi_{nl}=\varphi_{nl}\bbracket{\zeta,\bar{\zeta}}\in C^{\infty}_{\sym}(\Spheqtres\setminus L)$ 
whose graph satisfies (nonlinear) minimality outside a small neighborhood of $\T$ 
(see \ref{lem:varphinl}).  
The initial surfaces $ M=M[\zeta,\bar{\zeta}] $ are then constructed in \ref{def:initial} by gluing (smoothly) the catenoidal bridges to the graphs of $\pm \varphi_{nl}$.   

The control of the mean curvature in \ref{lem:glue} needs the totally geodesic condition on the base and careful analysis of the quadratic terms in \ref{lem:SphereH} 
as applied to \ref{eq:thus}. 
This leads to satisfactory global estimates for the mean curvature in \ref{lem:H}. 
The global solutions to the linear problem are estimated in Proposition \ref{prop:lineareq} where we solve modulo the obstruction space 
$ \skernel[L] \oplus \skernele[L]$ 
to ensure exponential decay away from the Clifford torus $\T$ and apropriate decay towards the waists of the catenoidal bridges (see \ref{def:norm} and \ref{def:weight}). 
We eventually close the argument and prove the main Theorem \ref{thm}. 

Finally we remark that a major difference 
is that the strength of the singularities is of order $m^{-4}$ (see \ref{eq:tau}) 
and not exponentially small as in two dimensions. 
This corresponds to size (that is waist radius or scaling parameter) of the bridges $\radius   =\sqrt{\tau}\sim m^{-2}$ (see \ref{def:initial})  
which is smaller than the distance of nearby bridges
by a factor $1/m$. This factor is not as small as the corresponding exponentially small factor in dimension two, making the details of the construction more delicate. 
Note also that the range of the parameters (as in in \ref{eq:zeta}) differs from the two-dimensional case.  

\subsection*{Organization of the presentation}
$\vspace{.062cm}$ 
\nopagebreak

The article consists of eight sections. 
The first section is the introduction with a general discussion and determination of the notation used. 
In section \ref{S:linS} we discuss the linear theory on $\Spheqt$. 
In Section \ref{S:cat} we discuss the three-dimensional catenoids and their use as models for the bridges. 
In Section \ref{section:LD} we discuss the construction and estimation of the LD solutions $\varphi$ and the auxiliary functions $\underline v$ and $\bar{\underline{v}}$.   
In Section \ref{S:init} we discuss the construction and properties of the initial surfaces. 
In section \ref{S:lin} we combine the linear theory on the three-catenoid with the linear theory of Section \ref{S:linS} to produce the global linear theory 
Proposition \ref{prop:lineareq}; we also estimate the nonlinear terms in Lemma \ref{lem:nonlinear}. 
In Section \ref{S:main} we state and prove the main Theorem \ref{thm} constructing minimal hypersurface doublings over $\Spheq$; 
we also prove an asymptotic formula \ref{E:vol} for their volume. 
Finally in section \ref{S:shr} we state and prove the main Theorem \ref{thmshr} for self-shrinkers by appropriately modifying the previous construcion.  
 

\subsection*{General notation and conventions}
\nopagebreak

	  	\begin{notation}
	  	$ f(x) \lesssim_b g(x)$ indicates that there exists some absolute constant $C=C(b)>0$ just depending on $b$ so that $ f(x) \leq C g(x)$ for all $x$ in some specified domain. We will also use this notation for functions that depend on several variables or parameters. When the constant $C$ can be chosen absolutely, we will just write $f(x) \lesssim g(x)$.
	  	
	  	$ f(x) \sim_b g(x)$ indicates  that $ c(b)g(x) \leq f(x) \leq C(b)g(x)$ for some $C=C(k)>c=c(k)>0$ just depending on $b$. When the constants $C$, $c$ can be chosen absolutely, we will just write $f(x) \sim g(x)$.
	  \end{notation}
	  
	  \begin{definition}
	  	\label{D:newweightedHolder}
	  	Assuming that $\Omega$ is a domain inside a manifold,
	  	$g$ is a Riemannian metric on the manifold, 
	  	$f,\rho:\Omega\to(0,\infty)$ are given functions, 
	  	$k\in \N$, 
	  	$\beta\in[0,1)$, 
	  	$u\in C^{k,\beta}_{loc}(\Omega)$ 
	  	or more generally $u$ is a $C^{k,\beta}_{loc}$ tensor field 
	  	(section of a vector bundle) on $\Omega$, 
	  	and that the injectivity radius in the manifold around each point $x$ in the metric $\rho^{-2}(x)g$
	  	is at least $1/10$,
	  	$\|u: C^{k,\beta} ( \Omega,\rho,g,f)\|$ is defined by
	  	$$
	  	\|u: C^{k,\beta} ( \Omega,\rho,g,f)\|:=
	  	\sup_{x\in\Omega}\frac{\,\|u:C^{k,\beta}(\Omega\cap B_x, \rho^{-2}(x)g)\|\,}{f(x) },
	  	$$
	  	where $B_x$ is a geodesic ball centered at $x$ and of radius $1/100$ in the metric $\rho^{-2}(x)g$.
	  	For simplicity any of $\beta$ or $\rho$ or $f$ may be omitted, 
	  	when $\beta=0$ or $\rho\equiv1$ or $f\equiv1$, respectively.
	  \end{definition}
	  
	  $f$ can be thought of as a ``weight'' function because $f(x)$ controls the size of $u$ in the vicinity of
	  the point $x$.
	  From the definition it follows that 
	  \begin{equation}
	  	\label{E:norm:derivative}
	  	\| \, \nabla u: C^{k-1,\beta}(\Omega,\rho,g,\rho^{-1}f)\|
	  	\le
	  	\|u: C^{k,\beta}(\Omega,\rho,g,f)\|, 
	  \end{equation}
	  and the multiplicative property 
	  \begin{equation}
	  	\label{E:norm:mult}
	  	\| \, u_1 u_2 \, : C^{k,\beta}(\Omega,g,\rho, f_1 f_2 \, )\|
	  	\lesssim_k	   
	  	\| \, u_1 \, : C^{k,\beta}(\Omega,g,\rho, f_1 \, )\|
	  	\,\,
	  	\| \, u_2 \, : C^{k,\beta}(\Omega,g,\rho, f_2 \, )\|.
	  \end{equation}

	  Cut-off functions will be used extensively,
	  and for this reason the following is adopted.
	  
	  \begin{definition}
	  	\label{DPsi} 
	  	A smooth function $\Psi:\R\to[0,1]$ is fixed with the following properties:
	  	\newline
	  	(i).
	  	$\Psi$ is non-decreasing.
	  	\newline
	  	(ii).
	  	$\Psi\equiv1$ on $[1,\infty]$ and $\Psi\equiv0$ on $(-\infty,-1]$.
	  	\newline
	  	(iii).
	  	$\Psi-\frac12$ is an odd function.
	  
	  Given now $a,b\in \R$ with $a\ne b$, the smooth function
	  $\psicut[a,b]:\R\to[0,1]$ is defined
	  by
	  \begin{equation}
	  	\label{Epsiab}
	  	\psicut[a,b]:=\Psi\circ L_{a,b},
	  \end{equation}
	  where $L_{a,b}:\R\to\R$ is the linear function defined by the requirements $L(a)=-3$ and $L(b)=3$.
	  
	  Clearly then $\psicut[a,b]$ has the following properties:
	  \newline
	  (i).
	  $\psicut[a,b]$ is weakly monotone.
	  \newline
	  (ii).
	  $\psicut[a,b]=1$ on a neighborhood of $b$ and 
	  $\psicut[a,b]=0$ on a neighborhood of $a$.
	  \newline
	  (iii).
	  $\psicut[a,b]+\psicut[b,a]=1$ on $\R$.
	  
	  Suppose now we have two sections $f_0$, $f_1$ of some vector bundle over some domain $\Omega$. (A special case is when the vector bundle is trivial and $f_0, f_1$ real-valued functions). Suppose we also have some real-valued function $d$ defined on $\Omega$. We define a new section
	  \begin{equation}
	  	\cutoff{a}{b}{d}(f_0,f_1):=f_1\psicut[a,b]\circ d +f_0\psicut[b,a]\circ d ,
	  \end{equation}
	  
	\end{definition}

	  \begin{notation}\label{not:On}
	  	We denote by $g_{\Euc}$ the standard Euclidean metric on $\R^n$ and by $g_{\Sph^{n-1}}$ the induced standard metric on the unit sphere $\Sph^{n-1}:= \{v \in\R^n:\abs{v}=1\}$. By standard notation $\mathfrak{O}(n)$ denotes the isometric group of $\Sph^{n-1}$.
	  \end{notation}

	  \begin{definition}\label{def:jacobi}
	  	For an oriented embedding hypersurface $\Sigma$ in a four-dimensional Riemannian manifold $N$ with metric $g$, $A_{\Sigma}$ is defined to be the second fundamental form of $\Sigma$ and $\abs{A}^2_{\Sigma}$ is defined to be the square of the second fundamental form. Also the second variation of three-volume or Jacobi operator $\mathcal{L}_{\Sigma}$ is defined by
	  	\begin{equation*}
	  		\mathcal{L}_\Sigma=\Delta_\Sigma+\abs{A}^2_{\Sigma}+\Ric(\nu_\Sigma,\nu_\Sigma),
	  	\end{equation*}
	  	where $\Delta_\Sigma$ is the Laplacian on $\Sigma$ defined by $g|_\Sigma$, $\nu_\Sigma$ in the normal vector of $\Sigma$ in $N$.
	  \end{definition}
	  
	  \begin{notation}\label{not:manifold}
	  	For $(N^{n+1}, g)$ a Riemannian manifold, $\Sigma^n \subset N^{n+1}$ a two-sided hypersurface equipped with a
	  	(smooth) unit normal $\nu$, and $\Omega\subset \Sigma$, we introduce the following notation where any of $N$, $g$, $\Sigma$ or $\Omega$ may be omitted when clear from context.
	  	\begin{enumerate}[(i)]
	  		\item \label{i:neighbour} For $A\subset N$ we write $\dist_A^{N,g}$ for the distance function from $A$ with respect to $g$ and we define the tubular neighborhood of $A$ with radius $\delta>0$ by  $D_A^{N,g}(\delta):=\{p\in N:\dist_A^{N,g}(p)<\delta\}$. If $A$ is finite we may just enumerate its points in both cases, for example, if $A=\{q\}$ we write $\dist^{N,g}_q(p)$. 
	  		\item We denote by $\exp^{N,g}$ the exponential map, by $\mathrm{dom}(\exp^{N,g}) \subset TN$ its maximal domain, and by $\mathrm{inj}^{N,g}$ the injectivity radius of $(N, g)$. Similarly by $\exp_p^{N,g}$, $\mathrm{dom}(\exp_p^{N,g})$ and $\mathrm{inj}_p^{N,g}$ the same at $p\in N$.
	  		\item\label{item:graph} Given a function $f : \Sigma \to \R$ satisfying $\abs{f}(p) < \mathrm{inj}^{N,g}$,
	  		$\forall p \in \Omega$, we use the notation
	  		\begin{equation*}
	  			X^{N,g}_{\Omega,f}:=\exp^{N,g}\circ(f\nu)\circ I_\Omega^{N},\quad \Graph_{\Omega}^{N,g}(f):=X^{N,g}_{\Omega,f}(\Omega),
	  		\end{equation*}
	  		where $I_\Omega^{N}:\Omega\to N$ denotes the inclusion map of $\Omega$ in $N$. 
\item\label{item:sym} For $A\subset N$ we write $\sgr A{N,g}$ for the group of isometries of $(N,g)$ which preserve $A$ as a set. 
	  	\end{enumerate}
	  \end{notation}
	  
	  As in \cite[Definition A.1]{gLD}, we also define the Fermi exponential map.
	  \begin{definition}[Fermi exponential map]\label{def:fermi}
	  	 We assume given a hypersurface $\Sigma^n$ in a Riemannian manifold $(N^{n+1}, g)$ and a unit normal $\nu_p \in T_pN$ at some $p\in\Sigma $. For $\delta > 0$ we define
	  	 \begin{equation*}
	  	 	\hat{D}^{\Sigma, N, g}_p(\delta):=\{v+\zcheck\nu_p:v\in D^{T_p\Sigma,  g|_p}_O(\delta)\subset T_p\Sigma,\zcheck\in(-\delta,\delta)\}\subset T_pN.
	  	 \end{equation*}
	  	 For small enough $\delta$, the map $\exp^{\Sigma,N,g}_p:\hat{D}^{\Sigma, N, g}_p(\delta)\to N$, defined by
	  	 \begin{equation*}
	  	 	\exp^{\Sigma,N,g}_p(v+\zcheck\nu_p):=\exp^{N,g}_q(\zcheck\nu_v),\quad \forall v+\zcheck\nu_p\in \hat{D}^{\Sigma, N, g}_p(\delta)\text{ with }v\in T_p\Sigma,
	  	 \end{equation*}
	  	 where $q := \exp^{\Sigma,g}_p(v)$ and $\nu_v \in T_qN$ is the unit normal to $\Sigma$ at $q$ pointing to the same side of $D^{\Sigma,g|_{\Sigma}}_p(\delta)$ (which is two-sided) as $\nu_p$, is a diffeomorphism onto its image which we will denote by $D^{\Sigma,N,g}_p\subset N$. We define the injectivity radius $\inj^{\Sigma,N,g}_p$ of $(\Sigma,N,g)$ at $p$ to be the supremum of such $\delta$’s. Finally when $\delta<\inj^{\Sigma,N,g}_p$ we define on $D^{\Sigma,N,g}_p(\delta)$ the following.
	  	 \begin{enumerate}[(i)]
	  	 	\item\label{item:proj}  
$\Pi_{\Sigma}: D^{\Sigma,N,g}_p(\delta)\to \Sigma\cap   D^{\Sigma, N, g}_p(\delta)$ to be the nearest point projection in $(D^{\Sigma,N,g}_p(\delta),g)$. Alternatively	$\Pi_{\Sigma}$ corresponds through $\exp^{\Sigma,N,g}_p$ to orthogonal projection to $T_p\Sigma$ in $(T_pN, g|_p)$.
	  	 	\item 
$\zzcheck: D^{\Sigma,N,g}_p(\delta)\to (-\delta, \delta)$ to be the signed distance from $\Sigma\cap   D^{\Sigma, N, g}_p(\delta)$ in $(D^{\Sigma, N, g}_p(\delta) g)$. Alternatively $(\zzcheck\circ\exp^{\Sigma,N,g}_p(v))\nu_p$ is the orthogonal projection of $v$ to $\langle\nu_p\rangle$ in $(T_pN, g|_p)$. 
	  	 	\item A foliation by the level sets $\Sigma_{\zcheck} := \zzcheck^{-1}(\zcheck) \subset D^{\Sigma, N, g}_p(\delta)$ for $\zcheck \in (-\delta, \delta)$.
	  	 	\item Tensor fields $g_{\Sigma,\zzcheck}$, $A_{\Sigma,\zzcheck}$ 
	  	 	by requesting that on each level set $\Sigma_{\zcheck}$ they are equal to the
	  	 	first and second fundamental forms
	  	 	of $\Sigma_{\zcheck}$ respectively.
	  	 \end{enumerate}
	  \end{definition}
	  
	  \subsection*{The parametrizations and coordinates on the unit three-sphere}
$\vspace{.162cm}$ 
	 
	  We consider now the unit four-sphere $\Sph^4=\Sph^4(1)\subset\R^5$ with the standard metric $g_{\Sph^4}$ as in \ref{not:On}, 
where for simplicity we will often denote $\Sph^4(1)$ by just $\Sph^4$. 
We denote by $(x_1, x_2, x_3, x_4, x_5)$ the standard coordinates of $\R^5$ and we define by
	  \begin{equation*}
	  	\Spheqtres:=\Sph^4\cap \{x_5=0\}
	  \end{equation*}
	  an equatorial three-sphere in $\Sph^4$. As \cite[(A.1)]{kapouleas:yang}, we define the surjective map $\Thetaeq:\R^2\times[-\frac{\pi}{4},\frac{\pi}{4}]\to\Spheqtres$ by 
	  \begin{align}
	  	&\Thetaeq(\xx, \yy, \zz)\\
	  	:=&(\cos(\zz+\frac{\pi}{4})\cos(\sqrt{2}\yy),\cos(\zz+\frac{\pi}{4})\sin(\sqrt{2}\yy),\sin(\zz+\frac{\pi}{4})\cos(\sqrt{2}\xx),\sin(\zz+\frac{\pi}{4})\sin(\sqrt{2}\xx),0).\nonumber
	  \end{align}
	  By calculating further we obtain as \cite[(A.5)]{kapouleas:yang}
	  \begin{equation}\label{eq:thetaeqmetric}
	  	\Thetaeq^*g_{\Sph^3}=(1+\sin(2\zz))d\xx^2+(1-\sin(2\zz))d\yy^2+d\zz^2.
	  \end{equation}
	  
	  We will also refer to 
	  \begin{equation} 
\label{eq:clifford}
	  \begin{aligned}
	  \T_c&:=\Spheqtres\cap\{x_1^2+x_2^2=x_3^2+x_4^2\}= \Thetaeq(\{\zz=c\}),\quad 
	  \text{where }c\in (-\pi/4,\pi,4)\\
	  \uC&:=\Spheqtres\cap\{x_1=x_2=0\}=\Thetaeq(\{\zz=\pi/4\}),\\
	  \uC^{\perp}&:=\Spheqtres\cap\{x_3=x_4=0\}=\Thetaeq(\{\zz=-\pi/4\}).
	 \end{aligned} 
	 \end{equation} 
We may omit the subscript $c$ when $c=0$ so $\T=\T_0$ denotes the Clifford torus. 
     
     Finally for any interval $[a,b]\subset[-\frac{\pi}{4},\frac{\pi}{4}]$ we define the thickened torus in $\Spheqtres$ by 
     \begin{equation}\label{eq:ttorus}
     	\Omega[a,b]:=\cup_{c\in [a,b]}\T_c.
     \end{equation}
	 
	 \subsection*{Symmetries of $\Thetaeq$ and symmetries of the construction}
	 \begin{notation}\label{not:hgrp}
	 	From the identification $\R^5\cong\R^4\times\R$, we can consider the embedding $\mathfrak{O}(4)\hookrightarrow \mathfrak{O}(5)$ by the standard action of $\mathfrak{O}(4)$ (recall \ref{not:On}) on the first components of $\R^4$ and by identity on the last component of $\R$.
	 	
	 	 From the identification $\R^5\cong\R^2\times \R^2\times\R$, we can consider the embedding $\mathfrak{O}(2)\times \mathfrak{O}(2)\hookrightarrow \mathfrak{O}(5)$ by the standard action of $\mathfrak{O}(2)$ (recall \ref{not:On}) on the first two components of $\R^2$ respectively and by identity on the last component of $\R$. We will then denote the image of $\mathfrak{O}(2)\times \mathfrak{O}(2)$ inside $\mathfrak{O}(5)$ by $\Hgrp$.
	 \end{notation}

	 We now define reflections $\XXXuh_c, \YYYuh_c$, rotations $\XXXh_c, \YYYh_c$ and the involution $\hat{\SSS}$ in $\R^2\times (-\frac{\pi}{4},\frac{\pi}{4})$, where $c\in\R$ by
\begin{equation} 
\label{eq:symhat}
	  \begin{aligned} 
	 	\XXXuh_c(\xx,\yy,\zz):=(2c-\xx,\yy,\zz),&\quad
	 	\YYYuh_c(\xx,\yy,\zz):=(\xx,2c-\yy,\zz),\\
	 	\XXXh_c(\xx,\yy,\zz):=(\xx+c,\yy,\zz),&\quad
	 	\YYYh_c(\xx,\yy,\zz):=(\xx,\yy+c,\zz),\\
	 	\hat{\SSS}(\xx,\yy,\zz):&=(\yy,\xx,-\zz).
	 \end{aligned}
\end{equation} 
	 
	 We then define corresponding reflections $\XXXu_c, \YYYu_c\in \Hgrp$, rotations $\XXX_c, \YYY_c\in \Hgrp$, the involution $\SSS$ and the reflection $\ZZZu$ in $\R^5$ by
\begin{equation} 
\label{eq:sym}
	 \begin{aligned} 
	 	\XXXu_c(x_1,x_2,x_3,x_4,x_5)&:=(x_1,x_2,x_3 \cos 2c + x_4 \sin 2c , x_3 \sin 2c - x_4 \cos 2c , x_5),\\
	 	\YYYu_c(x_1,x_2,x_3,x_4,x_5)&:=(x_1 \cos 2c + x_2 \sin 2c , x_1 \sin 2c - x_2 \cos 2c ,x_3, x_4, x_5),\\
	 	\XXX_c(x_1,x_2,x_3,x_4,x_5)&:=(x_1, x_2, x_3 \cos c - x_4 \sin c , x_3 \sin c + x_4 \cos c , x_5),\\
	 	\YYY_c(x_1,x_2,x_3,x_4,x_5)&:=(x_1 \cos c - x_2 \sin c , x_1 \sin c + x_2 \cos c , x_3 , x_4, x_5).\\
	 	\SSS(x_1,x_2,x_3,x_4,x_5)&:=(x_3,x_4,x_1,x_2,x_5).\\
	 	\ZZZu(x_1,x_2,x_3,x_4,x_5)&:=(x_1,x_2,x_3,x_4,-x_5).
	 \end{aligned}
\end{equation} 
	 
	 We record the symmetries of $\Thetaeq$ in the following lemma:
	 \begin{lemma}\label{lem:thetaeq}
	 	$\Thetaeq$ restricted to $\R^2\times(-\frac{\pi}{4},\frac{\pi}{4})$ is a covering map onto $\Spheqtres\setminus(\uC\cup \uC^{\perp})$. Moreover, the following hold:
	 	\begin{enumerate}[(i)]
	 		\item The group of deck transformations is generated by $\XXXh_{2\pi}$ and $\YYYh_{2\pi}$.
	 		\item\label{item:thetaeqconj} $\XXXu_c\circ\Thetaeq=\Thetaeq\circ \XXXuh_c$, $\YYYu_c\circ\Thetaeq=\Thetaeq\circ \YYYuh_c$, $\XXX_c\circ\Thetaeq=\Thetaeq\circ \XXXh_c$, $\YYY_c\circ\Thetaeq=\Thetaeq\circ \YYYh_c$ and $\SSS\circ\Thetaeq=\Thetaeq\circ \hat{\SSS}$.
	 	\end{enumerate}
	 \end{lemma}
	 
	 The symmetry group of our constructions depends on a large number $ m\in\N\setminus\{1\}$ which we assume now fixed. We define $\Lmer = \Lmer[m]\subset\Spheqtres$ to be the union of $m^2$ quarter-circles symmetrically arranged
	 \begin{equation}\label{eq:Lmer}
	 	\Lmer = \Lmer[m]:=\Thetaeq(\{(\xx,\yy,\zz):\xx=2\pi i/m,\yy=2\pi j/m,i,j\in\Z\}).
	 \end{equation}
	 \begin{definition}\label{def:grp}
Let $\GrpT:= \sgr \T{\Sph^4(1)} $, the subgroup of $\mathfrak{O}(5)$ which fixes $\T_0$ as a set. 
Let $\Grp_{\Spheqtres}[m] := \sgr {\Lmer[m]}{\Spheqt  } $ the subgroup of $\mathfrak{O}(4)$ which fixes $\Lmer[m]$ as a set, 
and $\Grp_{\Sph^4}[m] := \sgr {\Lmer[m]}{\Sph^4(1)} $ the subgroup of $\mathfrak{O}(5)$ which fixes $\Lmer[m]$ as a set.
	 \end{definition}
	 We record some properties of $\GrpT$, $\Grp_{\Spheqtres}[m]$ and $\Grp_{\Sph^4}[m]$ in the following lemma.
	 \begin{lemma}
	 	The following hold.
	 	\begin{enumerate}[(i)]
	 		\item $\GrpT$ is generated by $\Hgrp$ (recall \ref{not:hgrp}) and the involution $\SSS$. 
	 		\item $\Grp_{\Spheqtres}[m]$ is generated by the reflections $\XXXu_0,\XXXu_{\pi/m},\YYYu_0,\YYYu_{\pi/m}$ and the involution $\SSS$.
	 		\item $\Grp_{\Sph^4}[m]$ is generated by the reflections $\XXXu_0,\XXXu_{\pi/m},\YYYu_0,\YYYu_{\pi/m},\ZZZu$ and the involution $\SSS$.
	 	\end{enumerate} 	
	 \end{lemma}
	  
	 \begin{remark}
	 	$\Grp_{\Spheqtres}[m]$ is \emph{not} isomorphic to the group $D_{2m}\times D_{2m}\times \Z_2$, indeed $\SSS$ does not commute with the other generators.
	 \end{remark}
	 
	   \begin{notation}[Symmetric functions] 
	 	\label{N:G} 
	 	Given $\Omega\subset \Sph^4(1)$ 
	 	invariant under the action of $\Grp_{\Sph^4}[m]$, and a space of functions $\mathcal{X}\subset C^0(\Omega)$,
	 	we use a subscript ``sym'' to denote the subspace $\mathcal{X}_{\sym}\subset\mathcal{X}$ consisting of the functions in $\mathcal{X}$ which are invariant under the action of $\Grp$.
	 \end{notation}

	\subsection*{Acknowledgments}
$\vspace{.162cm}$ 

JZ would like to thank Daniel Ketover for his  support and for the helpful conversations about the topics of this article.
NK would like to thank the Simons Foundation for Collaboration Grant 962205 which facilitated travel related to this research.

	  \section{Linear theory on $\Spheqt$}
\label{S:linS} 

	 \subsection*{The linearized equation and rotationally invariant solutions}
$\vspace{.162cm}$ 

	By \ref{def:jacobi}, 
	\begin{equation}\label{eq:jacobi}
		\Lcal_{\Spheqtres}=\Delta_{\Spheqtres}+3.
	\end{equation}
	
	\begin{lemma}\label{lem:nokernel}
		$(\ker\Lcal_{\Spheqtres})_{\sym}$ is trivial when $m\geq 2$.
	\end{lemma}

	\begin{proof}
		The space $\ker\Lcal_{\Spheqtres}$ is generated by the four coordinate functions $x_i$, $i=1,2,3,4$, restricted on $\Spheqtres$. Thus any function $u\in(\ker\Lcal_{\Spheqtres})_{\sym}$ has the form $u=\sum_{i=1}^4a_ix_i$. By the assumption $u\circ\SSS=u$ and \eqref{eq:sym}, we have $a_1=a_3$, $a_2=a_4$. By the assumption $u\circ\YYYu_{\pi/m}=u$ and \eqref{eq:sym}, we have
		\begin{equation*}
			a_1\left(x_1\cos\frac{2\pi}{m}+x_4\sin\frac{2\pi}{m}\right)+a_2\left(x_1\cos\frac{2\pi}{m}-x_4\sin\frac{2\pi}{m}\right)\equiv a_1x_1+a_2x_2,
		\end{equation*}
		 which implies $a_1=\lambda\sin\frac{2\pi}{m}$, $a_2=\lambda(1-\cos\frac{2\pi}{m})$ where $\lambda\in \R$. However, when $m\geq 2$, by the assumption $u\circ\YYYu_{2\pi/m}=u$, we also have $a_1=\lambda'\sin\frac{4\pi}{m}$, $a_2=\lambda'(1-\cos\frac{4\pi}{m})$, where $\lambda'\in \R$. 
Thus the only function $u\in(\ker\Lcal_{\Spheqtres})_{\sym}$ is trivial.
	\end{proof}

	 It will be easier later to state some estimates if we use a scaled metric on $\Spheqtres$:
	 	\begin{definition}
	 	We define the metric $\tilde{g}$ on $\Spheqtres$ and coordinates $(\tilde{\xx},\tilde{\yy},\tilde{\zz})$ by
	 \begin{equation}\label{eq:tildeg}
	 	\tilde{g}:=m^2g_{\Sph^3},\quad (\tilde{\xx},\tilde{\yy},\tilde{\zz}):=m(\xx,\yy,\zz),
	 \end{equation}
	  and the corresponding Jacobi operator by
	  \begin{equation}\label{eq:tildeL}
		 \Lcal_{\tilde{g}}:=\Delta_{\tilde{g}}+3m^{-2}=m^{-2}\Lcal_{\Spheqtres}.
		\end{equation}
	\end{definition}

	 By a \emph{rotationally invariant function} we mean a function on a domain of $\Spheqtres$ which depends only on the coordinates $\zz$, i.e. invariant under the action of $\Hgrp=\mathfrak{O}(2)\times \mathfrak{O}(2)$. When the solution $\phi$ is rotationally invariant, by calculating from \eqref{eq:thetaeqmetric}, the linearized equation $\Lcal_{\Spheqtres}\phi=0$ amounts to the ODE:
	  \begin{equation}\label{eq:jacobirotinvar}
	  	\frac{d^2\phi}{d\zz^2}-2\tan{2\zz}\frac{d\phi}{d\zz}+3\phi=0.
	  \end{equation}

\begin{lemma} 
\label{lem:phiC}
	  	The space of solutions of the ODE \eqref{eq:jacobirotinvar} in $\zz$ on $(-\pi/4,\pi/4)$ is spanned by the two functions
	  \begin{equation} 
\label{eq:phiC}
	  \begin{aligned}
	  		\phi_\uC:= \frac{\Gamma^2(1/4)}{4\sqrt{\pi}} {}_2F_1\left(-\frac{1}{2},\frac{3}{2};1;\frac{1}{2}(\cos\zz-\sin\zz)^2\right),\\
	  		\phi_{\uC^{\perp}}:= \frac{\Gamma^2(1/4)}{4\sqrt{\pi}} {}_2F_1\left(-\frac{1}{2},\frac{3}{2};1;\frac{1}{2}(\cos\zz+\sin\zz)^2\right), 
	 \end{aligned} 
	 \end{equation} 
	  	  where ${}_2F_1(\cdot,\cdot;\cdot;\cdot)$ is the hypergeometric function (see e.g. \cite[Chapter 15]{AS}).  Moreover, the following hold.
	  	  \begin{enumerate}[(i)]
	  	  	\item 
\label{phiC:1} 
$\phi_\uC$ has singularity at $\uC=\{z=-\pi/4\}$ and is smooth at $\uC^{\perp}=\{z=\pi/4\}$, $\phi_{\uC^{\perp}}$ has singularity at $\uC^{\perp}=\{z=\pi/4\}$ and is smooth at $\uC=\{z=-\pi/4\}$.
	  	  	\item 
\label{phiC:2} 
$\phi_\uC$ is strictly increasing in $\zz$ on $(-\pi/4,\pi/4)$, $\phi_{\uC^{\perp}}$ is strictly decreasing in $\zz$ on $(-\pi/4,\pi/4)$.
	  	  	\item 
\label{phiC:3} 
	  	  	$\phi_{\uC^{\perp}}=\phi_\uC\circ\SSS$.
	  	  	\item 
\label{phiC:4} 
$\phi_\uC(0)=\phi_{\uC^{\perp}}(0)=1$, $\phi'_\uC(0)=-\phi'_{\uC^{\perp}}(0)=F_0$ where
\begin{equation} 
\label{eq:flux}
F_0:=\frac{\Gamma^4(1/4)}{8\pi^{2}}\in (2.18,2.19). 
\end{equation}
\end{enumerate}
	  \end{lemma}

	  \begin{proof}
	  	By changing of variable $w:=\frac{1}{2}(\cos\zz-\sin\zz)^2=\frac{1-\sin2\zz}{2}$, \eqref{eq:jacobirotinvar} becomes 
	  	\begin{equation*}
	  		w(1-w)\frac{d^2\phi}{dw^2}+(1-2w)\frac{d\phi}{dw}+\frac{3}{4}\phi=0,
	  	\end{equation*}
	  	which is the hypergeometric equation that ${}_2F_1\left(-\frac{1}{2},\frac{3}{2};1;w\right)$ satisfies by e.g. \cite[15.5.1]{AS}. \cite[15.1.1]{AS} also implies that ${}_2F_1$ converges absolutely in the disk $\abs{w}<1$, and diverges at $w=1$. 
(\ref{phiC:1}) then follows. 
Moreover, by (\ref{phiC:1}), $\phi_{\uC}$ and $\phi_{\uC^{\perp}}$ are linearly independent when $\zz\in(-\pi/4,\pi/4)$, thus they span the space of solutions of the ODE \eqref{eq:jacobirotinvar}.
	  	
	  	The monotonicity in (\ref{phiC:2}) can be shown by \cite[15.2.1]{AS} along with the integral representation of hypergeometric function \cite[15.3.1]{AS}. 
(\ref{phiC:3}) is clear from the definitions \eqref{eq:phiC} and \eqref{eq:symhat} along with the property \ref{lem:thetaeq}\ref{item:thetaeqconj}. 
The values of $\phi_{\uC}(0)$ and $\phi_{\uC^{\perp}}(0)$ in (\ref{phiC:4}) can be calculated by e.g. \cite[15.1.24]{AS}, 
and the values of $\phi_{\uC}'(0)$ and $\phi_{\uC^{\perp}}'(0)$ can be calculated by making use of \cite[15.2.1]{AS} along with \cite[15.1.24]{AS}. 
	  \end{proof}

	  \subsection*{Linear theory on a slab}
	  \begin{definition}
	  	Given a function $\varphi$ on some domain $\Omega\subset\Spheqtres$, we define a rotationally invariant function $\varphi_{\avg}$ on the union $\Omega'$ 
of the $\Hgrp$ orbits (recall \ref{not:hgrp}) on which $\varphi$ is integrable (whether contained in $\Omega$ or not), by requesting that on each such orbit $\T_c$,
	  	\begin{equation*}
	  		\varphi_{\avg}|_{\T_c}:=\frac{1}{\Area{ \T_c}}\int_{\T_c}\varphi.
	  	\end{equation*}
	  	We also define $\varphi_{\osc}$ on $\Omega\cap\Omega'$ by $\varphi_{\osc} := \varphi - \varphi_{\sym}$.
	  \end{definition}
	
	 \begin{lemma}[Slab case cf. {\cite[Proposition A.3]{kapouleas:finite}, \cite[Lemma 4.7]{kapouleas:zou:S2}}]\label{lem:linearCyl}
		For $\delta_s>0$ small enough and $m$ big enough in terms of absolute value, and $\gamma'\in(0,1)$,
		there exists a bounded linear map (recall \eqref{eq:clifford}, \eqref{eq:ttorus} and \eqref{eq:tildeg})

		\begin{multline*}
			(\Rcal_{\Omega},B_{\Omega}) \, : \, C^{2,\beta}_{\sym}(\T_{2/m},\tilde{g})\times C^{0,\beta}_{\sym}(\Omega[2/m,2/m+\delta_s],\tilde{g},e^{-\gamma'\tilde{\zz}})
\\ 
\to \, C^{2,\beta}_{\sym}(\Omega[2/m,2/m+\delta_s],\tilde{g},e^{-\gamma'\tilde{\zz}})\times\R,
		\end{multline*}

		  such that for $v=\Rcal_{\Omega}(f,E)$ the following hold.
		\begin{enumerate}[(i)]
			\item $\Lcal_{\tilde{g}} v=E$.
			\item $v=f-f_{\avg}+B_{\Omega}(f,E)$ on $\T_{2/m}$.
			\item $v\equiv 0$ on $\T_{2/m+\delta_s}$.
			\item $\norm{v:C^{2,\beta}(\Omega[2/m,2/m+\delta_s],\tilde{g},e^{-\gamma'\tilde{\zz}})}+\abs{B_{\Omega}(f,E)}$ 
\\
$\phantom{kk}$ \hfill $\lesssim_{\beta,\gamma'}  \norm{f:C^{2,\beta}(\T_{2/m},\tilde{g})}+\norm{E:C^{0,\beta}(\Omega[2/m,2/m+\delta_s],\tilde{g},e^{-\gamma'\tilde{\zz}})}$.
			
		\end{enumerate}
	\end{lemma}
	\begin{proof}
		By defining $(\Rcal_{\Omega},B_{\Omega})(1,0)=(0,0)$, we can assume that $f_{\avg}=0$.
		
		For this proof we define the flat metric $g_0$ on $\Omega[2/m,2/m+\delta_s]$ by (recall \eqref{eq:tildeg})
		\begin{equation*}
			g_0:=d\tilde{\xx}^2+d\tilde{\yy}^2+d\tilde{\zz}^2.
		\end{equation*}
		We then have a bounded linear map
		\begin{multline*}
			(\Rcal_0,B_0) \, : \, C^{2,\beta}_{\sym}(\T_{2/m},g_0)\times C^{0,\beta}_{\sym}(\Omega[2/m,2/m+\delta_s],g_0,e^{-\gamma'\tilde{\zz}}) 
\\ 
\to \, C^{2,\beta}_{\sym}(\Omega[2/m,2/m+\delta_s],g_0,e^{-\gamma'\tilde{\zz}})\times\R
		\end{multline*}
		satisfying (i)-(iv) with $\Delta_0$ instead of $\Lcal_{\tilde{g}}$.
		
		By \eqref{eq:thetaeqmetric}, \eqref{eq:tildeg} and \eqref{eq:tildeL} for any $\epsilon>0$, by choosing $m$ big enough and when $\delta_s$ is small enough
		\begin{equation*}
			\norm{(\Lcal_{\tilde{g}}-\Delta_0)v:C^{0,\beta}(\Omega[2/m,2/m+\delta_s],g_0,e^{-\gamma'\tilde{\zz}})}\leq \epsilon \norm{ v:C^{2,\beta}(\Omega[2/m,2/m+\delta_s],g_0,e^{-\gamma'\tilde{\zz}})}.
		\end{equation*}
		Thus the difference of the linear operator defined by
		\begin{equation*}
			(f,E)\mapsto ( \Rcal_0 (f,E)|_{\T_{2/m}}-B_0(f,E), \Lcal_{\tilde{g}}\circ \Rcal_0 (f,E))
		\end{equation*}
		with the identity map on the Banach space $C^{2,\beta}_0(\T_{2/m},g_0)\times C^{0,\beta}(\Omega[2/m,2/m+\delta_s],g_0,e^{-\gamma'\tilde{\zz}})$ is smaller then $1/2$ by choosing $\epsilon$ small enough. Therefore, there exists a bounded linear map $\Ccal$ on this space such that
		\begin{equation*}
			(f,E)= ( \Rcal_0\circ \Ccal (f,E)|_{\T_{2/m}}-B_0\circ\Ccal(f,E), \Lcal_{\tilde{g}}\circ \Rcal_0\circ \Ccal (f,E)).
		\end{equation*}
		 We can then define
		\begin{equation*}
			\Rcal_{\Omega}(f,E):=\Rcal_0\circ \Ccal(f,E),\quad B_{\Omega}(f,E):= B_0\circ\Ccal(f,E).
		\end{equation*}
		By \eqref{eq:thetaeqmetric} and \eqref{eq:tildeg}, when $\delta_s$ is small enough we have
		\begin{equation*}
			\norm{v:C^{k,\beta}(\Omega[2/m,2/m+\delta_s],\tilde{g},e^{-\gamma'\tilde{\zz}})} \sim_k	\norm{v:C^{k,\beta}(\Omega[2/m,2/m+\delta_s],g_0,e^{-\gamma'\tilde{\zz}})}. 
		\end{equation*}
		The results then follow.
	\end{proof}
	
	  \subsection*{Linear theory on the solid torus $D_\uC(\pi/4)$}

	\begin{definition}\label{def:weight}
		We define two functions $f_{0,\gamma'},f_{2,\gamma'}$ on $\Spheqtres$ by the following
		\begin{equation*}
			f_{0,\gamma'}:=\max\{e^{-\gamma'\tilde{\zz}},e^{-m\gamma'\delta_s}\}, \qquad f_{2,\gamma'}:=\max\{e^{-\gamma'\tilde{\zz}},m^{4+\beta}e^{-m\gamma'\delta_s}\}.
		\end{equation*}
	\end{definition}
	
	\begin{lemma}[Solid torus case cf. {\cite[Proposition 7.10]{kapouleas:finite}, \cite[Proposition 6.14]{kapouleas:zou:S2}}] 
\label{lem:linearSd}
		For $l=2,4$, there exists a bounded linear map
		\begin{equation*}
			(\Rcal_{\Sph},B_{\Sph}):C^{2,\beta}_{\sym}(\T_{l/m},\tilde{g})\times C^{0,\beta}_{\sym}(\Spheqtres\setminus D^{\Spheqtres}_{\T_0}(l/m),\tilde{g},f_{0,\gamma'})\to C^{2,\beta}_{\sym}(\Spheqtres\setminus D^{\Spheqtres}_{\T_0}(l/m),\tilde{g},f_{2,\gamma'})\times\R,
		\end{equation*}
			such that for $v=\Rcal_{\Sph}(f,E)$ the following hold.
		     \begin{enumerate}[(i)]
		    	\item $\Lcal_{\Spheqtres} v=E$.
		    	\item $v=0$ on $\uC$.
		    	\item $v-f$ is a constant on $\T_{l/m}$ (recall that $\partial D^{\Spheqtres}_{\T_0}(l/m)=\T_{l/m}$).
		    	\item $v=f-f_{\avg}+B_{\Sph}(f,E)$ on $\T_{l/m}$.
		        \item 
		    		$\norm{v:C^{2,\beta}(\Spheqtres\setminus D^{\Spheqtres}_{\T_0}(l/m),\tilde{g},f_{2,\gamma'})}+\abs{B_{\Sph}(f,E)} $ 
\\ $\phantom{kkk}$ \hfill 
		    		$ \lesssim  \norm{f-f_{\avg}:C^{2,\beta}(\T_{l/m},\tilde{g})}+ m^{-2} \norm{E:C^{0,\beta}(\Spheqtres\setminus D^{\Spheqtres}_{\T_0}(l/m),\tilde{g},f_{0,\gamma'})}. $ 
			\end{enumerate}
			Moreover, any $v'\in C^{2,\beta}(\Spheqtres\setminus D^{\Spheqtres}_{\T_0}(l/m),\tilde{g})$ satisfying (i)-(iii) is the same as $v$.
	\end{lemma}

	\begin{proof}
		We will just prove the case $l=2$, the other case is similar.
		
		The uniqueness simply follows from the fact that the only function which satisfies (i) and (iii) with $(f,E)=(0,0)$ is a multiple of $\phi_\uC$. And by (ii) it must be $0$. By defining $(\Rcal_{\Sph},B_{\Sph})(1,0)=(0,0)$, we can again assume that $f_{\avg}=0$. 
		
		For the proof we define cutoff functions $\psi_s,\underline{\psi}\in C^{\infty}_{\sym}(\Spheqtres)$ by
			\begin{equation*}
				\psi_s:=\cutoff{m\delta_s}{m\delta_s-1}{\tilde{\zz}}(0,1),\quad \underline{\psi}:=\cutoff{\log{m}}{\log{m}+1}{\tilde{\zz}}(0,1).
			\end{equation*}
		
		Now given $E\in C^{0,\beta}_{\sym}(\Spheqtres\setminus D^{\Spheqtres}_{\T_0}(2/m),\tilde{g},f_{0,\gamma'})$, $f\in C^{2,\beta}_{\sym}(\T_{2/m},\tilde{g})$, by using \ref{lem:linearCyl} we can define $(v'_1,\bar{\mu}_1)\in C^{2,\beta}_{\sym}(\Omega[2/m,2/m+\delta_s],\tilde{g},e^{-\tilde{\zz}})\times \mathbb{R}$ by $(v'_1,\bar{\mu}_1)=(\Rcal_{\Omega},B_{\Omega})(m^{-2}\psi_s E,f)$, and we have by the definitions in \ref{def:weight}
			\begin{multline}
\label{eq:v1prime}
			 \norm{v'_1:C^{2,\beta}(\Omega[2/m,2/m+\delta_s],\tilde{g},e^{-\tilde{\zz}})}+\abs{\bar{\mu}_1}
\\
			\lesssim  \norm{f:C^{2,\beta}(\T_{2/m},\tilde{g})}+ m^{-2} \norm{E:C^{0,\beta}(\Spheqtres\setminus D^{\Spheqtres}_{\T_0}(2/m),\tilde{g},f_{0,\gamma'})}.
		\end{multline}
		
	Then by standard theory on the sphere $\Spheqtres$, there is a function $v''_{1}\in C_{\sym}^{2,\beta}(\Spheqtres\setminus D^{\Spheqtres}_{\T_0}(\log m/m))$, such that 
		\begin{equation*}
			\Lcal_{\Spheqtres}v''_{1}=(1-\psi_s^2)E-[\mathcal{L}_{\Spheqtres},\psi_s]v'_1,\quad v''_{1}\equiv \mathrm{const} \text{ on } \T_{\log m/m},\quad v''_{1}\equiv 0 \text{ on } \uC.
		\end{equation*}
	And we have for $g=g_{\Sph^3}$ by the definition of $\psi_s$ and \eqref{eq:v1prime}
	\begin{equation}\label{eq:v1dprimeg}
\begin{aligned} 
	    &\norm{v''_1:C^{2,\beta}(\Spheqtres\setminus D^{\Spheqtres}_{\T_0}(\log m/m),g)} 
\\
		\lesssim& \,  \norm{(1-\psi_s^2)E-[\mathcal{L}_{\Spheqtres},\psi_s]v'_1:C^{0,\beta}(\Spheqtres\setminus D^{\Spheqtres}_{\T_0}(\log m/m),g)} 
\\
		\lesssim&  \, m^{2+\beta}\left(\norm{E:C^{0,\beta}(\Spheqtres\setminus D^{\Spheqtres}_{\T_0}(\delta_s-1/m),\tilde{g})}
+\norm{v'_1:C^{2,\beta}(\Omega[\delta_s-1/m,\delta_s],\tilde{g})}\right) 
\\
		\lesssim & \, m^{4+\beta}e^{-m\gamma'\delta_s}\left( \norm{f:C^{2,\beta}(\T_{2/m},\tilde{g})}
+ m^{-2} \norm{E:C^{0,\beta}(\Spheqtres\setminus D^{\Spheqtres}_{\T_0}(2/m),\tilde{g},f_{0,\gamma'})}\right).
\end{aligned} 
	\end{equation}
	Thus
	\begin{multline}\label{eq:v1dprime}
		 \norm{v''_1:C^{2,\beta}(\Spheqtres\setminus D^{\Spheqtres}_{\T_0}(\log m/m),\tilde{g},f_{2,\gamma'})}\\
		\lesssim  m^{-4-\beta} e^{m\gamma'\delta_s}\norm{v''_1:C^{2,\beta}(\Spheqtres\setminus D^{\Spheqtres}_{\T_0}(\log m/m),g)}\\
		\lesssim   \norm{f:C^{2,\beta}(\T_{2/m},\tilde{g})}+ m^{-2} \norm{E:C^{0,\beta}(\Spheqtres\setminus D^{\Spheqtres}_{\T_0}(2/m),\tilde{g},f_{0,\gamma'})}.
	\end{multline}	
		
	We then define $v_1:=\psi_sv'_1+\underline{\psi}v''_1$ and $-E_1:=\Lcal_{\Spheqtres}v_1-E$. From \eqref{eq:v1prime} and \eqref{eq:v1dprime}, we have
	\begin{equation*}
		\norm{v_1:C^{2,\beta}(\Spheqtres\setminus D_{\T_0},\tilde{g},f_{2,\gamma'})}\lesssim\norm{f:C^{2,\beta}(\T_{2/m},\tilde{g})}+ m^{-2} \norm{E:C^{0,\beta}(\Spheqtres,\tilde{g},f_{0,\gamma'})}.
	\end{equation*}
	Moreover, from all the definitions and the fact that $\underline{\psi}(1-\psi_s^2)=(1-\psi_s^2)$, we have $-E_1=[\Lcal_{\Spheqtres},\underline{\psi}]v''_1$, thus by \eqref{eq:v1dprimeg} when $m$ is big enough
	\begin{align*}
		&\norm{E_1:C^{0,\beta}(\Spheqtres\setminus D_{\T_0},\tilde{g},f_{0,\gamma'})}\\
		\lesssim &	\norm{[\Lcal_{\Spheqtres},\underline{\psi}]v''_1:C^{0,\beta}(\Omega[\log m/m,(\log m+1)/m],\tilde{g},1/m)}\\
		\lesssim &	m^3\norm{v''_1:C^{2,\beta}(\Omega[\log m/m,(\log m+1)/m],g)}\\
		\lesssim&m^{7+\beta}e^{-m\gamma'\delta_s}\left( \norm{f:C^{2,\beta}(\T_{2/m},\tilde{g})}+ m^{-2} \norm{E:C^{0,\beta}(\Spheqtres\setminus D^{\Spheqtres}_{\T_0}(2/m),\tilde{g},f_{0,\gamma'})}\right)\\
		\leq&1/2\left( \norm{f:C^{2,\beta}(\T_{2/m},\tilde{g})}+ m^{-2} \norm{E:C^{0,\beta}(\Spheqtres\setminus D^{\Spheqtres}_{\T_0}(2/m),\tilde{g},f_{0,\gamma'})}\right).
	\end{align*}
	
We can then iterate this process to get the sequence $\{(v_n,E_n,\bar{\mu}_n)\}_{n=1}^{\infty}$ and get the results by defining $(\Rcal_{\Sph},B_{\Sph})(f,E):=\sum_{i=1}^{n}(v_n,\bar{\mu}_n)$.
	\end{proof}
	
\subsection*{Main Proposition on $\Spheqt$} 
	
	 \begin{proposition}[Linear theory on $\Spheqt$ cf. {\cite[Proposition 7.1]{kapouleas:finite} and \cite[Proposition 6.1]{kapouleas:zou:S2}}] 
\label{lem:linearS}
$\phantom{k}$ \\ 
Gi\-ven 
$E \in C^{0,\beta}_{\sym}(\Spheqtres,\tilde{g})$ with support in $D^{\Spheqtres}_{\T_0}(4/m)$,  
there exists $v\in C^{2,\beta}_{\sym}(\Spheqtres,\tilde{g})$ such that $\Lcal_{\Spheqtres}v=E$. 
Moreover, the following hold.
		\begin{enumerate}[(i)]
			\item\label{item:vosc} $\norm{v_{\osc}:C^{2,\beta}(\Spheqtres,
				\tilde{g},f_{2,\gamma'})}\lesssim m^{-2} \norm{E:C^{0,\beta}(\Spheqtres,
				\tilde{g})}$.
			\item\label{item:vavg} $v_{\avg}$ is given by
			\begin{multline*}
				v_{\avg}(\zz) \, = \, -\phi_{\uC}(\zz)\left(\int_{0}^{\zz}\frac{\phi_{\uC^{\perp}}}{W}(t)E_{\avg}(t)dt \, + \, \int_{0}^{\pi/4}\frac{\phi_{\uC}}{W}(t)E_{\avg}(t)dt\right) 
\\ 
+ \, \phi_{\uC^{\perp}}(\zz)\int_{\pi/4}^{\zz}\frac{\phi_{\uC}}{W}(t)E_{\avg}(t)dt,
			\end{multline*}
			where $W=\phi_{\uC}\phi'_{\uC^{\perp}}-\phi_{\uC^{\perp}}\phi'_{\uC}$ is the Wronskian.
			\item\label{item:vavgest} There exists $\bar{\mu}\in\R$ such that (recall \ref{def:weight}) 
			\begin{equation*}
				\norm{v-\bar{\mu}\phi_{\uC}:C^{2,\beta}(\Omega[0,\pi/4],
					\tilde{g},f_{2,\gamma'})}+m^{-1}\abs{\bar{\mu}}\lesssim m^{-2} \norm{E:C^{0,\beta}(\Spheqtres,
					\tilde{g})}. 
			\end{equation*}
		\end{enumerate}
	\end{proposition}
	\begin{proof}
		The existence of $v$ simply follows by the standard theory on the sphere. (ii) simply follows by variation of parameters of the ODE \eqref{eq:jacobirotinvar} and the boundary conditions on $\uC$ and $\T_0$. Moreover, the eigenvalues of $\Lcal_{\tilde{g}}$ on the space $H^{1}_{\osc}(\Spheqtres):=\{u\in H^{1}_{\sym}(\Spheqtres):u_{\avg}=0\}$ are uniformaly bounded below, thus
		\begin{equation}\label{eq:vosc}
			\norm{v_{\osc}:C^{2,\beta}(\Spheqtres,
				\tilde{g})}\lesssim m^{-2} \norm{E_{\osc}:C^{0,\beta}(\Spheqtres,
				\tilde{g})}.
		\end{equation} 
		Moreover, by separation of variable, we have $v_{\osc}\equiv 0$ on $\uC$.
			
		Now we define $\bar{\mu}$ by requiring on $\T_{4/m}$, $v_{\avg}=\bar{\mu}	\phi_\uC$.
	 From (ii) and the support of $E$, we see that
		\begin{equation*}
			\bar{\mu}=-\int_{0}^{\pi/4}\frac{\phi_{\uC^{\perp}}}{W}(t)E_{\avg}(t)dt-\int_{0}^{\pi/4}\frac{\phi_{\uC}}{W}(t)E_{\avg}(t)dt.
		\end{equation*}
		Thus by \eqref{eq:phiC} and the support of $E$
		\begin{equation*}
			\abs{\bar{\mu}}\lesssim m^{-1} \norm{E:C^{0,\beta}(\Spheqtres,
				\tilde{g})}
		\end{equation*}
		Moreover, now the function $v-\bar{\mu}\phi_\uC$ satisfies (i)-(iii) in \ref{lem:linearSd} with $f=v|_{\T_{4/m}}$ and $l=4$. We then must have $v-\bar{\mu}\phi_\uC=\Rcal_{\Spheqtres}(v|_{\T_{4/m}},E)$ and thus by \eqref{eq:vosc}
		\begin{multline}\label{eq:vbarmu}
			\norm{v-\bar{\mu}\phi_{\uC}:C^{2,\beta}(\Spheqtres\setminus D_{\T_0}(4/m), \tilde{g},f_{2,\gamma'})}
\\ 
\lesssim m^{-2} \norm{E:C^{0,\beta}(\Spheqtres, \tilde{g})}+\norm{v_{\osc}:C^{2,\beta}(\T_{4/m},\tilde{g})}          
				\lesssim m^{-2} \norm{E:C^{0,\beta}(\Spheqtres,
					\tilde{g})}. 
		\end{multline}
		Finally by noticing that $v_{\avg}-\bar{\mu}\phi_{\uC}$ vanishes outside $D_{\T_0}(4/m)$ and satisfies the ODE
			\begin{equation*}
			\left(\frac{d^2}{d\tilde{\zz}^2}-\frac{2}{m}\tan\left(\frac{2\tilde{\zz}}{m}\right)\frac{d}{d\tilde{\zz}}+\frac{3}{m^2}\right)(v_{\avg}-\bar{\mu}\phi_{\uC})=\frac{E_{\avg}}{m^2},
		\end{equation*}
		by variation of constant we have the estimate 
		\begin{equation*}
			\norm{v_{\avg}-\bar{\mu}\phi_{\uC}:C^{2,\beta}(\Spheqtres\setminus D_{\T_0}(4/m),
				\tilde{g}}\lesssim m^{-2} \norm{E:C^{0,\beta}(\Spheqtres,
				\tilde{g})}.
		\end{equation*}
		 combing with \eqref{eq:vosc} and \eqref{eq:vbarmu}, we have the full estimate for $v-\bar{\mu}\phi_{\uC}$.
	\end{proof}
	  
	  \section{Three-catenoidal bridges}
\label{S:cat} 

	  \subsection*{Three-catenoids in $T_p\Sph^4$}
	  \begin{definition}\label{def:fermisph}
	  	Let $g$ denote $g_{\Sph^4}$. Given $p\in\Spheqtres$, we first define the function $\varsigma$ on $\T_p\Sph^4$ by requesting $\varsigma=(\exp_p^{\Spheqtres,\Sph^4})^{-1}(\zzcheck)$ (recall \ref{def:fermi}). We then extend $\varsigma$ to cylindrical coordinates $(\varrho, \vartheta, \varsigma):\T_p\Sph^4\setminus \mathrm{span}\{\nu_{\Spheqtres}(p)\}\to \R_+\times\Sph^2\times\R$. 
	  	 We also define $U:=D^{\Spheqtres,\Sph^4}_p(\inj^{\Spheqtres,\Sph^4}_p/2)\subset \Sph^4$
	  	and a Riemannian metric $\mathring{g}$ on $U$ by
	  	\begin{equation}\label{eq:mathringg}
	  		\mathring{g}=(\exp_p^{\Spheqtres,\Sph^4})_*g|_p.
	  	\end{equation}
	  	
	  \end{definition}
	 \begin{remark}[cf. {\cite[Example 2.20]{gLD}}]
	 	We have for $p=(0,0,0,1,0)\in \Spheqtres\subset \Sph^4\subset \R^5$
	 	\begin{equation}\label{eq:fermiexp}
	 		\exp_p^{\Spheqtres,\Sph^4}(\varrho, \vartheta, \varsigma)=(\vartheta\sin\varrho\cos\varsigma,\cos\varrho\cos\varsigma,\sin\varsigma),
	 	\end{equation}
	 	and thus
	 	\begin{equation}\label{eq:fermimetric}
	 		(\exp_p^{\Spheqtres,\Sph^4})^{*}g=\cos^2\varsigma(d\varrho^2+\sin^2\varrho g_{\Sph^2})+d\varsigma^2.
	 	\end{equation}
	 	
	 \end{remark}
	    
	 \begin{notation}\label{not:cyl}
	 	Let $\Cylinder := \Sph^2\times\R\subset\R^3\times\R$ be the three-cylinder embedded in $\R^4$ and $\chi$ the standard product metric on $\Cylinder$.
	 \end{notation}
	 
	 Given $p \in \Spheqtres$ and $\radius \in \R_+$, we define a three-catenoid $\K[p, \radius] \subset T_p\Sph^4 \cong \R^4$ of size $\radius$ and its parametrization $X_\K=X_\K[p, \radius] : \Cylinder \to \K[p,\radius]$ by taking (recall \ref{def:fermisph})
	 \begin{equation}
	 	(\varrho, \vartheta, \varsigma)\circ X_\K[p, \radius] (\theta,s)=(\rho(s),\theta,\sigma(s)),\label{eq:XK}
	\end{equation}
	where (cf. \cite[Section 3.1]{Pacard})
	\begin{equation}
		\rho(s):=\radius \sqrt{\cosh{2s}},\quad \sigma(s):= \int_0^s\frac{\radius^2}{\rho(t)}dt=\frac{\radius}{\sqrt{2}}F\left(\arccos\frac{1}{\sqrt{\cosh{2s}}},\frac{\sqrt{2}}{2}\right), \label{eq:rhot}
	\end{equation}	
	 where $F(\cdot,\cdot)$ is the incomplete elliptic integral of the first kind (see e.g. \cite[Chapter 17]{AS}). From now on we will use $X_\K$ to identify $\K[p, \radius]$ with $\Cylinder$; $\theta$ and $s$ can then be considered as maps on $\K[p, \radius]$ and by \eqref{eq:XK}, \eqref{eq:rhot} and \ref{not:cyl}, we can calculate
	 \begin{equation}\label{eq:gK}
	 	g_{\K}:=X_{\K}^*(g|_p)=\rho^2\chi.
	 \end{equation}
	 Alternatively $\{(\varrho,\vartheta,\varphi_{cat}(\varrho)):(\varrho,\vartheta)\in [\radius,+\infty)\times\Sph^2\}\subset T_p\Sph^4$ is the part above the waist of $\K[p, \radius]$, where the function $\varphi_{cat}(r) = \varphi_{cat}[\radius](r) : [\radius, +\infty) \to \R$ is defined by
	  \begin{equation}\label{eq:varphicat}
	  	\varphi_{cat}[\radius](r):=\int_{\radius}^r\frac{1}{\sqrt{(\frac{t}{\radius})^4-1}}dt=\frac{\radius}{\sqrt{2}}F\left(\arccos\frac{\radius}{r},\frac{\sqrt{2}}{2}\right).
	  \end{equation}
	  
	  \begin{lemma}\label{lem:varphicat}
	  	$\norm{\varphi_{cat}(r)-\radius(T_3-\frac{\radius}{r}):C^k((9\radius,+\infty),r,dr^2,r^{-5})}\lesssim_k\radius^6$, where as \cite[(A.19)]{breiner:kapouleas:high},
	  	\begin{equation}\label{eq:T3}
	  		T_3:=\int_{1}^{\infty}\frac{1}{\sqrt{t^4-1}}dt=\frac{\sqrt{2}\Gamma\left(\frac{5}{4}\right)}{\Gamma\left(\frac{3}{4}\right)}\in(1.04,1.05).
	  	\end{equation} 
	  \end{lemma}
	  \begin{proof}
	  	This estimate can be seen by expanding the integral of \eqref{eq:varphicat} at $r\to\infty$.
	  \end{proof}
	  
	  \begin{lemma}\label{lem:AK}
	  	 Let $A_{\K}$ be the second fundamental form of $\K[p,\radius]\subset T_p\Sph^4\cong\R^4$ with respect to $g|_p$, then
	  	 \begin{equation*}
	  	 	A_{\K}=-2\frac{\radius^2}{\rho}ds^2+\frac{\radius^2}{\rho}g_{\Sph^2}.
	  	 \end{equation*}
	  \end{lemma}

	  \begin{proof}
	  	This can be calculated by \eqref{eq:rhot} or \eqref{eq:varphicat}. A similar calculation has been done in \cite[Proposition 2.1]{TZ}.
	  \end{proof}
	  
  \begin{definition}[Catenoidal bridges in $\Sph^4(1)$]
\label{def:catebridge}
	Given $p\in \Spheqtres$, $\radius>0$, we define in the notation of \ref{not:cyl} 
a \emph{catenoidal bridge} $\Kcech[p,\radius] \subset \Sph^4$ and its \emph{core} $\Ku[p,\radius] \subset \Sph^4$ as follows, 
where $b$ is a large constant to be chosen independently of $\radius$ later (see the proof of \ref{prop:lineareq}), 
and $\delta'_p>0$ is a small constant which will be fixed later (see \eqref{eq:deltaprime}).  
\begin{equation*} 
\begin{gathered} 
\Kcech[p,\radius] :=X_{\Kcech}[p,\radius](\Cylinder[\radius,2\delta'_p])\subset \Sph^4, 
\qquad 
\Ku[p,\radius] :=X_{\Kcech}[p,\radius](\Cylinder[\radius,2b\radius])\subset \Sph^4, 
\\
\text{where } \qquad X_{\Kcech}[p,\radius] :=\exp_p^{\Spheqtres,\Sph^4}\circ X_{\K}[p,\radius]:\Cylinder\to \Sph^4, 
\\ 
\text{ and  } \qquad \Cylinder[\radius,r]:=\{(\theta,s)\in \Cylinder: \radius< \rho(s)< r\}.
\end{gathered} 
\end{equation*} 
	Finally using the above maps we take the coordinates $(\theta,s)$ on the cylinder as in \ref{not:cyl} to be functions on $\K[p,\radius]$ and $\Kcech[p,\radius]$, 
and thus we take $\rho(s)$ and $\sigma(s)$ in \eqref{eq:rhot} to be functions on $\Kcech[p,\radius]$. 
We also use $\chi$ to denote the metric on $\Kcech[p,\radius]$ defined by the pushfoward of $\chi$ in \ref{not:cyl} by $X_{\Kcech}[p,\radius]$.
\end{definition}

	  \subsection*{Mean curvature and volume on catenoidal bridges in $\Sph^4(1)$}
$\vspace{.162cm}$ 

	  For the rest of this section, we fix $p\in\Spheqtres$.     
	  We also denote $\Kcech=\Kcech[p,\radius]$ and $g=g_{\Sph^4}$.

  	\begin{lemma}[cf. {\cite[Example 2.20]{gLD}}]
  	Let $g_{\Kcech}$, $\nu_{\Kcech}$ and $A_{\Kcech}$ be the metric, the unit normal and the second fundamental form on $\Kcech$ induced by $g$, then
  			\begin{align}
\label{eq:gKcech}
  				g_{\Kcech}&=\rho^2(s)(1-\tanh^2(2s)\sin^2\sigma(s))ds^2+\sin^2\rho(s)\cos^2\sigma(s)g_{\Sph^2}, 
\\
\label{eq:nuKcech} 
  				\nu_{\Kcech}&=\frac{\tanh(2s)\partial_{\varsigma}-\sec^2\sigma(s)\sech^2(2s)\partial_{\varrho}}{\sqrt{1+\tan^2\sigma(s)\sech^2(2s)}}, 
  			\end{align} 

\begin{multline} 
\label{eq:AKcech}
  				\sqrt{1+\tan^2\sigma(s)\sech^2(2s)}A_{\Kcech} 
\\ 
=\left(\frac{\radius^4}{\rho^{2}(s)}\tanh(2s)(2\tan\sigma(s)+\frac{1}{2}\sinh^2(2s)\sin(2\sigma(s)))-2\frac{\radius^2}{\rho(s)}\right)ds^2  
\\
  				 +\frac{1}{2}\left(\sin(2\rho(s))\sech(2s)+\sin^2(\rho(s))\sin(2\sigma(s))\tanh(2s)\right)g_{\Sph^2}.   
\end{multline} 
  \end{lemma}
  \begin{proof}
  	In this proof we fix the notation $X=\exp_p^{\Spheqtres,\Sph^4}\circ X_\K:\Sph^2\times\R\to\Sph^4\subset \R^5$. Without loss of generality, we can assume that $p=(0,0,0,1,0)\in\Spheqtres$. Notice that $g_{\Kcech}=X^* g$, the results for $g_{\Kcech}$ and $\nu_{\Kcech}$ then follow by \eqref{eq:fermiexp}, \eqref{eq:fermimetric}, \eqref{eq:XK} and \eqref{eq:rhot}.
  	
  	For the calculation of the second fundamental form, we use the formula $(A_{\Kcech})_{\alpha\beta} = \langle \partial_{\alpha}\partial_{\beta}X, \nu_{\Kcech} \rangle$ and we have renamed the cylinder coordinates $(x_1,x_2,x_3) := (\theta_1,\theta_2,s )$ by choosing some local coordinates $\theta_1,\theta_2$ on $\Sph^2$, and $\alpha,\beta$ take the values 1, 2 and 3. The result then follows by \eqref{eq:fermiexp}, \eqref{eq:XK}, \eqref{eq:rhot} along with \eqref{eq:nuKcech}. 
  \end{proof}

  \begin{lemma}\label{lem:rhozh}
  	The following hold.
  	\begin{enumerate}[(i)]
  		\item\label{item:rho} $\norm{\rho^{\pm 1}:C^k(\Kcech,\chi,\rho^{\pm 1})}\lesssim_k 1$.
  		\item $\norm{\sigma:C^k(\Kcech,\chi, \radius)}\lesssim_k 1$.
  	\end{enumerate}
  \end{lemma}
  \begin{proof}
  	By \eqref{eq:rhot}, the estimate in (i) with $\rho^2$ instead of $\rho$ is obvious, and the estimate for $\rho$ then simply follows by induction on $k$ and general Leibniz rule for derivative. 
The estimate for $\rho^{-2}$ also follows by \eqref{eq:rhot} along with the observation that $\rho^{-2}=\radius^{-2}\sech s$ and for each $k\geq 1$, 
the $k$-th derivative of $\sech s$ is a polynomial expression in $\sech s$ and $\tanh s$, and each term of	
which contains a factor of $\sech s$. The estimate for $\rho^{-1}$ again follows by induction on $k$ and general Leibniz rule. When $k\geq 1$, (ii) then follows by (i) along with \eqref{eq:rhot} and $\rho\geq \radius$; and the case $k=0$ follows by
  	\begin{equation*}
  		\lim_{s\to \infty}\sigma(s)=\radius T_3,
  	\end{equation*}
  	which is implied by the formulae \eqref{eq:rhot} and \eqref{eq:T3} along with a change of variable.
  \end{proof}	
  	
  	\begin{lemma}\label{lem:HK}
  		Let $\mathring{g}_{\Kcech}$ and $\mathring{A}_{\Kcech}$ be the metric, the unit normal and the second fundamental form on $\Kcech$ induced by $\mathring{g}$ (recall \eqref{eq:mathringg}), the following hold on $\Kcech$.
  		\begin{enumerate}[(i)]
  			\item\label{item:fermialpha} $\norm{g_{\Kcech}-\mathring{g}_{\Kcech} :C^k(\Kcech,\chi,\rho^4)}\lesssim_k 1$.
  			\item\label{item:fermiA} $\norm{A_{\Kcech}-\mathring{A}_{\Kcech}:C^k(\Kcech,\chi,\rho^2\radius)}\lesssim_k 1$
  			\item \label{item:fermiH}	$\norm{\rho^2H:C^k(\Kcech,\chi,\rho^2\radius)}\lesssim_k 1$, where $H$ is the mean curvature on $\Kcech$ induced by $g$.
  		\end{enumerate}
  	
  	\end{lemma}
  	\begin{proof}
  		We first notice that by the definitions, $(\exp_p^{\Spheqtres,\Sph^4})^*\mathring{g}_{\Kcech}=g_{\K}$, $(\exp_p^{\Spheqtres,\Sph^4})^*\mathring{A}_{\Kcech}=A_{\K}$, where $g_{\K}$, $A_{\K}$ are the metric and second fundamental form on $\K$ induced by the Euclid metric.
  		
  		By \eqref{eq:gK} and \eqref{eq:gKcech}, we see that
  		\begin{equation*}
  			g_{\Kcech}-\mathring{g}_{\Kcech}=-\rho^2\tanh^2(2s)\sin^2\sigma ds^2+\rho^2(\rho^{-2}\sin^2\rho\cos^2\sigma-1)g_{\Sph^2}.
  		\end{equation*}
  		(i) then follows by \ref{lem:rhozh}. Similarly, by \ref{lem:AK} and \eqref{eq:AKcech}, and noticing that $\radius^2/\rho=\rho\sech(2s)$ by \eqref{eq:rhot}, we see that
  		\begin{multline*}
  			\sqrt{1+\tan^2\sigma\sech^2(2s)}A_{\Kcech}-\mathring{A}_{\Kcech}=\frac{\radius^4}{\rho^{2}}\tanh(2s)\left(2\tan\sigma+\frac{1}{2}\sinh^2(2s)\sin(2\sigma)\right)ds^2\nonumber\\
  			+\left(\rho\left(\frac{\sin(2\rho)}{2\rho}-1\right)\sech(2s)+\sin^2(\rho)\sin(2\sigma)\tanh(2s)\right)g_{\Sph^2}.
  		\end{multline*}
  		(ii) then follows by \ref{lem:rhozh}. (iii) follows by (i), (ii) and the minimality of $\Kcech$ in $\mathring{g}$.
  	\end{proof}
  	
\begin{lemma}[Volume on catenoidal bridges]
\label{lem:volcat}
	For any $\bar{r} \geq \radius >0$, the volume $|\K(\bar{r})|$ of $\K(\bar{r}) := \K[p,\radius] \cap \Pi_{T_p\Spheqtres}^{-1} D^{T_p\Spheqtres}_O(\bar{r})$ satisfies
	\begin{equation*}
		|\K(\bar{r})|=2|D^{T_p\Spheqtres}_O(\bar{r})|-T_3\frac{4\pi\radius^3}{3}+\oint_{\partial D^{T_p\Spheqtres}_O(\bar{r})} \varphi_{cat } \frac{\partial \varphi_{cat} }{\partial \eta}+O\left(\frac{\radius^8}{\bar{r}^5}\right).
	\end{equation*}
\end{lemma}	 

\begin{proof}
	By direct calculation from \eqref{eq:varphicat}:
	\begin{align*}
		|\K(\bar{r})|&=4 \pi \int_{\radius}^{\bar{r}} 2 t^2 \sqrt{1+\frac{1}{(t/\radius)^4-1}} d t 
		=\frac{8 \pi}{3} \int_{\radius}^{\bar{r}} \frac{3 t^4}{\sqrt{t^4-\radius^4}} d t \\
		&=\frac{8 \pi}{3} \int_{\radius}^{\bar{r}} \frac{t^4-\radius^4}{\sqrt{t^4-\radius^4}}+\frac{2 t^4}{\sqrt{t^4-\radius^4}} d t+\frac{8 \pi}{3} \int_{\radius}^{\bar{r}} \frac{\radius^4}{\sqrt{t^4-\radius^4}} d t \\
		&=\frac{8 \pi}{3} \int_{\radius}^{\bar{r}} \frac{d}{d t} (t \sqrt{t^4-\radius^4}) d t+\frac{8 \pi}{3} \radius^2 \varphi_{c a t}(\bar{r}) \\
		&=\quad 2 \frac{4 \pi\bar{r}^3}{3}  \sqrt{1-\frac{\radius^4}{\bar{r}^4}}+\frac{2}{3} 4 \pi\bar{r}^2 \varphi_{c a t}(\bar{r}) \frac{1}{\sqrt{(\bar{r}/\radius)^4-1}} \sqrt{1-\frac{\radius^4}{\bar{r}^4}} \\
		&=\sqrt{1-\frac{\radius^4}{\bar{r}^4}}\left(2|D^{T_p\Spheqtres}_O(\bar{r})|+\frac{2}{3} \oint_{\partial D^{T_p\Spheqtres}_O(\bar{r})} \varphi_{cat } \frac{\partial \varphi_{cat} }{\partial \eta}\right) .
	\end{align*}
	Therefore by \ref{lem:varphicat}
	\begin{align*}
		|\K(\bar{r})|	&=2|D^{T_p\Spheqtres}_O(\bar{r})|-\frac{\radius}{\bar{r}}|D^{T_p\Spheqtres}_O(\radius)|+\oint_{\partial D^{T_p\Spheqtres}_O(\bar{r})} \varphi_{cat } \frac{\partial \varphi_{cat} }{\partial \eta} 
		-\frac{ 4 \pi \radius^2}{3} \varphi_{c a t}(\bar{r})+O\left(\frac{\radius^8}{\bar{r}^5}\right) \\
		&=2|D^{T_p\Spheqtres}_O(\bar{r})|-T_3|D^{T_p\Spheqtres}_O(\radius)|+\oint_{\partial D^{T_p\Spheqtres}_O(\bar{r})} \varphi_{cat } \frac{\partial \varphi_{cat} }{\partial \eta}+O\left(\frac{\radius^8}{\bar{r}^5}\right).
	\end{align*}
\end{proof}

\begin{lemma}[Volume of truncated ${\Kcech[p,\radius] \subset \Sph^4}$]\label{lem:voltrun}
	Let $\Kcech[p,\radius]$ be as in \ref{def:catebridge}. The volume $\abs{\Kcech(\bar{r})}$ of $\Kcech(\bar{r}):=\Kcech[p,\radius]\cap  \Pi_{\Spheqtres}^{-1} (D^{\Spheqtres}_p(\bar{r}))$ satisfies the following, where $\varphi_{cat}$ is as in \eqref{eq:varphicat},
	\begin{equation*}
		|\Kcech(\bar{r})|=2|D^{\Spheqtres}_p(\bar{r})|-T_3\frac{4\pi\radius^3}{3}+\oint_{\partial D^{\Spheqtres}_p(\bar{r})} \varphi_{cat }\circ\dist_p^{\Spheqtres} \frac{\partial \varphi_{cat}\circ\dist_p^{\Spheqtres} }{\partial \eta}+O\left(\frac{\radius^8}{\bar{r}^5}+\bar{r}^5\right).
	\end{equation*}
\end{lemma}
\begin{proof}
	From \eqref{eq:gKcech} and \ref{lem:rhozh}, $\sqrt{\det g_{\Kcech}}=\sqrt{\det \mathring{g}_{\Kcech}}(1+O(\bar{r}^2))$. Thus by \ref{lem:volcat} $\abs{\Kcech(\bar{r})}=\abs{\K(\bar{r})}(1+O(\bar{r}^2))=\abs{\K(\bar{r})}+O(\bar{r}^5)$.
	
	By \eqref{eq:fermimetric},  we have that $|D^{\Spheqtres}_p(\bar{r})|=|D^{T_p\Spheqtres}_O(\bar{r})|+O(\bar{r}^5)$, and that the area elements	$d\sigma_g$ and $d\sigma_{\mathring{g}}$ on $\partial D^{\Spheqtres}_p(\bar{r})$ with respect to $g$ and $\mathring{g}$ satisfy $d\sigma_g=d\sigma_{\mathring{g}}(1+O(\bar{r}^2))$. The conclusion follows by combining \ref{lem:volcat} with the preceding estimates.
\end{proof}
	  
	  \section{LD solutions}
\label{section:LD} 

	  \subsection*{Green's function and LD solutions}
$\vspace{.162cm}$ 

	    We discuss now the $\mathfrak{O}(3)$-invariant Green’s function for $\Lcal_{\Spheqtres}$. 
	  \begin{lemma}\label{lem:green}
	  	There is a function $G \in C^{\infty}(0,\pi)$ uniquely characterized by (i) and (ii) and moreover satisfying (iii-iv) below. We denote by $r$ the standard coordinate of $\R^+$:
	  	\begin{enumerate}[(i)]
	  		\item\label{item:greendprop} For small $r$ we have $G(r) = -(1 + O(r^2)) \frac{1}{r}$.
	  		\item\label{item:greenp} For each $p\in\Spheqtres$ we have $\Lcal_{\Spheqtres} G_p= 0$ where $G_p := G\circ \dist^{\Spheqtres}_p\in C^{\infty}(\Spheqtres\setminus\{p,-p\})$.
	  		\item $G(r)=-\frac{\cos{2r}}{\sin{r}}$.
	  		\item\label{item:greenest} $\norm{G+1/r:C^k((0,1),r,dr^2,r)}\lesssim 1$.
	  	\end{enumerate}
	  \end{lemma}
	  \begin{proof}
	  	The metric $g_{\Spheqtres}$ can be written as 
	  	\begin{equation*}
	  		g_{\Spheqtres}=dr^2+r^2g_{\Sph^2},
	  	\end{equation*}
	  	 where $r:=\dist_p^{\Spheqtres}$ is the distance to a given point $p$. For a solution $G_p$ depending only on $r$, the equation $\Lcal_{\Spheqtres} G_p= 0$ can be written as
	  	 \begin{equation*}
	  	 	G_p''(r)+2\cot(r)G_p'(r)+3G_p(r)=0.
	  	 \end{equation*}
	  	 The equation has two linearly independent solutions $\cos r$ and $\frac{\cos{2r}}{\sin{r}}$. (i-iii) are then clear. (iv) follows by (iii).
	  \end{proof}
	  
	   Because of the symmetries imposed on our constructions we concentrate now on LD solutions which are invariant under the action of $\Grp_{\Spheqtres}[m]$ and moreover the singular set $L$ is contained in the Clifford Torus $\T_0$. 
	  
	  \begin{definition}\label{def:L}
	  	We define a finite set $L\subset \Spheqtres$ consisting of $m^2$ points by (recall \ref{eq:Lmer} and \eqref{eq:clifford}):
	  	 \begin{equation}
	  		L=L[m]=\Grp_{\Spheqtres}[m]p_0=\T_0\cap\Lmer, 
	  	\end{equation}
	  	where $p_0=\Thetaeq(0,0,0)=(\sqrt{2}/2,0,\sqrt{2}/2,0,0)$.
	  \end{definition}
	  
	  \begin{definition}[LD solutions in the case of Clifford torus]\label{def:LD}
	  	For $L$ as in \ref{def:L}, we call $\varphi$ a linearized doubling (LD) solution on $\Spheqtres$ 
	  	when there exists a number $\tau\in \R$
	  	satisfying the following.
	  	\renewcommand{\theenumi}{\roman{enumi}}
	  	\begin{enumerate}
	  		\item $\varphi\in C^{\infty}(\Spheqtres\setminus L)$ and $\Lcal_{\Spheqtres}\varphi=0$ on $\Spheqtres\setminus L$. 
	  		\item $\forall p\in L$, the function $\varphi+\tau/\dist_{\Spheqtres}^g$ is bounded on some neighborhood of $p$ in $\Spheqtres$.
	  	\end{enumerate}	
	  \end{definition}

	   \begin{definition}[The constants $\delta$]\label{def:delta}
	   	We define a constant $\delta>0$ by
	   	\begin{equation}\label{eq:delta}
	   		\delta=\delta[m]:=1/(9m).
	   	\end{equation}
	   	Moreover, given $L$ as in \ref{def:L}, we assume that $m$ is big enough so that the following are satisfied.
	  	\renewcommand{\theenumi}{\roman{enumi}}
	  	\begin{enumerate}
	  		\item $\forall p,p'\in L$ with $p\neq p'$ we have $D_p^{\Spheqtres}(9\delta)\cap D_{p'}^{\Spheqtres}(9\delta)=\emptyset$.
	  		\item $\forall p\in L$, $\delta<\mathrm{inj}_p^{\Spheqtres,\Sph^4,g}$, where $\mathrm{inj}_p^{\Spheqtres,\Sph^4,g}$ is the injective radius as in \ref{def:fermi}.
	  	\end{enumerate}	
	  \end{definition}

	  \begin{lemma}\label{lem:LD}
	 Given a 
	 number $\tau\in\R$, there is a unique $\Grp_{\Spheqtres}[m]$-symmetric LD solution $\varphi = \varphi[\tau;m]$ satisfying the conditions in \ref{def:LD}. 
Moreover the following hold. 
	  	\begin{enumerate}[(i)]
	  		\item  $\varphi$ depends linearly on $\tau$.
	  		\item $\varphi_{\avg} \in C^0(\Spheqtres)\cap C^{\infty}(\Spheqtres\setminus \T_0)$ and on $\Spheqtres\setminus \T_0$ it
	  		satisfies the ODE $\Lcal_{\Spheqtres}\varphi_{\avg}=0$.
	  	\end{enumerate}
	  	
	  \end{lemma}
	  \begin{proof}
	  	We define $\varphi_1 \in C^{\infty}_{\sym}(\Spheqtres\setminus L)$ by requesting that it is supported on $D^{\Spheqtres}_L(2\delta)$ and $\varphi_1=\cutoff{\delta}{2\delta}{\dist_p^{\Spheqtres}}(\tau G_p,0)$ (recall \ref{lem:green}\ref{item:greenp}) on $D^{\Spheqtres}_p(2\delta)$ for each $p \in L$. Note that the function $\Lcal_{\Spheqtres}\varphi_1\in C^{\infty}_{\sym}(\Spheqtres)$ (by assigning 0 values on $L$) and it is supported on $\sqcup_{p\in L}D^{\Spheqtres}_p(2\delta)\setminus D^{\Spheqtres}_p(\delta)$. Because of the symmetries, by \ref{lem:nokernel}, the operator $\Lcal_{\Spheqtres}$ has no kernel in the space $C^{\infty}_{\sym}(\Spheqtres)$, there is a function $\varphi_2\in C^{\infty}_{\sym}(\Spheqtres)$ such that $\Lcal_{\Spheqtres}\varphi_2=-\Lcal_{\Spheqtres}\varphi_1$. We can define then $\varphi := \varphi_1 + \varphi_2$. Uniqueness and linearity follow then immediately. To prove (ii) we need to check that $\varphi$ is integrable on $\T_0$ and that $\varphi_{avg}$ is continuous there also. But these follow easily by the singularity behavior of $G_p$ in \ref{lem:green}.
	  \end{proof}
	  
	  \begin{definition}\label{def:Phi}
	  	By making use of \ref{lem:LD}, we define the LD solution $\Phi = \Phi[m] := \varphi[1;m] \in C^{\infty}_{sym}(\Spheqtres\setminus L)$.
	  \end{definition}

	  	\subsection*{The rotationally invariant part $\phi := \Phi_{avg}$}
	  	
	  	\begin{definition}\label{def:phihat}
	  		We define the rotationally invariant function (RLD solution) $\hat{\phi}\in C^{0}(\Spheqtres)\cap C^{\infty}(\Spheqtres\setminus \T_0)$ by (recall \ref{lem:phiC})
	  		\begin{align*}
	  			\hat{\phi}:=
	  			\begin{cases} 
	  			\phi_{\uC}, &\text{ when }\zz\geq 0;\\
	  			\phi_{\uC^{\perp}}, &\text{ when }\zz\leq 0.
	  			\end{cases}
	  		\end{align*}

	  	\end{definition}

	  	\begin{lemma}
	  		$\phi:=\Phi_{avg}[m]$ is given by (recall \eqref{eq:flux})
	  		\begin{equation}\label{eq:phi}
	  			\phi=\phi_1
	  			\hat{\phi},\text{ where }\phi_1:=\frac{m^2}{\pi F_0}.
	  		\end{equation}
	  	\end{lemma}
	  	\begin{proof}
	  		Because of the symmetries it is clear that $\phi=A\hat{\phi}$ for some constant $A$. For $0 < \epsilon_1 \ll \epsilon_2$, we consider now the domain $\Omega_{\epsilon_1, \epsilon_2}:=D^{\Spheqtres}_{\T_0}(\epsilon_2)\setminus D^{\Spheqtres}_L(\epsilon_1) $. By integrating $\Lcal_{\Spheqtres}\Phi=0$ on $\Omega_{\epsilon_1, \epsilon_2}$ and integrating by parts we obtain
	  		\begin{equation*}
	  			\oint_{\partial\Omega_{\epsilon_1, \epsilon_2}}\frac{\partial\Phi }{\partial\eta}+3\int_{\Omega_{\epsilon_1, \epsilon_2}}\Phi=0.
	  		\end{equation*}
	  		By taking the limit as $\epsilon_1\to 0$ first and then as $\epsilon_2\to 0$ we obtain by using the $1/r$ behavior of $\Phi$ near $L$, \ref{def:phihat} 
and \ref{lem:phiC}(\ref{phiC:4}) that
	  		\begin{equation*}
	  			2AF_0\mathrm{Area}(\T_0)=m^2\mathrm{Area}(\Sph^2).
	  		\end{equation*}
	  		The result then follows by $\mathrm{Area}(\T_0)=2\pi^2$ and $\mathrm{Area}(\Sph^2)=4\pi$.
	  	\end{proof}
	  	
	  	\begin{definition} 
\label{def:phiju}
	  		For $a,b\in\mathbb{R}$, we define the rotationally invariant solution $\underline{\phi}[a]$ and (with singularity) $\underline{j}[b]$ by requesting the following initial data
	  		\begin{equation*}
	  			\underline{\phi}(0)=a,\quad \partial \underline{\phi}(0)=0;\quad \underline{j}(0)=0,\quad \partial_+ \underline{j}(0)=-\partial_- \underline{j}(0)=mb,
	  		\end{equation*}
	  		and the ODEs $\Lcal_{\Spheqtres}\underline{\phi}=0$ on $\{\zz\in(-\pi/4,\pi/4)\}$ and $\Lcal_{\Spheqtres}\underline{j}=0$ on $\{\zz\in(-\pi/4,0)\cup(0,\pi/4)\}$. In terms of $\phi_{\uC}$ and $\phi_{\uC^{\perp}}$ (recall \ref{lem:phiC}), we can also define $\underline{\phi}[a]$ and $\underline{j}[b]$ by
	  		\begin{align*}
	  			\underline{\phi}[a]:=\frac{a}{2}(\phi_\uC+\phi_{\uC^{\perp}}),\quad 	\underline{j}[b]:=
	  			\begin{cases} 
	  				\frac{mb}{2F_0}(\phi_\uC-\phi_{\uC^{\perp}}), &\text{ when }\zz\geq 0;\\
	  				\frac{mb}{2F_0}(\phi_{\uC^{\perp}}-\phi_\uC), &\text{ when }\zz\leq 0.
	  			\end{cases}
	  		\end{align*}
	  	\end{definition}
	  	
	  	\begin{corollary}
\label{cor:phiju}
	  		On $\{\zz\in(-\pi/4,0)\cup(0,\pi/4)\}$ we have that
	  		\begin{equation}\label{eq:Phiavg}
	  			\phi=\underline{\phi}[\phi_1]+\underline{j}[m/\pi].
	  		\end{equation}
	  	\end{corollary}
	  	\begin{proof}
	  		It follows by \ref{def:phihat}, \eqref{eq:phi}, \ref{def:phiju} and \ref{lem:phiC}(\ref{phiC:4}).  
	  	\end{proof}
	  	
	  	\begin{lemma}\label{lem:phiju}
	  		The following hold
	  		\begin{enumerate}[(i)]
	  			\item\label{item:phiu} $\norm{\underline{\phi}[1]-1:C^k(D_{\T_0}^{\Spheqtres}(3/m)\setminus\T_0,\tilde{g})}\lesssim_k 1/m^2$.
	  			\item \label{item:ju} $\norm{\underline{j}[1]-m\abs{\zz}:C^k(D_{\T_0}^{\Spheqtres}(3/m)\setminus\T_0,\tilde{g})}\lesssim_k 1/m^2$.
	  		\end{enumerate}
	  	\end{lemma}
	  	\begin{proof}
	  		By changing variable $\tilde{\zz}:=m\zz$ and $\underline{\phi}_1:=\underline{\phi}[1]-1$, $\underline{j}_1:=\underline{j}[1]-m\zz$, the ODE \eqref{eq:jacobirotinvar} that $\underline{\phi}[1]$ and $\underline{j}[1]$ satisfy becomes
	  		\begin{align*}
	  			&\frac{d^2\underline{\phi}_1}{d\tilde{\zz}^2}-\frac{2}{m}\tan\left(\frac{2\tilde{\zz}}{m}\right)\frac{d\underline{\phi}_1}{d\tilde{\zz}}+\frac{3\underline{\phi}_1}{m^2}=-\frac{3}{m^2}\\
	  			&\frac{d^2\underline{j}_1}{d\tilde{\zz}^2}-\frac{2}{m}\tan\left(\frac{2\tilde{\zz}}{m}\right)\frac{d\underline{j}_1}{d\tilde{\zz}}+\frac{3\underline{j}_1}{m^2}=\frac{2}{m}\tan\left(\frac{2\tilde{\zz}}{m}\right)-\frac{3\tilde{\zz}}{m^2},
	  		\end{align*}
	  		with zero initial condition at $\tilde{\zz}=0$. The results then follows by applying variation of parameters as the homogeneous equation has fundamental solutions $	\frac{1}{2}(\phi_\uC(\tilde{\zz}/m)+\phi_{\uC^{\perp}}(\tilde{\zz}/m))=\underline{\phi}[1](\tilde{\zz}/m)$ and $\frac{m}{2F_0}(\phi_\uC(\tilde{\zz}/m)-\phi_{\uC^{\perp}}(\tilde{\zz}/m))$.
	  	\end{proof}
	  	
	  	\subsection*{Estimates on $\Phi=\Phi[m]$}
	  		\begin{definition}\label{def:sol}
	  		Given $\delta$ as in \eqref{eq:delta}, we define the $\delta'$ by 
	  		\begin{align}\label{eq:deltaprime}
	  			\delta':=\delta^{1+\alpha},
	  		\end{align} 
	  		where 
	  	we choose $\alpha$ to satisfy 
	  	\begin{equation}\label{eq:alpha}
	  		\alpha\in (0,1/4). 
	  	\end{equation}
Notice that we have $9\delta'<\delta/10$ by assuming $m$ big enough.
	  	\end{definition}

	  	\begin{definition}\label{def:decompose}
	  		We define $\hat{G}\in C^{\infty}_{\sym}(\Spheqtres\setminus L)$ by
	  		\begin{equation}\label{eq:Ghat}
	  			\hat{G}:=\cutoff{2\delta}{3\delta}{\dist^{\Spheqtres}_{p}}(G_p,0)
	  		\end{equation} 
	  		on $D_L^{\Spheqtres}(3\delta)$ (recall \ref{lem:green}(ii)) and $0$ otherwise. We define the rotationally invariant function $\hat{\Phi}\in C^{\infty}(\Spheqtres)$ by
	  		\begin{equation}\label{eq:Phihat}
	  			\hat{\Phi}:=\phi-\cutoff{2/m}{3/m}{\dist^{\Spheqtres}_{\T_0}}(\underline{j}[m/\pi],0).
	  		\end{equation}
	  		We then define the functions $\Phi'\in C^{\infty}_{\sym}(\Spheqtres)$ and $E'\in C^{\infty}_{\sym}(\Spheqtres)$  by
	  		\begin{equation}\label{eq:Phiprime}
	  			\Phi':=\Phi-\hat{G}-\hat{\Phi},\quad E':=\Lcal_{\Spheqtres}\Phi'.
	  		\end{equation}
	  	\end{definition}
	  	
	  	\begin{lemma}\label{lem:GEPhi}
	  		The following hold (recall \ref{eq:tildeg}).
	  		\begin{enumerate}[(i)]

	  			\item\label{item:Ghatprime} $\norm{\hat{G}:C^k(\Spheqtres\setminus D_L^{\Spheqtres}(\delta'),\rr,g)}\lesssim  m^{1+\alpha}$, where $\rr:=\min\{\dist_L^{\Spheqtres},\delta\}$.
	  			\item\label{item:Ghat} $\norm{\hat{G}:C^k(\Spheqtres\setminus D_L^{\Spheqtres}(\delta),\tilde{g})}\lesssim_k m$.
	  			\item\label{item:Eprime} $\norm{m^{-2}E':C^k(\Spheqtres ,\tilde{g})}\lesssim_k m$.
	  			\item\label{item:Phiosc} $\norm{\Phi_{\osc}':C^k(\Spheqtres,\tilde{g})}\lesssim_k m$.
	  		\end{enumerate}
	  	\end{lemma}
	  	\begin{proof}
	  		(i-ii) follow by \ref{lem:green}\ref{item:greenest}, \eqref{eq:delta}, \eqref{eq:deltaprime} and the definition \eqref{eq:Ghat}. By \eqref{eq:Phiprime} and \eqref{eq:Phihat}, on $\Spheqtres\setminus  D_L^{\Spheqtres}(\delta)$,
	  		\begin{equation*}
	  			m^{-2}E'=-\Lcal_{\tilde{g}}\hat{G}+\Lcal_{\tilde{g}}\cutoff{2/m}{3/m}{\dist^{\Spheqtres}_{\T_0}}(\underline{j}[m/\pi],0).
	  		\end{equation*}
	  		The first term vanishes on $D_L^{\Spheqtres}(\delta)$, and on $\Spheqtres\setminus  D_L^{\Spheqtres}(\delta)$ it is controlled by (i). The second term vanishes on $\Spheqtres\setminus D_{\T_0}^{\Spheqtres}(3/m)$, and on $D_{\T_0}^{\Spheqtres}(3/m)$, it is controlled by \ref{lem:phiju}\ref{item:ju}. (iii) then follows. (iv) now follows by \ref{lem:linearS}\ref{item:vosc} along with (iii) by noticing that $E'$ is supported in $D_{\T_0}^{\Spheqtres}(3/m)$.
	  	\end{proof}
	  	
	  	\begin{lemma}\label{lem:Phiprime}
	  		$\norm{\Phi':C^k(\Spheqtres ,\tilde{g})}\lesssim_k m$.
	  	\end{lemma}
	  	\begin{proof}
	  		By \ref{lem:GEPhi}\ref{item:Phiosc}, we just need to show the bound for $\Phi'_{avg}$. By \eqref{eq:Phiprime} and \eqref{eq:Phihat},
	  		\begin{equation*}
	  			\Phi'_{avg}=-\hat{G}_{\avg}+\cutoff{2/m}{3/m}{\dist^{\Spheqtres}_{\T_0}}(\underline{j}[m/\pi],0),
	  		\end{equation*}
	  		which vanishes on $\Spheqtres\setminus D_{\T_0}^{\Spheqtres}(3/m)$ completely. While on $D_{\T_0}^{\Spheqtres}(3/m)$, $\Lcal_{\tilde{g}}\Phi'_{avg}=m^{-2}E'_{\avg}$ provides the ODE
	  		\begin{equation*}
	  			\frac{d^2\Phi'_{avg}}{d\tilde{\zz}^2}-\frac{2}{m}\tan\left(\frac{2\tilde{\zz}}{m}\right)\frac{d\Phi'_{avg}}{d\tilde{\zz}}+\frac{3\Phi'_{avg}}{m^2}=\frac{E'_{\avg}}{m^2}.
	  		\end{equation*}
	  		The bound required then again comes from the variation of parameters as in \ref{lem:phiju}\ref{item:ju} by \ref{lem:GEPhi}\ref{item:Eprime}.
	  	\end{proof}
	  	
	  	\subsection*{Mismatch and obstruction spaces $\skernelv$, $\skernel$, $\skernelev$, $\skernele$}
	  	
	  	\begin{definition}[Spaces of affine functions]\label{def:affine}
	  		Given $p\in\Spheqtres$ let $\Vcal[p]\subset C^{\infty}(T_p\Spheqtres)$ be the space of \emph{affine functions on} $T_p\Spheqtres$. Given a function $v$ which is defined on a neighborhood of $p$ in $\Spheqtres$ and is differentiable at $p$ we define $\Ecalu_pv:=v(p)+d_pv\in\Vcal[p]$. $\forall \kappau\in\Vcal[p]$ let $\kappau=\kappa^{\perp}+\kappa$ be the unique decomposition with $\kappa^{\perp}\in\R$ and $\kappa\in T^*_p\Spheqtres$ and let $\abs{\kappau}:=\abs{\kappa^{\perp}}+\abs{\kappa}$. we define also $\Vcal[L]:=\oplus_{p\in L}\Vcal[p]$ for any finite set $L\in\Spheqtres$.
	  	\end{definition}
	  	
	  	\begin{lemma}\label{lem:varphi}
	  		For $L$ as in \ref{def:L} and the $\Grp_{\Spheqtres}[m]$-symmetry LD solution $\varphi=\varphi[\tau;m]$, $\forall p\in L$ there exists $\hat{\varphi}_p\in C^{\infty}(D_p^{\Spheqtres}(2\delta))$ such that the following hold.
	  		\renewcommand{\theenumi}{\roman{enumi}}
	  		\begin{enumerate}
	  			\item $\varphi=\varphihat_p+\tau G_p$ on $D_p^{\Spheqtres}(2\delta)\setminus\{p\}$ (recall \ref{lem:green}\ref{item:greenp}).
	  			\item $\Ecalu_p\varphihat_p:T_p\Spheqtres\to \Vcal[p]$ is independent of the choice of $\delta$ and depends only on $\varphi$.
	  			\item $\varphi\circ\exp_p^{\Spheqtres}(v)=-\tau/\abs{v}+\Ecalu_p\varphihat_p(v)+O(\abs{v}^2)$ for small $v\in T_p\Spheqtres$.
	  			\item\label{item:varphisym} $\Ecalu_pC^{\infty_{\sym}}(\Spheqtres)\subset$ is a one-dimensional subspace of $\Vcal[p]$ and $\forall p,q\in L$, $\Ecalu_p\varphihat_p=\Ecalu_q\varphihat_q$.
	  		\end{enumerate}	
	  	\end{lemma}
	  	\begin{proof}
	  		(i) follows from \ref{def:LD} 
	  		and (i) serves then as the definition of $\varphihat_p$. (ii) follows then from (i) and (iii) from a Taylor expansion of $\varphihat_p$ combined with \ref{lem:green}\ref{item:greenest}.From the symmetry imposed on $\varphi$ and $G_p$ in \ref{lem:green}, $\forall p\in L$, $d\varphihat_p=0$ and $\forall p,q\in L$, $\varphihat_p(p)=\varphihat_q(q)$. (iv) then follows.
	  	\end{proof}
	  	
	  	\begin{definition}[Mismatch of LD solutions]\label{def:mismatch}
	  		For $L$ as in \ref{def:L} and the $\Grp_{\Spheqtres}[m]$-symmetry LD solution $\varphi=\varphi[\tau;m]$ as in \ref{lem:LD} with $\tau>0$, we define the \emph{mismatch of $\varphi$}, $\Mcal_L\varphi\in\Vcal[L]$, by $\Mcal_L\varphi:=\oplus_{p\in L}\Mcal_p\varphi$, where $\Mcal_p\varphi\in\Vcal[p]$ is defined by (recall \eqref{eq:T3})
	  		\begin{equation*}
	  			\Mcal_p\varphi:=\Ecalu_p\varphihat_p-\sqrt{\tau}T_3,
	  		\end{equation*} 
	  		or by \ref{lem:varphi}(iii) equivalently requesting that for small $v\in T_p\Spheqtres$
	  		\begin{equation*}
	  			\varphi\circ\exp_p^{\Spheqtres}(v)=\sqrt{\tau}T_3-\tau/\abs{v}+(\Mcal_p\varphi)(v)+O(\abs{v}^2).
	  		\end{equation*}
	  		
	  		By \ref{lem:varphi}\ref{item:varphisym}, $\Mcal_L\varphi$ lies in a one-dimensional vector subspace of $\Vcal[L]$, which we now denote by $\Vcal_{\sym}[L]$.
	  	\end{definition}
	  	
	  	\begin{definition}\label{def:VW}
	  		We define the functions $V,W\in C^{\infty}_{\sym}(\Spheqtres)$ by (recall \ref{def:phiju})
	  		\begin{equation}\label{eq:VW}
	  			V:=\cutoff{\delta}{2\delta}{\dist_L^{\Spheqtres}}(\underline{\phi}[1],0),\quad W:=\Lcal_{\Spheqtres}V
	  		\end{equation}
	  		on $D_L^{\Spheqtres}(2\delta)$ and $0$ otherwise; and we define the functions $\bar{V},\bar{W}\in C^{\infty}_{\sym}(\Spheqtres)$ by (recall \eqref{def:phihat})
	  		\begin{equation}\label{eq:VWbar}
	  			\bar{V}:=\cutoff{3/m}{2/m}{\dist_{\T_0}^{\Spheqtres}}(\hat{\phi},\underline{\phi}[1]),\quad \bar{W}:=\Lcal_{\Spheqtres}	\bar{V}
	  		\end{equation} 
	  		on $\Spheqtres \setminus D^{\Spheqtres}_{\T_0}(2/m)$ and $0$ otherwise.
	  		We then define spaces $\skernelv[L]$, $\skernelev[L]$, $ \skernel[L]$, $\skernele[L]\subset C^{\infty}_{\sym}(\Spheqtres)$ by
	  		\begin{equation*}
	  			\skernelv[L]:=\mathrm{span}\{V\},\quad \skernel[L]:=\mathrm{span}\{W\}; 
\quad\skernelev[L] :=\mathrm{span}\{\bar{V}\}, \quad \skernele[L]:=\mathrm{span}\{\bar{W}\}.
	  		\end{equation*}
	  	\end{definition}
	  	\begin{lemma}[Obstruction spaces]\label{lem:obstr}
	  		The following hold.
	  		
	  		\begin{enumerate}[(i)]
	  			\item\label{item:vsupp} The functions in $\skernelv$ are supported on $D^{\Spheqtres}_L(4\delta)$.
	  			\item\label{item:wsupp} The functions in $\skernel$ are supported on $\sqcup_{p\in L}D^{\Spheqtres}_p(4\delta)\setminus D^{\Spheqtres}_p(\delta/4)$, the functions in $\skernele$ are supported on $D^{\Spheqtres}_{\T_0}(3/m)\setminus D^{\Spheqtres}_{\T_0}(2/m)$.
	  			\item\label{item:V} $\norm{V:C^k(\Spheqtres ,\tilde{g})}\lesssim_k 1$, $\norm{\bar{V}:C^k(\Spheqtres ,\tilde{g})}\lesssim_k 1$.
	  			\item $\Ecal_L:\skernelv[L]\to\Vcal_{\sym}[L]$ is a linear isomorphism, where the map $\Ecal_L:\skernelv[L]\to\Vcal_{\sym}[L]$ is defined by $\Ecal_L(v):=\sum_{p\in L}\Ecalu_p v$.
	  			\item\label{item:Vbar} $\bar{V}=\hat{\phi} $ on $\Spheqtres\setminus D^{\Spheqtres}_{\T_0}(3/m)$ and  $\norm{\bar{V}-\phi_\uC:C^{2,\beta}(\Omega[0,\pi/4],\tilde{g})}\lesssim 1/m$.
	  			\item $\norm{\Ecal_L^{-1}}\lesssim1$, where $\norm{\Ecal_L^{-1}}$ is the operator norm of $\Ecal_L^{-1}:\Vcal_{\sym}[L]\to\skernelv[L]$ with respect to the $C^{2,\beta}(\Spheqtres,\delta,g_{\Spheqtres})$ norm on the target and the maximum norm on the domain subject to the metric $\delta^{-2}g_{\Spheqtres}$ on $\Spheqtres$.
	  		\end{enumerate}	
	  	\end{lemma}
	  	
	  	\begin{proof}
	  		(i-v) in \ref{lem:obstr} just follows by the definitions \eqref{eq:VW}, \eqref{eq:VWbar} and \ref{def:VW}. Notice that by \eqref{eq:delta} $\delta^{-2}g_{\Spheqtres}\sim\tilde{g}$, (vi) then follows by (iii).
	  	\end{proof}
	  	
	  	\subsection*{The family of LD solutions}
	  	\begin{definition}[The space of parameters]\label{def:zeta}
	  		We define the space of parameters $\Pcal:=\R^2$. The continuous parameter of the LD solutions is $(\zeta,\bar{\zeta})\in\BPcal$, where 
	  		\begin{equation}\label{eq:zeta}
	  			\BPcal:=\{(\zeta,\bar{\zeta})\in\Pcal:\abs{\zeta}\leq 1/m^{\alpha'}, \abs{\bar{\zeta}}\leq 1/m^{\alpha'}\},
	  		\end{equation}
	  		where $\alpha'\in(0,1)$ is a small constant which will be fixed later (cf. \eqref{eq:alphaprime}).
	  	\end{definition}
	  	
	  	\begin{definition}[LD solutions {$\varphi\bbracket{\zeta}$}]\label{def:varphi}
For $(\zeta,\bar{\zeta})\in \BPcal$ as in \ref{def:zeta}, 
we define the LD solution $\varphi\bbracket{\zeta;m}$,  
$\underline{v}\in\skernelv[L]$ and $\bar{\underline{v}}\in \skernelev[L]$  
using \ref{lem:LD} or \ref{def:Phi} (recall \eqref{eq:T3} and \eqref{eq:phi})
by 
	  		\begin{align}
	  			\label{eq:varphizeta}
	  			&\varphi=\varphi\bbracket{\zeta}=\varphi\bbracket{\zeta;m}:=\tau\Phi[m]=\varphi[\tau;m],\\
\label{eq:vunder}
	  			\underline{v}:=& -\Ecal_L^{-1}\Mcal_L\varphi\in\skernelv[L],\qquad \bar{\underline{v}}:=\bar{\underline{v}}[\bar{\zeta}]:=\tau\bar{\zeta}\phi_1 \bar{V} \in \skernelev[L], 
\\
	  			\label{eq:tau}
	  			\text{where }\quad &\tau=\tau\bbracket{\zeta}=\tau\bbracket{\zeta;m}:=\left(\frac{T_3}{\phi_1}\right)^2e^{2\zeta}=\left(\frac{ \pi F_0T_3}{m^2}\right)^2e^{2\zeta}.
	  		\end{align}
\end{definition}
	  	
	  	\begin{lemma}[Matching equation and matching estimate for {$\varphi\bbracket{\zeta}$}]\label{lem:miss}
	  		Given $\varphi=\varphi\bbracket{\zeta}$ as in \ref{def:varphi} with $\zeta\in\BPcal$, 
we have for $\Mcal_L(\varphi+\bar{\underline{v}})=\tau\mu\in\Vcal_{\sym}[L]$, and $\Lcal(\varphi+\bar{\underline{v}})=\Lcal \bar{\underline{v}}=-\tau \bar{\mu} \bar{W}\in \bar{\Vcal}_{\sym}[L]$  (recall \ref{def:mismatch}),
	  		\begin{equation}\label{eq:mu}
	  			\mu=\phi_1(1-e^{-\zeta}+\bar{\zeta})+\Phi'(p),\text{ where }p\in L,\quad \bar{\mu}=-\bar{\zeta}\phi_1.
	  		\end{equation}
	  		Moreover with $\alpha'$ as in \ref{def:zeta} we have 
	  		\begin{equation}\label{eq:zetamu}
	  			\abs{\zeta-\frac{1}{\phi_1}(\mu+\bar{\mu})}\lesssim \frac{1}{m^{\alpha'}}, \qquad  \bar{\zeta}+\frac{1}{\phi_1}\bar{\mu}=0.
	  		\end{equation}
	  	\end{lemma}	
	  	\begin{proof}
	  		For any $p\in L$, by  \eqref{eq:varphizeta}, \ref{def:mismatch} and \ref{def:decompose}, inside $D^{\Spheqtres}_p(2\delta)$
	  		\begin{align}\label{eq:hatphi}
	  			\hat{\varphi}_p=\tau(\underline{\phi}[\phi_1]+\Phi')
	  		\end{align}
	  		By \ref{def:mismatch}, \ref{def:phiju}, \eqref{eq:Phiavg} and \eqref{eq:tau}
	  		\begin{align*}
	  			\Mcal_L(\varphi+\bar{\underline{v}})&=\hat{\varphi}_p-\sqrt{\tau} T_3+\tau\bar{\zeta}\phi_1=\tau(\phi_1+\Phi'(p))-\sqrt{\tau}+ T_3\tau\bar{\zeta}\phi_1\\
	  			&=\sqrt{\tau} T_3(e^{\zeta}-1)+\tau\Phi'(p)+\tau\bar{\zeta}\phi_1.
	  		\end{align*}
	  		Then \eqref{eq:mu} follows by \eqref{eq:tau} and \eqref{eq:vunder}. \eqref{eq:zetamu} follows by \eqref{eq:zeta} and \ref{lem:Phiprime}.
	  	\end{proof}

	  	\begin{lemma}[Estimates for {$\varphi\bbracket{\zeta}$}]\label{lem:varphiest}
	  		Given $\varphi\bbracket{\zeta}$ and $\tau$ as in \ref{def:varphi}, the following hold when $m$ is large enough.  
	  		\begin{enumerate}[(i)]
	  				\item\label{item:hatvarphi} $\forall p\in L$, $\norm{\hat{\varphi}_p-\sqrt{\tau}T_3\underline{\phi}[1]+\underline{v}+\bar{\underline{v}}:C^{2,\beta}(\partial D_p^{\Spheqtres}(\delta),\delta,g)}\lesssim m^{-3}$.
	  			\item\label{item:solboundn} $\norm{\varphi+\underline{v}+\bar{\underline{v}}:C^{k,\beta}(\Spheqtres\setminus D_L^{\Spheqtres}(\delta'),\rr,g)}\lesssim_k m^{-2}$, where $\rr:=\min\{\dist^{\Spheqtres}_L,\delta\}$.
	  			\item\label{item:solembed} On $\Spheqtres\setminus \sqcup_{q\in L}D_q^{\Spheqtres}(\delta')$ we have $m^{-2}\lesssim \varphi+\underline{v}+\bar{\underline{v}}$.
	  			\item\label{item:solboundgsphere} $\norm{\varphi+\underline{v}+\bar{\underline{v}}:C^{k,\beta}(\Spheqtres\setminus D_{\T_0}^{\Spheqtres}(3/m),g)}\lesssim m^{-2}+ m^{-3+k+\beta}$.
	  		\end{enumerate}
	  	\end{lemma}
	  	\begin{proof}
	  	From \eqref{eq:hatphi} and using \eqref{eq:tau} and \eqref{eq:phi} inside $D^{\Spheqtres}_p(\delta)$ we have
	  	\begin{equation*}
	  		\hat{\varphi}_p-\sqrt{\tau}T_3\underline{\phi}[1]=\tau(\underline{\phi}[\phi_1]+\Phi')-\sqrt{\tau}T_3\underline{\phi}[1]=\sqrt{\tau}T_3(e^{\zeta}-1)\underline{\phi}[1]+\tau\Phi'.
	  	\end{equation*}
	  	On the other hand, by definition \eqref{eq:vunder}, \eqref{eq:VW}, \eqref{eq:mu} and \eqref{eq:tau}, inside $D^{\Spheqtres}_p(\delta)$,
	  	\begin{equation*}
	  		\underline{v}+\bar{\underline{v}}=-\tau\mu V-\tau\bar{\mu}\bar{V}=-\tau\phi_1(1-e^{-\zeta})\underline{\phi}[1]-\tau\Phi'(p)\underline{\phi}[1]=-\sqrt{\tau}T_3(e^{\zeta}-1)\underline{\phi}[1]-\tau\Phi'(p)\underline{\phi}[1].
	  	\end{equation*}
      Thus inside $D^{\Spheqtres}_p(\delta)$
        \begin{equation*}
        	\hat{\varphi}_p-\sqrt{\tau}T_3\underline{\phi}[1]+\underline{v}+\bar{\underline{v}}=\tau(\Phi'-\Phi'(p)\underline{\phi}[1]).
        \end{equation*}
	  	 (i) then follows by \ref{lem:Phiprime}.
	  	
	  	By \eqref{eq:mu}, \eqref{eq:zeta}, \eqref{eq:phi} and \ref{lem:Phiprime}, $\abs{\mu}+\abs{\bar{\mu}}\lesssim  m^{2-\alpha'}$.
	  	Thus by \ref{lem:obstr}\ref{item:V}
	  \begin{equation}\label{eq:vu}
	  	\norm{\underline{v}:C^{2,\beta}( \Spheqtres,\delta,g)}+	\norm{\bar{\underline{v}}:C^{2,\beta}( \Spheqtres,\delta,g)}\lesssim m^{-2-\alpha'}.
	  \end{equation}
	  By using \ref{lem:phiju}\ref{item:ju}, \eqref{eq:phi} \ref{lem:GEPhi}\ref{item:Ghatprime} and \ref{lem:Phiprime} to control the terms in the decomposition \ref{def:decompose} along with \eqref{eq:deltaprime} to compare the norms, we have
	  \begin{align*}
	  	&\norm{\Phi:C^{k,\beta}(\Spheqtres\setminus \sqcup_{q\in L}D_q^{\Spheqtres}(\delta'),\rr,g)}\lesssim_k m^2,\\
	  \end{align*}
	  and $m^2\lesssim\Phi$ on $\Spheqtres\setminus \sqcup_{q\in L}D_q^{\Spheqtres}(\delta')$. (ii), (iii) then follows by  \ref{def:varphi} and \eqref{eq:vu}.
	  
	  Finally, on $\Spheqtres\setminus D_{\T_0}^{\Spheqtres}(3/m)$, by \ref{def:decompose}, \ref{lem:obstr}\ref{item:vsupp} and the definitions \eqref{eq:phi}, \eqref{eq:VW}
	  \begin{equation*}
	  	\varphi+\underline{v}+\bar{\underline{v}}=\tau(\phi+\Phi'_{\osc}+\bar{\zeta}\phi_1 \bar{V})=\tau(1+\bar{\zeta})\phi_1\hat{\phi} +\tau\Phi'_{\osc},
	  \end{equation*}
	  (iv) then follows by \ref{lem:Phiprime} and \eqref{eq:tau}, \eqref{eq:zeta}.
	\end{proof}
	  
  \section{The initial surfaces}
\label{S:init} 

\subsection*{Mean curvature on the sphere}

\begin{lemma}\label{lem:SphereH}
	There exists an absolute constant $\epsilon_{\Sph}$, such that given a region $\Omega\subset \Spheqtres$ $\varphi\in C^{3,\beta}(\Omega)$ with $\norm{\varphi:C^{3,\beta}(\Omega,g)}\leq \epsilon_{\Sph}$, $g=g_{\Spheqtres}$, then
	\begin{equation*}
		\norm{Q_{\varphi}:C^{k,\beta}(\Omega,g)}\lesssim \sum_{k_1+k_2+k_3=k}\norm{\varphi:C^{k_1+2,\beta}(\Omega,g)}\norm{\varphi:C^{k_2+1,\beta}(\Omega,g)}\norm{\varphi:C^{k_3+1,\beta}(\Omega,g)},
	\end{equation*}
	where $Q_{\varphi}$ is defined to satisfy $H_{\varphi}=\Lcal_{\Spheqtres}\varphi+Q_{\varphi}$ and $H_{\varphi}\circ\Pi_{\Spheqtres}$ is the mean curvature of the graph $\Graph_{\Omega}^{\Sph^4}(\varphi)$ pushed forward to $\Spheqtres$ by the projection $\Pi_{\Spheqtres}$ (recall \ref{not:manifold}\ref{item:proj} and \ref{not:manifold}\ref{item:graph}).
\end{lemma}	
\begin{proof}
	In this proof we fix the notation $X$ to denote the standard embedding of $\Sph^4\subset \R^5$, and we use $X_{\varphi}:=X^{\Sph^4,g_{\Sph^4}}_{\Omega,\varphi}$ (recall \ref{not:manifold}\ref{item:graph}) to denote the embedding of $\Graph_{\Omega}^{\Sph^4}(\varphi)$ into $\Sph^4\subset \R^5$. We also use $g_{\varphi}$, $A_{\varphi}$, $\nu_{\varphi}$ to denote the first and second fundamental forms and normal vector of $\Graph_{\Omega}^{\Sph^4}(\varphi)$ induced by $(\Sph^4,g_{\Sph^4})$. For the calculation of the fundamental forms, we use the formula $(g_{\varphi})_{ij} = \langle \partial_{i}X_{\varphi}, \partial_{j}X_{\varphi} \rangle$, $(A_{\varphi})_{ij} = \langle \partial_{i}\partial_{j}X_{\varphi}, \nu_{\varphi} \rangle$ and we have chosen some local orthonormal coordinates at a point of $\Spheqtres$, where $i,j$ take the values 1, 2 and 3.
	
	We then calculate ($\varphi_i$ denotes $\partial_{i}\varphi$ and similarly for second order derivatives)
	\begin{align*}
		&X_{\varphi}=(X\cos\varphi,\sin\varphi),\quad
		(g_{\varphi})_{ij}=\cos^2\varphi\delta_{ij}+\varphi_i\varphi_j,\\
		&\nu_{\varphi}=(-X\sin\varphi-\sum_i \varphi_i \partial_i X/ \cos\varphi,\cos\varphi)/\sqrt{1+\abs{\nabla\varphi}^2/\cos^2\varphi},\\
		&(A_{\varphi})_{ij}=(\cos\varphi\sin\varphi\delta_{ij}+2\varphi_i\varphi_j\tan\varphi+\varphi_{ij})/\sqrt{1+\abs{\nabla\varphi}^2/\cos^2\varphi}.\\
	\end{align*}
From the expression of $(g_{\varphi})_{ij}$ and Sherman-Morrison formula, 
we can calculate the inverse matrix
	\begin{equation*}
		(g_{\varphi})^{ij}=\delta_{ij}/\cos^2\varphi-\varphi_i\varphi_j/(\cos^4\varphi+\cos^2\varphi\abs{\nabla\varphi}^2).
	\end{equation*}
	Thus
	\begin{multline}
		\label{eq:SphereH}
		\sqrt{1+\abs{\nabla\varphi}^2/\cos^2\varphi}H_{\varphi}=\sqrt{1+\abs{\nabla\varphi}^2/\cos^2\varphi}\sum_{i,j}(g_{\varphi})^{ij}(A_{\varphi})_{ij}\\
		=\frac{\Delta \varphi}{\cos^2\varphi}+3\tan\varphi+\frac{2\abs{\nabla\varphi}^2\tan\varphi}{\cos^2\varphi}-\frac{\mathrm{Hess}(\varphi)(\nabla\varphi,\nabla\varphi)+\abs{\nabla\varphi}^2\cos\varphi\sin\varphi+\abs{\nabla\varphi}^4\tan\varphi}{\cos^4\varphi+\cos^2\varphi\abs{\nabla\varphi}^2 }.
	\end{multline}
	The result then follows by \eqref{eq:jacobi}.
\end{proof}

\begin{lemma}\label{lem:varphinl}
	Given $\varphi=\varphi\bbracket{\zeta}$, $\underline{v}$, $\bar{\underline{v}}$ as in \ref{def:varphi}, there exists $\varphi_{nl}=\varphi_{nl}\bbracket{\zeta,\bar{\zeta}}\in C^{\infty}_{\sym}(\Spheqtres\setminus L)$ such that the following hold.
	\begin{enumerate}[(i)]
		\item $H_{\varphi_{nl}}=0$ on $\Spheqtres\setminus  D_{\T_0}^{\Spheqtres}(4/m)$
		and $H_{\varphi_{nl}}$ is the mean curvature of the graph
		$\Graph_{\Spheqtres}^{\Sph^4}(\varphi_{nl})$ pushed forward to $\Spheqtres$ by the projection $\Pi_{\Spheqtres}$ (recall \ref{not:manifold}\ref{item:proj} and \ref{not:manifold}\ref{item:graph}).
		\item\label{item:varphinlest} $\varphi_{nl}-\varphi-\underline{v}-\bar{\underline{v}}$ can be extended to a smooth function supported on $\Spheqtres\setminus  D_{\T_0}^{\Spheqtres}(3/m)$ which satisfies
		\begin{equation*}
			\norm{\varphi_{nl}-\varphi-\underline{v}-\bar{\underline{v}}:C^{3,\beta}(\Spheqtres,g)}\lesssim m^{-4+6\beta}.
		\end{equation*}
	\end{enumerate}
\end{lemma}
\begin{proof}
	We first define cutoff function $\psi''\in C^{\infty}_{\sym}(\Spheqtres)$ 
	\begin{equation*}
		\psi'':=\cutoff{3/m}{4/m}{\dist_{\T_0}^{\Spheqtres}}(0,1).
	\end{equation*}
	We then define inductively sequences $\{u_n\}_{n=1}^{\infty}\subset C^{2,\beta}_{\sym}(\Spheqtres)$ and $\{\phi_n\}_{n=-1}^{\infty}\subset C^{2,\beta}_{\sym}(\Spheqtres)$ by $\phi_{-1}=0$, $\phi_{0}=\varphi+\underline{v}+\bar{\underline{v}}$, and for $n>0$,
	\begin{equation*}
		\phi_n=\phi_{n-1}+u_n,\quad \Lcal_{\Spheqtres} u_n=\psi''(Q_{\phi_{n-2}}-Q_{\phi_{n-1}}),
	\end{equation*}
	where $Q_{\phi_{k}}$ is the nonlinear term of the mean curvature of the graph $\Graph_{\Spheqtres}^{\Sph^4}(\phi_{k})$ as in \ref{lem:SphereH}. By standard linear theory and the triviality of the kernel of 	$\Lcal_{\Spheqtres}$ under the symmetries in \ref{lem:nokernel}, we conclude that for $n \geq 1$ we have for $\Omega:=\Spheqtres\setminus D_{\T_0}^{\Spheqtres}(3/m)$
	\begin{multline*} 
		\norm{u_n:C^{3,\beta}(\Spheqtres,g)}\lesssim \norm{\psi''(Q_{\phi_{n-2}}-Q_{\phi_{n-1}}):C^{1,\beta}(\Omega ,g)}\\
		\lesssim m^{\beta}\norm{Q_{\phi_{n-2}}-Q_{\phi_{n-1}}:C^{1,\beta}(\Omega ,g)}.
	\end{multline*} 
	
	From \eqref{eq:SphereH} we have by taking $m$ big enough,
	
	\begin{align*}
		\norm{Q_{\phi_{n-2}}-Q_{\phi_{n-1}}:C^{1,\beta}(\Omega ,g)}\lesssim 	
		\begin{cases} 
			\norm{\varphi+\underline{v}+\bar{\underline{v}}:C^{3,\beta}(\Omega,g)}\norm{\varphi+\underline{v}+\bar{\underline{v}}:C^{1,\beta}(\Omega,g)}^2 \    \\
			+\norm{\varphi+\underline{v}+\bar{\underline{v}}:C^{2,\beta}(\Omega,g)}^2\norm{\varphi+\underline{v}+\bar{\underline{v}}:C^{1,\beta}(\Omega,g)},\ &n=1;\\
			\norm{\phi_{n-2}-\phi_{n-1}:C^{3,\beta}(\Omega,g)}\norm{\phi_{n-2}:C^{3,\beta}(\Omega,g)}^2,\     &n\geq 2.
		\end{cases}
	\end{align*}
	
	Combining the last two estimates and substituting $u_{n-1}$ for $\phi_{n-1}-\phi_{n-2}$
	\begin{align*}
		\norm{u_n:C^{3,\beta}(\Omega ,g)}\lesssim 	
		\begin{cases} 
			m^{\beta}\Big(\norm{\varphi+\underline{v}+\bar{\underline{v}}:C^{3,\beta}(\Omega,g)}\norm{\varphi+\underline{v}+\bar{\underline{v}}:C^{1,\beta}(\Omega,g)}^2\\
			+\norm{\varphi+\underline{v}+\bar{\underline{v}}:C^{2,\beta}(\Omega,g)}^2\norm{\varphi+\underline{v}+\bar{\underline{v}}:C^{1,\beta}(\Omega,g)}\Big),\quad &n=1;\\
			m^{\beta}\norm{u_{n-1}:C^{3,\beta}(\Omega,g)}\norm{\phi_{n-2}:C^{3,\beta}(\Omega,g)}^2, \quad &n\geq 2.
		\end{cases}
	\end{align*}
	We then conclude inductively that for $n \geq 1$ by taking $m$ big enough and \ref{lem:varphiest}\ref{item:solboundgsphere},  
	\begin{multline*}
		\norm{u_n:C^{2,\beta}(\Spheqtres ,g)}
		\lesssim 2^{-n}m^{3\beta}\Big(\norm{\varphi+\underline{v}+\bar{\underline{v}}:C^{3,\beta}(\Omega,g)}\norm{\varphi+\underline{v}+\bar{\underline{v}}:C^{1,\beta}(\Omega,g)}^2\\
		+\norm{\varphi+\underline{v}+\bar{\underline{v}}:C^{2,\beta}(\Omega,g)}^2\norm{\varphi+\underline{v}+\bar{\underline{v}}:C^{1,\beta}(\Omega,g)}\Big).	
	\end{multline*}
	Taking limits and sums and using standard regularity theory for the smoothness we conclude the proof along with \ref{lem:varphiest}\ref{item:solboundgsphere}.
\end{proof}

\subsection*{The initial surfaces and their regions}
\begin{definition}[Initial surfaces] 
\label{def:initial}
	Given $\varphi=\varphi\bbracket{\zeta}$, $\underline{v}$, $\bar{\underline{v}}$ as in \ref{def:varphi},
	we define the smooth initial surface (recall \ref{not:manifold}\ref{item:graph} and \ref{def:catebridge})
	\begin{equation*}
		M=M[\zeta,\bar{\zeta}] := \Graph_{\Omega}^{\Sph^4}(\varphi^{gl})\cup \Graph_{\Omega}^{\Sph^4}(-\varphi^{gl})\cup\sqcup_{p\in L}\Kcech[p,\radius], 
	\end{equation*}
	where $\radius:=\sqrt{\tau}$, $\Omega:=\Spheqtres\setminus D_L^{\Spheqtres}(9\radius)$ and the functions $\varphi^{gl}:=\varphi^{gl}\bbracket{\zeta,\bar{\zeta}}:\Omega\to\R$ are defined as the following.
	\begin{equation}
		\varphi^{gl}:=
		\begin{cases} 
			\varphi_{nl} & \text{on } \Spheqtres\setminus  D_L^{\Spheqtres}(3\delta'), \\
			\Psi[2\delta',3\delta';\dist^{\Spheqtres}_p](\varphi_{cat}[\radius]\circ \dist^{\Spheqtres}_p,\varphi_{nl}) & \text{on }  \sqcup_{p\in L} D_p^{\Spheqtres}(3\delta')\setminus D_p^{\Spheqtres}(\radius).
		\end{cases}
	\end{equation}
	
\end{definition}

\begin{lemma}[The gluing region]\label{lem:glue}
	For $M=M[\zeta,\bar{\zeta}]$ as in \ref{def:initial}, $g=g_{\Spheqtres}$ and $\forall p\in L$, the following hold.
	\begin{enumerate}[(i)]
		\item\label{item:gluecat} $\norm{\varphi^{gl}-\varphi_{cat}[\sqrt{\tau}]:C^{2,\beta}(D^{\Spheqtres}_p(4\delta')\setminus D^{\Spheqtres}_p(\delta'),\delta',g)}\lesssim m^{-3-2\alpha}$.
		\item\label{item:gluevarphi} $\norm{\varphi^{gl}:C^{3,\beta}(D^{\Spheqtres}_p(4\delta')\setminus D^{\Spheqtres}_p(\delta'),\delta',g)}\lesssim m^{-2}$.
		\item\label{item:glueH} $\norm{(\delta')^2 H':C^{2,\beta}(D^{\Spheqtres}_p(3\delta')\setminus D^{\Spheqtres}_p(2\delta'),\delta',g)}\lesssim m^{-3-2\alpha}$, where $H'$ denotes the pushforward of the mean curvature of the graph of $\varphi^{gl}$ to $\Spheqtres$ by $\Pi_{\Spheqtres}$.
	\end{enumerate}
\end{lemma}
\begin{proof}
	We have for each $p \in L$ on $D^{\Spheqtres}_p(4\delta')\setminus D^{\Spheqtres}_p(\delta')$
	\begin{align}\label{eq:varphiplusminus}
		&\varphi^{gl}=\tau G_p+\sqrt{\tau}T_3\underline{\phi}[1]+\cutoff{2\delta'}{3\delta'}{\dist_p^{\Spheqtres}}(\varphi_-,\varphi_+),\nonumber\\
		\text{where }&\varphi_-:=\varphi_{cat}[\sqrt{\tau}]-\tau G_p-\sqrt{\tau}T_3\underline{\phi}[1],\\
		&\varphi_+:=\hat{\varphi}_p-\sqrt{\tau}T_3\underline{\phi}[1]+\underline{v}.\nonumber
	\end{align} 
	Thus by \ref{eq:tau}
	\begin{align}\label{eq:phiglestn}
		\norm{\varphi^{gl}}\lesssim m^{-2}+\norm{\varphi_-}+\norm{\varphi_+},
	\end{align}
	where in this proof when we do not specify the norm we mean the $C^{3,\beta}(D^{\Spheqtres}_p(4\delta')\setminus D^{\Spheqtres}_p(\delta'),\delta',g)$ unless specified otherwise.
	
	Note that on $D^{\Spheqtres}_p(4\delta')\setminus D^{\Spheqtres}_p(\delta')$ 
	\begin{align*}
		\varphi^{gl}-\varphi_{cat}[\sqrt{\tau}]&=\cutoff{2\delta'}{3\delta'}{\dist_p^{\Spheqtres}}(0,\varphi_+-\varphi_-),\\
		\Lcal_{\Spheqtres}\varphi^{gl}&=\Lcal_{\Spheqtres}\cutoff{2\delta'}{3\delta'}{\dist_p^{\Spheqtres}}(\varphi_-,\varphi_+).
	\end{align*}
	Using these, we have by \ref{lem:green}\ref{item:greenest} and \eqref{eq:deltaprime}
	\begin{align}
		\norm{\varphi^{gl}-\varphi_{cat}[\sqrt{\tau}]}&\lesssim \norm{\varphi_-}+\norm{\varphi_+},\label{eq:phiglcat}\\
		\norm{(\delta')^2\Lcal_{\Spheqtres}\varphi^{gl}:C^{0,\beta}(D^{\Spheqtres}_p(4\delta')\setminus D^{\Spheqtres}_p(\delta'),\delta',g)}&\lesssim \norm{\varphi_-}+\norm{\varphi_+}.\label{eq:Lphigl}
	\end{align}

	By \eqref{eq:varphiplusminus}, \ref{lem:varphicat}, \ref{lem:green}\ref{item:greenest} and \eqref{eq:deltaprime}
	\begin{align}\label{eq:phiminusestn}
		\norm{\varphi_-}&\lesssim \norm{\varphi_{cat}[\sqrt{\tau}]-\sqrt{\tau}T_3-\tau(\dist_p^{\Spheqtres})^{-1}}+\norm{\tau G_p-\tau(\dist_p^{\Spheqtres})^{-1}}+\norm{\sqrt{\tau}T_3(1-\underline{\phi}[1])}\\
		&\lesssim \tau^3(\delta')^{-5}+\tau\delta'+\sqrt{\tau}(\delta')^2 \lesssim  m^{-4-2\alpha}+m^{-7+5\alpha}.\nonumber
	\end{align}

	Because $\Lcal_{\Spheqtres}\varphi_+ = 0$ on $D^{\Spheqtres}_p(4\delta')\setminus D^{\Spheqtres}_p(\delta')$ and $\Ecalu_{p}\varphi_+=0$ (recall \ref{def:mismatch} and \eqref{eq:varphiplusminus}), it follows from standard linear theory, \ref{def:sol} and \ref{lem:varphiest}\ref{item:hatvarphi} that
	\begin{equation}\label{eq:phiplusestn}
		\norm{\varphi_+}\lesssim\left(\frac{\delta'}{\delta}\right)^2\norm{\varphi_+:C^{2,\beta}(\partial D_p^{\Spheqtres}(\delta),\delta,g)}\lesssim m^{-3-2\alpha}. 
	\end{equation}
	(i) now follows by \eqref{eq:phiglcat}, \eqref{eq:phiminusestn} and \eqref{eq:phiplusestn}, and (ii) now follows by \eqref{eq:phiglestn}, \eqref{eq:phiminusestn} and \eqref{eq:phiplusestn}.

	By expanding the mean curvature in linear and higher order terms we have
	\begin{equation*}
		(\delta')^2 H'=	(\delta')^2\Lcal_{\Spheqtres}\varphi^{gl}+\delta'\tilde{Q}_{(\delta')^{-1}\varphi^{gl}}.
	\end{equation*}
	Thus by \ref{lem:SphereH}
	\begin{multline} 
		\label{eq:thus} 
		\norm{\delta'\tilde{Q}_{(\delta')^{-1}\varphi^{gl}}:C^{0,\beta}(D^{\Spheqtres}_p(4\delta')\setminus D^{\Spheqtres}_p(\delta'),\delta',g)}\\
		\lesssim(\delta')^{-2}\norm{\varphi^{gl}:C^{2,\beta}(D^{\Spheqtres}_p(4\delta')\setminus D^{\Spheqtres}_p(\delta'),\delta',g)} \norm{\varphi^{gl}:C^{1,\beta}(D^{\Spheqtres}_p(4\delta')\setminus D^{\Spheqtres}_p(\delta'),\delta',g)}^2\\
		\lesssim m^{-4+2\alpha}.
	\end{multline}
	We conclude along with \eqref{eq:Lphigl}
	\begin{equation*}
		\norm{(\delta')^2 H':C^{0,\beta}(D^{\Spheqtres}_p(4\delta')\setminus D^{\Spheqtres}_p(\delta'),\delta',g)}\lesssim m^{-7+5\alpha}+m^{-3-2\alpha}+m^{-4+2\alpha}.
	\end{equation*}
	
	By \ref{eq:alpha} and assuming $m$ big enough (iii) follows. 
\end{proof}

\begin{lemma}\label{lem:phigl}
	$M$ defined in \ref{def:initial} is embedded and moreover the following hold.
	\begin{enumerate}[(i)]
		\item\label{item:phiglembed} On $\Spheqtres\setminus D_L^{\Spheqtres}(\delta')$ we have $m^{-2}\lesssim \varphi^{gl}$.
		\item\label{item:phiglbound} $\norm{\varphi^{gl}:C^{3,\beta}(\Spheqtres\setminus D_L^{\Spheqtres}(\delta'),\rr,g)}\lesssim m^{-2}$, where $\rr:=\min\{\dist^{\Spheqtres}_L,\delta\}$.
	\end{enumerate}
\end{lemma}
\begin{proof}
	(i) on $\Spheqtres\setminus D_L^{\Spheqtres}(3\delta')$ follows from \ref{lem:varphiest}\ref{item:solembed}, \ref{lem:varphinl}\ref{item:varphinlest} and \ref{def:initial}, and on each $D^{\Spheqtres}_p(4\delta')\setminus D^{\Spheqtres}_p(\delta')$ follows from  \ref{lem:glue}\ref{item:gluecat}. (ii) on $\Spheqtres\setminus D_L^{\Spheqtres}(3\delta')$ follows from \ref{lem:varphiest}\ref{item:solboundn}, \ref{lem:varphinl}\ref{item:varphinlest} and \ref{def:initial}, and on each $D^{\Spheqtres}_p(4\delta')\setminus D^{\Spheqtres}_p(\delta')$ follows from  \ref{lem:glue}\ref{item:gluevarphi}. Finally, the embeddedness of $M$ follows from (i) and by comparing the rest of $M$ with standard catenoids. 
\end{proof}
\begin{definition}\label{def:region}
	We define the regions $\Kcech[M]\subset M$  and $\Ku[M]\subset M$ by (recall \ref{def:catebridge})
	\begin{equation}\label{eq:KcheckM}
		\Kcech[M]:=\sqcup_{p\in L}\Kcech[p,\radius],\quad \Ku[M]:=\sqcup_{p\in L}\Ku[p,\radius].
	\end{equation}
	We define the map $\Pi_{\K}:\Kcech[M]\to\K_M:=\sqcup_{p\in L}\K[p,\radius]$ by taking $\Pi_{\K}:=(\exp_p^{\Spheqtres,\Sph^4})^{-1}$ on each $\Kcech[p,\radius]$. We also define $\underline{\K}[p]:=\Pi_{\K}(\Ku[p,\radius])$. Finally, we define the region $\tilde{S'}\subset \Spheqtres$ by
	\begin{equation}\label{eq:tildeSprime}
		\tilde{S'}:=\Spheqtres\setminus D^{\Spheqtres}_L(b\radius).
	\end{equation}
\end{definition}

  \section{The linearized equation on the initial surfaces}
\label{S:lin} 

\subsection*{Global norms and the mean curvature on the initial surfaces}

\begin{definition}\label{def:norm}
	For $k \in \{0,2\}$, $\hat{\beta} \in (0, 1)$, $\hat{\gamma} \in \mathbb{R}$, $\gamma'\in (0,1)$ and $\Omega$ a domain in $\Spheqtres$ or $M$ or $\K_M$ define
	\begin{equation*}
		\norm{u}_{k,\hat{\beta},\hat{\gamma},\gamma';\Omega}:=\norm{u:C^{k,\hat{\beta}}(\Omega,\rr,g,f_{k,\hat{\gamma},\gamma'})},
	\end{equation*}
	where $\rr:=\min\{\dist_L^{\Spheqtres},\delta\}$ and $g$ is the standard metric on $\Spheqtres$ when $\Omega\subset\Spheqtres$; $\rr:=\min\{\dist_L^{\Spheqtres}\circ\Pi_{\Spheqtres},\delta\}$ and $g$ is the metric on $M$ induced by the standard metric on $\Sph^4$ when $\Omega\subset M$; $\rr = \rho(s)$ (recall \eqref{eq:rhot}) and $g$ is the metric induced by each Euclidean metric $g_{\Sph^4}|_p$ on $T_p\Sph^4$ $\forall p \in L$ when $\Omega\subset \K_M$; $f_{k,\hat{\gamma},\gamma'}$ is defined then by $\rr^{\hat{\gamma}}f_{k,\gamma'}$ when $\Omega\subset\Spheqtres$, by $\rr^{\hat{\gamma}}f_{k,\gamma'}\circ\Pi_{\Spheqtres}$ when $\Omega\subset M$ and by $\rr^{\hat{\gamma}}$ when $\Omega\subset \K_M$ (recall \ref{def:weight}).
\end{definition}

\begin{convention} 
\label{conv:gamma}
From now on we assume that $b$ (recall \ref{def:catebridge}) is as large as needed in absolute terms, and $m$ is as big and thus $\tau$ is as small as needed in absolute terms and $b$. We also fix some $\beta\in(0,1/100)$ and $\gamma\in (1,2)$. 
We will suppress the dependence of various constants on $\beta$. 
We then fix $\alpha$ as in \ref{def:sol} and $\alpha'>0 $ such that 
	  	\begin{equation}\label{eq:alphaprime} 
\alpha' < (2-\gamma)\alpha < 1/4 .
	  	\end{equation}
\end{convention} 

\begin{notation}\label{not:Bq}
	Throughout this subsection we let $q\in \tilde{S'}$ and consider the metric $\tilde{g}_q :=\rr(q)^{-2} g$ on $\Sph^4$, where $\rr(q):=\min\{\delta,\dist_L^{\Spheqtres}(q)\}$, $g=g_{\Sph^4}$ is the standard metric on $\Sph^4$. In this metric $M$ is locally the union of the graphs of $\pm\varphi_{:q}$, where $\varphi_{:q}:=\rr(q)^{-1} \varphi^{gl}$. We also define $\widecheck{B}_q:=D^{\Spheqtres,\tilde{g}_q}_q(1/10)$.
\end{notation}

\begin{lemma}\label{lem:equivnorm}
	For $k=0,2$ and $\hat{\gamma}\in\R$ the following hold.
	\begin{enumerate}[(i)]
		\item\label{item:equivnormK} If $\widecheck{\Omega}$ is a domain in $\Pi_{\K}(\Kcech[M])$ (recall \eqref{eq:KcheckM}), $\Omega:=\Pi_{\K}^{-1}(\widecheck{\Omega})\subset \Kcech[M]\subset M$ and $f\in C^{k,\beta}(\widecheck{\Omega})$, then
		\begin{equation*}
			\norm{f\circ\Pi_{\K}}_{k,\beta,\hat{\gamma},\gamma';\Omega}\sim 	\norm{f}_{k,\beta,\hat{\gamma},\gamma';\widecheck{\Omega}}.
		\end{equation*}
		\item\label{item:equivnormS} If $\Omega'$ is a domain in $\tilde{S'}$ (recall \ref{eq:tildeSprime}), $\Omega:=\Pi_{\Spheqtres}^{-1}(\Omega')\cap M$ and $f\in C^{k,\beta}(\Omega')$, then
		\begin{equation*}
			\norm{f\circ\Pi_{\Spheqtres}}_{k,\beta,\hat{\gamma},\gamma';\Omega}\sim 	\norm{f}_{k,\beta,\hat{\gamma},\gamma';\Omega'}.
		\end{equation*}
	\end{enumerate}
\end{lemma}
\begin{proof}
	To prove (i) it suffices to prove for each $p\in L$ and each $\Kcech=\Kcech[p,\radius]$ that (recall \ref{lem:HK})
	\begin{equation*}
		\norm{f\circ\Pi_{\K}:C^{k,\beta}(\Omega\cap\Kcech ,\rho,\mathring{g}_{\Kcech})}\sim 	\norm{f\circ\Pi_{\K}:C^{k,\beta}(\Omega\cap\Kcech ,\rho,g_{\Kcech})}.
	\end{equation*}
	(i) then follows from e.g. \cite[Lemma C.12]{gLD} and the estimate on $g_{\Kcech}-\mathring{g}_{\Kcech}$ in	\ref{lem:HK}\ref{item:fermialpha} by taking $m$ big enough as $\rho\lesssim m^{-1-\alpha'}$.
	
	To prove (ii) first suppose that $\dist_p^{\Spheqtres}(q)\leq 4\delta'$ for some $p \in L$. Then $\rr(q)=\dist_p^{\Spheqtres}(q)$ by \ref{not:Bq}. Note that (recall \ref{def:catebridge})
	\begin{equation*}
		\norm{(\sqrt{\tau}T_3-\tau/\dist_p^{\Spheqtres})/\dist_p^{\Spheqtres}:C^k(\widecheck{B}_q,\tilde{g}_q)}\lesssim_k b^{-1}.
	\end{equation*} 
	It follows by combining this with \ref{lem:varphicat}, \ref{lem:glue}\ref{item:gluecat}, \eqref{eq:alpha} and the definition \ref{def:initial}, and assuming $b$ and then $m$ large enough such that
	\begin{equation}\label{eq:varphiqin}
		\norm{\varphi_{:q}:C^k(\widecheck{B}_q,\tilde{g}_q)}\lesssim_k \norm{\frac{\sqrt{\tau}T_3-\tau/\dist_p^{\Spheqtres}}{\dist_p^{\Spheqtres}}:C^k(\widecheck{B}_q,\tilde{g}_q)}+\frac{m^{-12}}{(\dist_p^{\Spheqtres}(q))^6}+\frac{m^{-4}}{\delta'}\lesssim_{k} b^{-1}.
	\end{equation}
	On the other hand, if $\dist_p^{\Spheqtres}(q)\geq  4\delta'$ , then by \ref{lem:phigl}\ref{item:phiglbound} and \eqref{eq:deltaprime} we have
	\begin{equation}\label{eq:varphiqout}
		\norm{\varphi_{:q}:C^k(\widecheck{B}_q,\tilde{g}_q)}\lesssim_{k} m^{-1+\alpha}.
	\end{equation}
	By comparing the metrics and appealing to the definitions we complete the proof.
\end{proof}

\begin{lemma}[Mean curvature on the initial surfaces]\label{lem:H}
	We have the global estimate 
	\\ 
	\begin{equation}\label{eq:w}
		\begin{gathered} 
			\norm{H-(w+\bar{w})\circ\Pi_{\Spheqtres}}_{0,\beta,\gamma-2;M}\lesssim m^{-3-2\alpha+\gamma(1+\alpha)}, 
			\\ 
			\text{where} \qquad 
			w:=-\Lcal_{\Spheqtres}\Ecal_L^{-1}\Mcal_L\varphi=\Lcal_{\Spheqtres}\underline{v},\qquad \bar{w}:=\Lcal_{\Spheqtres}\bar{\underline{v}}.
		\end{gathered} 
	\end{equation} 
\end{lemma}
\begin{proof}
	Notice that on $M\cap\Pi_{\Spheqtres}^{-1}(D^{\Spheqtres}_L(3\delta'))$, $\rr=\dist_L^{\Spheqtres}$ by \eqref{eq:deltaprime}. Thus on $\Kcech[M]$, $\rr=\dist_L^{\Spheqtres}$ and $w\circ\Pi_{\Spheqtres}=0$ by \ref{lem:obstr}\ref{item:wsupp}, we have by \ref{lem:HK}\ref{item:fermiH} and \ref{lem:equivnorm}\ref{item:equivnormK}
	\begin{equation*}
		\norm{H}_{0,\beta,\gamma-2;\Kcech[M]}\lesssim m^{-2}(\delta')^{2-\gamma}.
	\end{equation*}
	Similarly, on the gluing region $M\cap\Pi_{\Spheqtres}^{-1}(\sqcup_{p\in L}D^{\Spheqtres}_p(3\delta')\setminus D^{\Spheqtres}_p(2\delta'))$, by \ref{lem:glue}\ref{item:glueH} and \eqref{eq:alpha}
	\begin{equation*}
		\norm{H}_{0,\beta,\gamma-2;D^{\Spheqtres}_p(3\delta')\setminus D^{\Spheqtres}_p(2\delta'))}\lesssim m^{-1}(\delta')^{2-\gamma}=m^{-3-2\alpha+\gamma(1+\alpha)}.
	\end{equation*}
	
	Now we consider the estimate on the exterior of the gluing region. In this region $\varphi^{gl}=\varphi_{nl}$. Let $q'\in M\setminus\Pi_{\Spheqtres}^{-1}(D^{\Spheqtres}_L(3\delta'))$, and $q:=\Pi_{\Spheqtres}(q')\in \Spheqtres\setminus D^{\Spheqtres}_L(3\delta')$. 
	By expanding $H'$ in linear and higher order terms, we find 
	\begin{equation*}
		\rr^2 H'=\rr^2 w+\rr^2 \bar{w}+\rr \tilde{Q}_{\varphi_{:q}}.
	\end{equation*}
	We estimate by \ref{lem:phigl}\ref{item:phiglbound} and by \ref{lem:SphereH}
	\begin{multline*}
		\norm{\rr^{2-\gamma} (H'-w-\bar{w}):C^{0,\beta}(D_{\T_0}^{\Spheqtres}(4/m)\setminus D_p^{\Spheqtres}(3\delta'),\rr,g)}\\
		\lesssim  \norm{\rr^{1-\gamma} \tilde{Q}_{\varphi_{:q}}:C^{0,\beta}(D_{\T_0}^{\Spheqtres}(4/m)\setminus D_p^{\Spheqtres}(3\delta'),\rr,g)}\\
		\lesssim  (\delta')^{-2-\gamma}\norm{\varphi^{gl}:C^{2,\beta}(D_{\T_0}^{\Spheqtres}(4/m)\setminus D_p^{\Spheqtres}(3\delta'),\rr,g)}\norm{\varphi^{gl}:C^{1,\beta}(D_{\T_0}^{\Spheqtres}(4/m)\setminus D_p^{\Spheqtres}(3\delta'),\rr,g)}^2\\
		\lesssim   m^{-4+2\alpha+(1+\alpha)\gamma}.
	\end{multline*}
	Finally by \ref{lem:varphinl} the mean curvature is $0$ outside of $D_{\T_0}^{\Spheqtres}(4/m)$.
	
	Combining all these estimates with the definition of the norm \ref{def:norm} and \ref{lem:equivnorm}\ref{item:equivnormS} we complete the proof. 
\end{proof}

\begin{lemma}\label{lem:equivop}
	For $\hat{\gamma}\in\R$, the following hold.
	\begin{enumerate}[(i)]
		\item \label{item:equivopK}If $u\in C^{2,\beta}(\Pi_{\K}(\Kcech[M]))$, then
		\begin{equation*}
			\norm{\Lcal_M(u\circ\Pi_{\K})-(\Lcal_{\K}u)\circ\Pi_{\K}}_{0,\beta,\hat{\gamma}-2,\gamma';\Kcech[M]}\lesssim (\delta')^2	\norm{u}_{2,\beta,\hat{\gamma},\gamma';\Pi_{\K}(\Kcech[M])}.
		\end{equation*}
		\item\label{item:equivopS} If $u\in C^{2,\beta}(\tilde{S'})$, then 
		\begin{equation*}
			\norm{\Lcal_M(u\circ\Pi_{\Spheqtres})-(\Lcal_{\Spheqtres}u)\circ\Pi_{\Spheqtres}}_{0,\beta,\hat{\gamma}-2,\gamma';\Pi_{\Spheqtres}^{-1}(\tilde{S'})}\lesssim b^{-1}	\norm{u}_{2,\beta,\hat{\gamma},\gamma';\tilde{S'}}.
		\end{equation*}
	\end{enumerate}
\end{lemma}

\begin{proof}
	We first prove (i). By \ref{def:norm} and \ref{lem:equivnorm}\ref{item:equivnormK}, it suffices to prove that
	\begin{equation*}
		\norm{\rho^2(\Lcal_M-\Delta_{\mathring{g}_{\Kcech}}-\abs{\mathring{A}_{\Kcech}}^2_{\mathring{g}_{\Kcech}})(u\circ\Pi_{\K})}_{0,\beta,\hat{\gamma},\gamma';\Kcech[M]}\lesssim (\delta')^2	\norm{u}_{2,\beta,\hat{\gamma},\gamma';\Pi_{\K}(\Kcech[M])},
	\end{equation*}
	where $\rho$ is as in \eqref{eq:rhot} on $\Kcech[M]$ by \ref{def:catebridge}, $\mathring{g}_{\Kcech}$ and $\mathring{A}_{\Kcech}$ are metric and the second fundamental form on on each $\Kcech=\Kcech[p,\sqrt{\tau}]$ induced by $\mathring{g}$. Recall from \eqref{eq:gK} that $(\Pi_{\K})_*(\rho^{-2}\mathring{g}_{\Kcech})$ is the flat metric $\chi$ on $\Cylinder$ from \ref{not:cyl}, and thus $\rho^2\Delta_{\mathring{g}}(u\circ\Pi_{\K})=(\Delta_{\chi}u)\circ\Pi_{\K}$. Estimating the difference in the Laplacians using \cite[Lemma C.10(iv)]{gLD}, we find
	\begin{multline*}
		\norm{\rho^2(\Delta_{g_{\Kcech}}-\Delta_{\mathring{g}_{\Kcech}})(u\circ\Pi_{\K}):C^{0,\beta}(\Kcech,\chi,\rho^{\hat{\gamma}})}\\
		\lesssim \norm{\rho^{-2}(g_{\Kcech}-\mathring{g}_{\Kcech}):C^{1,\beta}(\Kcech,\chi)} \norm{u\circ\Pi_{\K}:C^{2,\beta}(\Kcech,\chi,\rho^{\hat{\gamma}})}\lesssim (\delta') ^2\norm{u}_{2,\beta,\hat{\gamma},\gamma';\Pi_{\K}(\Kcech[M])},
	\end{multline*}
	where we have used \ref{lem:HK}\ref{item:fermialpha} to estimate $g_{\Kcech}-\mathring{g}_{\Kcech}$. Next observe that by \ref{lem:HK}\ref{item:fermiA}, \eqref{eq:deltaprime} and $\radius=\sqrt{\tau}$ now
	\begin{multline*}
		\norm{\rho^{2}(\abs{A_{\Kcech}}^2_{g_{\Kcech}}-\abs{\mathring{A}_{\Kcech}}^2_{\mathring{g}_{\Kcech}})(u\circ\Pi_{\K}):C^{0,\beta}(\Kcech,\chi,\rho^{\hat{\gamma}})}\\
		\lesssim \left(\norm{\rho^{-2}(g_{\Kcech}-\mathring{g}_{\Kcech}):C^{0,\beta}(\Kcech,\chi)}+\norm{\rho^{-1}(A_{\Kcech}-\mathring{A}_{\Kcech}):C^{0,\beta}(\Kcech,\chi)}\right) \norm{u\circ\Pi_{\K}:C^{2,\beta}(\Kcech,\chi,\rho^{\hat{\gamma}})}\\
		\lesssim  (\delta')^2\norm{u}_{2,\beta,\hat{\gamma},\gamma';\Pi_{\K}(\Kcech[M])}.
	\end{multline*}
	
	We now prove (ii). By \eqref{eq:varphiqin} and \eqref{eq:varphiqout} we have using scaling for the left hand side that for $q\in\tilde{S'}$ 
	\begin{multline*}
		\rr(q)^2\norm{\Lcal_M(u\circ\Pi_{\Spheqtres})-(\Lcal_{\Spheqtres}u)\circ\Pi_{\Spheqtres}:C^{0,\beta}(\Pi_{\Spheqtres}^{-1}(\widecheck{B}_q),\tilde{g}_q)}
		\\ 
		\lesssim (b^{-1}+m^{-1+\alpha})\norm{u:C^{2,\beta}(\widecheck{B}_q,\tilde{g}_q)}.
	\end{multline*}
	The results then follows by taking $m$ big enough in terms of $b$.
\end{proof}

Using now the same rescaling we prove a global estimate for the nonlinear terms of the mean curvature of the graph over the initial surfaces as follows.
\begin{lemma}\label{lem:nonlinear}
	Let $M$ as in \ref{def:initial} and $\textup{ф}\in C^{2,\beta}(M)$ satisfies $\norm{\textup{ф}}_{2,\beta,\gamma,\gamma';M}\leq m^{-3+\gamma}$, then $M_{\textup{ф}}:=\Graph_M^{\Sph^4}(\textup{ф})$ is well defined, is embedded, and if $H_{\textup{ф}}$ is the mean curvature of $M_{\textup{ф}}$ pulled back to $M$ by $\Graph_M^{\Sph^4}(\textup{ф})$ and $H$ is the mean curvature of $M$, then
	\begin{equation*}
		\norm{H_{\textup{ф}}-H-\mathcal{L}_M\textup{ф}}_{0,\beta,\gamma-2,\gamma';M}\lesssim  \norm{\textup{ф}}_{2,\beta,\gamma,\gamma';M}^2
	\end{equation*}
\end{lemma}

\begin{proof}
	By \eqref{eq:varphiqin} and \eqref{eq:varphiqout}, we have that for $q\in\tilde{S'}$ the graph $\widecheck{B}'_q$ of $\varphi_{:q}$ over $\widecheck{B}_q$ in $(\Spheqtres,\tilde{g}_q)$ can be described by an immersion $X_{:q} : \widecheck{B}_q\to \widecheck{B}'_q$ such that there are coordinates on $\widecheck{B}_q$ and a neighborhood in $\Sph^4$ of $\widecheck{B}'_q$ which are uniformly bounded and the immersion in these coordinates has uniformly bounded $C^{2,\beta}$ norms, the standard Euclidean metric on the domain is bounded by $X^*_{:q}\tilde{g}_q$, and the coefficients of $\tilde{g}_q$ in the target coordinates have uniformly bounded $C^{1,\beta}$ norms. By the definition of the norm \ref{def:norm} and since $\norm{\textup{ф}}_{2,\beta,\gamma,\gamma';M}\leq m^{-3+\gamma}$, we have that the restriction of $\textup{ф}$ on $\widecheck{B}'_q$ satisfies
	\begin{equation*}
		\norm{\rr^{-1}(q)\textup{ф}:C^{2,\beta}(\widecheck{B}'_q,\tilde{g}_q)}\lesssim\rr^{\gamma-1}(q)f_{2,\gamma'}(q)\norm{\textup{ф}}_{2,\beta,\gamma,\gamma';M}
	\end{equation*}
	Since the right hand side is small in absolute terms we can conclude that $\Graph_{\widecheck{B}'_q}^{\Sph^4}(\textup{ф})$ is well defined and embedded. Using scaling for the left hand side we further conclude that
	\begin{equation*}
		\norm{\rr(q)(H_{\textup{ф}}-H-\mathcal{L}_M\textup{ф}):C^{0,\beta}(\widecheck{B}'_q,\tilde{g}_q)}\lesssim\rr^{2\gamma-2}(q)f^2_{2,\gamma'}(q)\norm{\textup{ф}}_{2,\beta,\gamma,\gamma';M}^2.
	\end{equation*}
	Thus by \ref{def:weight}, the definition of $\rr$ in \ref{def:norm} and by taking $m$ big enough, we have
	\begin{multline*}
		\frac{\rr^{2-\gamma}(q)}{f_{0,\gamma'}(q)}\norm{(H_{\textup{ф}}-H-\mathcal{L}_M\textup{ф}):C^{0,\beta}(\widecheck{B}'_q,\tilde{g}_q)}
		\\ 
		\lesssim  \rr^{\gamma-1}(q)\frac{f^2_{2,\gamma'}(q)}{f_{0,\gamma'}(q)}\norm{\textup{ф}}_{2,\beta,\gamma,\gamma';M}^2\lesssim \delta^{\gamma-1}\norm{\textup{ф}}_{2,\beta,\gamma,\gamma';M}^2.
	\end{multline*}
	Finally, note that that the components of $\Ku[M]$ appropriately scaled are small perturbations of a fixed compact region of the standard catenoid, which allows us to repeat the arguments above in this case (by taking $m$ big enough in terms of $b$). By combining with the earlier estimates and using the definitions, we conclude the estimate in the statement of the lemma.
\end{proof}	

\subsection*{The definition of {$\Rcal_M^{appr}$}}
\begin{definition}\label{def:cutoff}
	Define $\psi'\in C^{\infty}(\Spheqtres)$ and $\widecheck{\psi}\in C^{\infty}(M)$ by requesting the following.
	\renewcommand{\theenumi}{\roman{enumi}}
	\begin{enumerate}
		\item $\widecheck{\psi}$ is supported on $\Kcech[M]\subset M$ and $\psi'$ on $\tilde{S}'\subset\Spheqtres$.
		\item $\psi'=1$ on $\Spheqtres\setminus D_{L}^{\Spheqtres}(2b\tau)\subset\Spheqtres$ and for each $p\in L$ we have
		\begin{align*}
			\psi'&=\Psibold[b\tau,2b\tau;\dist^{\Spheqtres}_p](0,1)\text{ on }D_{p}^{\Spheqtres}(2b\tau),\\
			\widecheck{\psi}&=\Psibold[2\delta',\delta';\dist^{\Spheqtres}_p\circ\Pi_{\Spheqtres}](0,1)\text{ on } \Kcech[p,\sqrt{\tau}].
		\end{align*}
	\end{enumerate}
	We also define the region $\widecheck{\Omega}[M]$ containing the support of $\widecheck{\psi}$ by
	\begin{equation*}
		\widecheck{\Omega}[M]:=\widecheck{\Omega}[p],\quad \text{where }\widecheck{\Omega}[p]:=\Pi_{\Spheqtres}^{-1}(D_p^{\Spheqtres}(2\delta')\setminus D_p^{\Spheqtres}(\delta')).
	\end{equation*}
\end{definition}

As in \cite{SdI} and \cite{gLD}, we will construct a linear
map $\mathcal{R}_{M,appr}:C^{0,\beta}_{\sym}(M)\to C^{2,\beta}_{\sym}(M)\oplus \skernel[L]\oplus\skernele[L]\oplus C^{0,\beta}_{\sym}(M)$, so that $(u_1,w_{E,1},\bar{w}_{E,1},E_1):=\mathcal{R}_{M,appr}(E)$ is an approximate solution to the equation \ref{prop:lineareq}\ref{item:LDeqmodulow}. The approximate solution will be constructed by combining semi-local approximate solutions.

Given $E\in C^{0,\beta}_{\sym}(M)$, we define $E'\in C^{0,\beta}_{\sym}(\Spheqtres)$ by requiring that they are supported on $\tilde{S}'$ and that
\begin{equation}\label{eq:Eprime}
	E'\circ\Pi_{\Spheqtres}=(\psi'\circ\Pi_{\Spheqtres})E.
\end{equation}
By \ref{lem:nokernel} and \ref{lem:obstr}, there are unique $u'\in C^{2,\beta}_{\sym}(\Spheqtres)$, $w_{E,1}\in \skernel[L]$, $\bar{w}_{E,1}\in \skernele[L]$ such that
\begin{equation}\label{eq:uprime}
	\Lcal_{\Spheqtres}u'=E'+w_{E,1}+\bar{w}_{E,1},\quad \Ecalu_p u'=0, \forall p\in L;\quad u'|_\uC\equiv 0.    
\end{equation}
Note that $\Lcal_{\Spheqtres}((1-\psi')u')=[\psi',\Lcal_{\Spheqtres}]u'+(1-\psi')E'$ is supported on $\Ku[M]\subset\Kcech[M]\subset M$, we define $\tilde{E}\in C^{0,\beta}_{\sym}(\K_M)$, by requesting that it is supported on $\Pi_{\K}(\Ku[M])$ and that on $\Ku[M]$ we have
\begin{equation}\label{eq:Etilde}
	\tilde{E}\circ\Pi_{\K}=(1-\psi'\circ\Pi_{\Spheqtres})E+(\Lcal_{\Spheqtres}((1-\psi')u'))\circ\Pi_{\Spheqtres}.
\end{equation}

For $k\in\{0,2\}$, we introduce decompositions $C^{k,\beta}(\K_M)=C^{k,\beta}_{low}(\K_M)\oplus C^{k,\beta}_{high}(\K_M)$ and also $H^{1}(\K_M)=H^{1}_{low}(\K_M)\oplus H^{1}_{high}(\K_M)$ into subspaces of functions which satisfy the condition that their restrictions to a parallel circle of a $\K[p,\sqrt{\tau}]$ belong or are ($L^2$-)orthogonal respectively to the the span of the constants and the first harmonics on the circle. Notice that $\tilde{E}$ has compact support, we then have
\begin{equation}
	\tilde{E}=\tilde{E}_{low}+\tilde{E}_{high},
\end{equation}
with $\tilde{E}_{low}\in C^{0,\beta}_{low}(\K_M)\cap H^{1}_{low}(\K_M)$, $\tilde{E}_{high} \in C^{0,\beta}_{high}(\K_M)\cap H^{1}_{high}(\K_M)$ supported on $\Pi_{\K}(\Ku[M])\subset \K_M$.

Let $\Lcal_{\K}$ denote the Jacobi linear operator on $\K[p,\sqrt{\tau}]$ (recall \ref{def:jacobi}), we define $\tilde{u}=\tilde{u}_{low}+\tilde{u}_{high}$ by requesting $\tilde{u}_{low}\in C^{2,\beta}_{low}(\K_M)$, $\tilde{u}_{high} \in C^{2,\beta}_{high}(\K_M)$ to be the solutions of
\begin{equation}\label{eq:utilde}
	\mathcal{L}_{\K}\tilde{u}_{low}= \tilde{E}_{low},\quad \mathcal{L}_{\K}\tilde{u}_{high}= \tilde{E}_{high}
\end{equation}
determined uniquely as follows. By separating variables the first equation amounts to uncoupled ODE equations which are solved uniquely by assuming vanishing initial data on the waist of the catenoids. For the second equation we use the following lemma.

\begin{lemma}[Linear estimate on catenoids]\label{lem:utildehigh}
	For $\K=\K[p,\sqrt{\tau}]$, $\tilde{E}_{high} \in C^{0,\beta}_{high}(\K)\cap H^{1}_{high}(\K)$ supported on $\underline{\K}=\underline{\K}[p]$  (recall \ref{def:catebridge} and \ref{def:region}), there exists a unique $\tilde{u}_{high} \in C^{2,\beta}_{high}(\K)\cap H^{1}_{high}(\K)$ satisfying the second equation in \eqref{eq:utilde} and 
	\begin{equation*}
		\norm{\tilde{u}_{high}:C^0(\K)}\lesssim_b \tau\norm{\tilde{E}_{high}:C^{0,\beta}(\underline{\K},\tau^{-1}g_{\K})}.
	\end{equation*}
\end{lemma} 
\begin{proof}
	By scaling we can rewrite the the second equation in \eqref{eq:utilde} to
	\begin{equation}\label{eq:utildehigh}
		\mathcal{L}_{\tilde{\K}}\tilde{u}_{high}= \tau\tilde{E}_{high},
	\end{equation}
	where $\tilde{\K}=\K[p,1]$ and $\mathcal{L}_{\tilde{\K}}$ is the Jacobi linear operator on it. 
	
	We first show that the bilinear form $B:H^{1}_{high}(\tilde{\K})\times H^{1}_{high}(\tilde{\K})\to\R$ is coercive, where 
	\begin{equation*}
		B(u,v):=\int_{\tilde{K}} \langle \nabla u,\nabla v \rangle -\abs{A_{\tilde{K}}}^2uv,
	\end{equation*}  
	and coercivity means that there exists $c>0$ such that $\abs{B(u,u)}\geq c \norm{u:H^{1}(\tilde{\K})}^2$. However, this follows by that $H^{1}_{high}(\tilde{\K})$ is spanned by eigenfunctions of $\mathcal{L}_{\tilde{\K}}$ with positive eigenvalues and that $\abs{A_{\tilde{K}}}^2$ is uniformly bounded by \ref{lem:AK}.
	
	Now by Lax-Milgram Theorem, there is then a unique $\tilde{u}_{high}\in H^{1}_{high}(\tilde{\K})$ satisfy \eqref{eq:utildehigh} with
	\begin{equation*}
		\norm{\tilde{u}_{high}:H^{1}(\tilde{\K})}\lesssim\norm{\tau\tilde{E}_{high}:L^{2}(\tilde{\K})}\lesssim_b\tau\norm{\tilde{E}_{high}:C^{0,\beta}(\underline{\K},\tau^{-1}g_{\K})}.
	\end{equation*}
	The fact that $\tilde{u}_{high}$ belongs to $C^{2,\beta}(\K)$ follows by the local estimates of Schauder theory. Finally by De Giorgi-Nash-Moser iteration (e.g. \cite[Theorem 4.1]{han2011elliptic}) and noticing that $\tilde{E}_{high}\in L^{\infty}(\K)$ has compact support, the $C^0$ estimate of $\tilde{u}_{high}$ follows.
\end{proof}

We conclude now the definition of $\mathcal{R}_{M,appr}$:
\begin{definition}\label{def:Rappr}
	Define the operator 
	\begin{equation*}
		\mathcal{R}_{M,appr}:C_{\sym}^{0,\beta}(M)\to C_{\sym}^{2,\beta}(M)\oplus \skernel[L]\oplus\skernele[L]\oplus C_{\sym}^{0,\beta}(M)
	\end{equation*}
	by $ \mathcal{R}_{M,appr}E=(u_1,w_{E,1},\bar{w}_{E,1},E_1)$, where $w_{E,1},\bar{w}_{E,1}$ are as above, $u_1:=\widecheck{\psi}\tilde{u}\circ\Pi_{\K}+(\psi'u')\circ\Pi_{\Spheqtres}$, $E_1=\Lcal_M u_1-E-w\circ\Pi_{\Spheqtres}$.
\end{definition}

\subsection*{The main Proposition}

\begin{proposition}\label{prop:lineareq}
	A linear map $\mathcal{R}_M:C^{0,\beta}_{\sym}(M)\to C^{2,\beta}_{\sym}(M) \oplus \skernel[L]\oplus \skernele[L]$, $E\mapsto (u,w_E,\bar{w}_E)$ can be defined by
	\begin{equation*}
		\mathcal{R}_M:=(u,w_E,\bar{w}_E):=\sum_{n=1}^{\infty}(u_n,w_{E,n},\bar{w}_{E,n}),
	\end{equation*}
	where the sequence $\{(u_n,w_{E,n},\bar{w}_{E,n},E_n)\}_{n\in \mathbb{N}}$ is defined inductively by
	\begin{equation*}
		(u_n,w_{E,n},\bar{w}_{E,n},E_n):=-\mathcal{R}_{M,appr}E_{n-1}, \quad E_0:=-E.
	\end{equation*}
	Moreover, the following hold:
	\renewcommand{\theenumi}{\roman{enumi}}
	\begin{enumerate}
		\item\label{item:LDeqmodulow} $\Lcal_M u=E+w_E\circ\Pi_{\Spheqtres}+\bar{w}_{E}\circ\Pi_{\Spheqtres}$.
		\item $\norm{u}_{2,\beta,\gamma,\gamma';M}\lesssim \norm{E}_{0,\beta,\gamma-2,\gamma';M}$ and $\abs {\mu_{E}}+m^{-1}\abs{\bar{\mu}_{E}}\lesssim m^{4-\gamma}\norm{E}_{0,\beta,\gamma-2,\gamma';M}$, where $\mu_E,\bar{\mu}_E$ are defined such that $ -\tau\mu_E W=w_E$, $ -\tau\bar{\mu}_E \bar{W}=\bar{w}_E$ (recall \ref{def:VW}).
	\end{enumerate}
\end{proposition}
\begin{proof}
	\emph{Step 1:}
	By the definitions and the equivalence of the norms from \ref{lem:equivnorm},
	\begin{equation*}
		\norm{E'}_{0,\beta,\gamma-2,\gamma';\Spheqtres}\lesssim \norm{E}_{0,\beta,\gamma-2,\gamma';M}.
	\end{equation*}
	For each $p \in L$, using e.g. \cite[Proposition C.1]{breiner:kapouleas:high} or \cite[Lemma D.1]{gLD}, we can solve the equation $\mathcal{L}_{\Spheqtres}u'_{p}=E'$ on $D_p^{\Spheqtres}(2\delta)$ with $\Ecalu_p u'_{p} = 0$, and that the restriction of $u'_{p}$ on $\partial D_p^{\Spheqtres}(2\delta)$ is a constant (by the symmetry there is no component corresponding to first harmonic); similarly, by \ref{lem:linearSd} with $l=2$, we have a function $u'_{\uC}:=\Rcal_{\Sph}(0,E')$, such that $\mathcal{L}_{\Spheqtres}u'_{\uC}=E'$ on $\Spheqtres\setminus D_{\T_0}^{\Spheqtres}(2/m)$, and the restriction of $u'_{\uC}$ on $\partial D_{\T_0}^{\Spheqtres}(2/m)=\T_{2/m}$ is a constant. 
	Moreover, by the definition of the norm \ref{def:norm}
	\begin{equation*}
		\norm{u'_{p}}_{2,\beta,\gamma,\gamma';D_p^{\Spheqtres}(2\delta)}\lesssim \norm{E'}_{0,\beta,\gamma-2,\gamma';\Spheqtres},\quad \norm{u'_{\uC}}_{2,\beta,\gamma,\gamma';\Spheqtres\setminus D_{\T_0}^{\Spheqtres}(2/m)}\lesssim \norm{E'}_{0,\beta,\gamma-2,\gamma';\Spheqtres}.
	\end{equation*}
	
	We define now $u''\in C^{2,\beta}(\Spheqtres)$ supported on $ D_L^{\Spheqtres}(2\delta)\sqcup (\Spheqtres\setminus D_{\T_0}^{\Spheqtres}(2/m))$ by requesting that for each $p \in L$,
	\begin{equation*}
		u'':=\cutoff{2\delta}{\delta}{\dist^{\Spheqtres}_p}(0,u'_{p})+\cutoff{2/m}{3/m}{\dist^{\Spheqtres}_{\T_0}}(0,u'_{\uC}).
	\end{equation*}
	Clearly, then
	\begin{equation*}
		\norm{u''}_{2,\beta,\gamma,\gamma';\Spheqtres}\lesssim \norm{E}_{0,\beta,\gamma-2,\gamma';M}.
	\end{equation*}
	
	Now	$E'-\Lcal_{\Spheqtres} u''$ is supported on $D_{\T_0}^{\Spheqtres}(3/m)\setminus D_L^{\Spheqtres}(\delta)$. Moreover, it satisfies
	\begin{equation*}
		\norm{E'-\Lcal_{\Spheqtres} u''}_{0,\beta,\gamma-2,\gamma';\Spheqtres}\lesssim \norm{E}_{0,\beta,\gamma-2,\gamma';M}.
	\end{equation*}
	Using the definition of the norms \ref{def:norm} and the restricted support, for $g=g_{\Sph^3}$
	\begin{equation*}
		\norm{E'-\Lcal_{\Spheqtres} u'':C^{0,\beta}(\Spheqtres,\delta,g)}\lesssim\delta^{\gamma-2}\norm{E'-\Lcal_{\Spheqtres} u''}_{0,\beta,\gamma-2,\gamma';\Spheqtres}.
	\end{equation*}
	The last two estimates and linear theory \ref{lem:linearS} and \ref{lem:obstr}\ref{item:Vbar} imply that there exists $\bar{\mu}_{E,1}\in\R$ such that the unique solution $u'''\in C^{2,\beta}(\Spheqtres)$ to $\Lcal_{\Spheqtres}u'''=E'-\Lcal_{\Spheqtres} u''$ satisfies
	\begin{multline*} 
		m^{-1}\tau\abs{\bar{\mu}_{E,1}}+\norm{u'''-\tau\bar{\mu}_{E,1}\bar{V}:C^{2,\beta}(\Spheqtres,\delta,g,f_{0,\gamma'})}\\
		\lesssim  m^{-2}\norm{E'-\Lcal_{\Spheqtres} u'':C^{0,\beta}(\Spheqtres,\delta,g)}\lesssim  \delta^{\gamma}\norm{E}_{0,\beta,\gamma-2,\gamma';M}.
	\end{multline*} 
	By \ref{lem:obstr}, there is a unique $v\in\skernelv[L]$ such that $\Ecalu_p (u''' + v+\bar{v})=0$ for each $p \in L$, where $\bar{v}:=-\tau\bar{\mu}_{E,1}\bar{V}$. Moreover, by the last estimate and \ref{lem:obstr}, $v$ satisfies the estimates
	\begin{align*}
		&\norm{v:C^{2,\beta}(\Spheqtres,\delta,g)}\lesssim \norm{\Ecal_L^{-1}}\norm{u'''-\tau\bar{\mu}_{E,1}\hat{\phi}:C^{0}(\Spheqtres,\delta,g)}\lesssim \delta^{\gamma} \norm{E}_{0,\beta,\gamma-2,\gamma';M},\\
		&\tau\abs{\mu_{E,1}}\lesssim\norm{u'''-\tau\bar{\mu}_{E,1}\bar{V}:C^{0}(\Spheqtres,\delta,g)}\lesssim \delta^{\gamma} \norm{E}_{0,\beta,\gamma-2,\gamma';M},
	\end{align*}
	where $\mu_{E,1}$ is defined such that $-\tau\mu_{E,1}V=v$ by \ref{def:VW}. By the definition of $u'''$, $\Lcal_{\Spheqtres}(u''+u'''+v+\bar{v})=E'+w_{E,1}+\bar{w}_{E,1}$, where $w_{E,1}:=\mathcal{L}_{\Spheqtres}v$ and $\bar{w}_{E,1}:=\mathcal{L}_{\Spheqtres}\bar{v}$.
	By the definitions of $u''$ and $v,\bar{v}$, $\Ecalu_{p}(u''+u'''+v+\bar{v})=0$, $\forall p\in L$, $(u''+u'''+v+\bar{v})|_\uC=0$ hence by \eqref{eq:uprime}
	\begin{equation*}
		u'=u''+u'''+v+\bar{v}.
	\end{equation*}
	
	By definitions $\mathcal{L}_{\Spheqtres}(u'''+v+\bar{v})=w_{E,1}+\bar{w}_{E,1}=0$ on $ D_L^{\Spheqtres}(\delta/4)$, moreover, $\Ecalu_p(u''' + v+\bar{v})=0$ for each $p \in L$, by  standard theory and separation of variables to estimate with decay of $u''' + v+\bar{v}$ on $ D_p^{\Spheqtres}(\delta/4)$ in terms of the Dirichlet data on $\partial D_p^{\Spheqtres}(\delta/4)$ along with the earlier estimates for $u'''+\bar{v}$ and $v$,
	\begin{equation*}
		\norm{u''' + v+\bar{v}}_{2,\beta,\gamma,\gamma';\Spheqtres}\lesssim  \norm{E}_{0,\beta,\gamma-2,\gamma';M},
	\end{equation*}

	Along with the estimate for $u''$,
	\begin{equation}\label{eq:uprimeest}
		\norm{u'}_{2,\beta,\gamma,\gamma';\Spheqtres}\lesssim \norm{E}_{0,\beta,\gamma-2,\gamma';M}.
	\end{equation}
	
	\emph{Step 2:}
	By the definition \eqref{eq:Etilde}, \ref{lem:equivnorm} and the estimate \eqref{eq:uprimeest}
	\begin{equation*}
		\norm{\tilde{E}}_{0,\beta,\gamma-2,\gamma';\K_M}\lesssim \norm{E}_{0,\beta,\gamma-2,\gamma';M}+\norm{u'}_{2,\beta,\gamma-2,\gamma';\Spheqtres}\lesssim \norm{E}_{0,\beta,\gamma-2,\gamma';M}.
	\end{equation*}
	By scaling the first equation of \eqref{eq:utilde}, the definitions of norms \ref{def:norm}, \ref{lem:equivnorm}\ref{item:equivnormK}, the definition \eqref{eq:Etilde} and standard theory, we conclude $\forall p\in L$ that (recall 
	\eqref{eq:gK})
	\begin{equation*}
		\norm{\tilde{u}_{low}:C^{2,\beta}(\underline{\K}[p],\tau^{-1}g_{\K})}\lesssim_b \tau\norm{\tilde{E}_{low}:C^{0,\beta}(\underline{\K}[p],\tau^{-1}g_{\K})} \lesssim_b \tau^{\gamma/2}\norm{\tilde{E}}_{0,\beta,\gamma-2,\gamma';\K_M}.
	\end{equation*}
	Notice that the ODE solutions of the Jacobi equation corresponding to constants do not grow in $\rho$ 
	(by the symmetry there is no component corresponding to first harmonic)
	and that $\tilde{E}$ is supported on $\sqcup_{p\in L}\underline{\K}[p]$,  we conclude by comparing weights of norms and using that $\rho\geq\sqrt{\tau}$ on $\K=\K[p,\sqrt{\tau}]$ and \ref{conv:gamma}, that
	\begin{align}\label{eq:ulow}
		\norm{\tilde{u}_{low}}_{2,\beta,0;\K}\lesssim \norm{\tilde{u}_{low}}_{2,\beta,0;\underline{\K}[p]}
		\lesssim_b  \norm{\tilde{u}_{low}:C^{2,\beta}(\underline{\K}[p],\tau^{-1}g)}\lesssim_b \tau^{\gamma/2} \norm{\tilde{E}}_{0,\beta,\gamma-2,\gamma';\K_M}.
	\end{align}
	Thus by the definition of norms \ref{def:norm}
	\begin{equation*}
		\norm{\tilde{u}_{low}}_{2,\beta,\gamma,\gamma';\K}\lesssim\tau^{-\gamma/2} \norm{\tilde{u}_{low}}_{2,\beta,0;\K}\lesssim_b  \norm{\tilde{E}}_{0,\beta,\gamma-2,\gamma';\K_M}.
	\end{equation*}
	
	By \ref{lem:utildehigh}, the definitions of norms \ref{def:norm} and \ref{lem:equivnorm}\ref{item:equivnormK}, we conclude that
	\begin{equation*}
		\norm{\tilde{u}_{high}:C^{0}(\K)}\lesssim_b \tau \norm{\tilde{E}_{high}:C^{0,\beta}(\underline{\K}[p],\tau^{-1}g_{\K})} \lesssim_b \tau^{\gamma/2}\norm{\tilde{E}}_{0,\beta,\gamma-2,\gamma';\K_M}.
	\end{equation*}
	Using standard theory on $\K\cong\Cylinder$ by $X_{\K}$ (recall \eqref{eq:XK}), the support of $\tilde{E}$ and \ref{lem:equivnorm}\ref{item:equivnormK} again, we have

\begin{multline} 
\label{eq:uhigh}
\norm{\tilde{u}_{high}}_{2,\beta,0;\K} \, \lesssim_b	
\\ 
\tau\norm{\tilde{E}_{high}:C^{0,\beta}(\underline{\K}[p],\tau^{-1}g_{\K})} \, +	 \, \norm{\tilde{u}_{high}:C^{0}(\K)} 
 \, \lesssim_b  \, \tau^{\gamma/2}\norm{\tilde{E}}_{0,\beta,\gamma-2,\gamma';\K_M}.
\end{multline}
	Thus by the definitions of norms \ref{def:norm} 
	\begin{equation*}
		\norm{\tilde{u}_{high}}_{2,\beta,\gamma,\gamma';\K}\lesssim \tau^{-\gamma/2}\norm{\tilde{u}_{high}}_{2,\beta,0;\K}\lesssim_b \norm{\tilde{E}}_{0,\beta,\gamma-2,\gamma';\K_M}.
	\end{equation*}
Combining the estimates on $\tilde{u}_{low}$ and $\tilde{u}_{high}$ we have
	\begin{equation}\label{eq:utildeest}
		\norm{\tilde{u}}_{2,\beta,\gamma,\gamma';\K_M}\lesssim_b \norm{\tilde{E}}_{0,\beta,\gamma-2,\gamma';\K_M} \lesssim_b  \norm{E}_{0,\beta,\gamma-2,\gamma';M} .
	\end{equation}
	
	\emph{Step 3:}
	Using \ref{def:Rappr}, \eqref{eq:Eprime}, \eqref{eq:uprime}, \eqref{eq:Etilde}, \eqref{eq:utilde}, we decompose $E_1$ as
	\begin{equation}
		E_1=E_{1,I}+E_{1,II}+E_{1,III},
	\end{equation}
	where $E_{1,I},E_{1,II},E_{1,III}\in C^{0,\beta}(M)$ are defined by (recall \ref{def:cutoff})
	\begin{align}
		E_{1,I}&:=[\Lcal_M,\widecheck{\psi}](\tilde{u}\circ \Pi_{\K}),\label{eq:EI}\\
		E_{1,II}&:=\widecheck{\psi}(\Lcal_M(\tilde{u}\circ \Pi_{\K})-(\Lcal_{\K}\tilde{u})\circ \Pi_{\K})=\widecheck{\psi}\Lcal_M((\tilde{u}\circ \Pi_{\K})-\tilde{E}\circ \Pi_{\K}),\label{eq:EII}\\
		E_{1,III}&:=\Lcal_M((\psi'u')\circ \Pi_{\Spheqtres})-(\Lcal_{\Spheqtres}(\psi'u'))\circ\Pi_{\Spheqtres},\label{eq:EIII}
	\end{align}
	on $\widecheck{\Omega}[M]$, $\Kcech[M]$, $\tilde{S'}$ respectively and to vanish elsewhere.
	
	\emph{Step 4:} Using \ref{lem:equivnorm} and \eqref{eq:uprimeest}, \eqref{eq:utildeest},  we conclude that
	\begin{equation*}
		\norm{u_1}_{2,\beta,\gamma,\gamma';\K}\lesssim_b  \norm{E}_{0,\beta,\gamma-2,\gamma';M}.
	\end{equation*} 
	
	By the definitions of norms \ref{def:norm}, the equivalence of norms \ref{lem:equivnorm}\ref{item:equivnormK},  \eqref{eq:utilde} and the definition \ref{def:cutoff}, 
	\begin{equation*}
		\norm{\tilde{u}\circ\Pi_{\K}}_{2,\beta,\gamma,\gamma';\widecheck{\Omega}[M]}\lesssim	\norm{\tilde{u}}_{2,\beta,\gamma,\gamma';\Pi_{\K}(\widecheck{\Omega}[M])}\lesssim (\delta')^{-\gamma}	\norm{\tilde{u}_{low}}_{2,\beta,0;\K_M}+(\delta')^{-\gamma}	\norm{\tilde{u}_{high}}_{2,\beta,0;\K_M}.
	\end{equation*}
	Then by the definition of $E_{1,I}$ in \eqref{eq:EI}, the definition of norms \ref{def:norm} and the estimates \eqref{eq:ulow}, \eqref{eq:uhigh}
	\begin{align*}
		&\norm{E_{1,I}}_{0,\beta,\gamma-2,\gamma';M}\lesssim\norm{\tilde{u}\circ\Pi_{\K}}_{2,\beta,\gamma,\gamma';\widecheck{\Omega}[M]}\lesssim (\delta')^{-\gamma}\norm{\tilde{u}_{low}}_{2,\beta,0;\K_M}+ (\delta')^{-\gamma}	\norm{\tilde{u}_{high}}_{2,\beta,0;\K_M}\\
		\lesssim_b& \left(\frac{\sqrt{\tau}}{\delta'}\right)^{\gamma}\norm{E}_{2,\beta,\gamma-2,\gamma';M}\lesssim_{b} m^{-(1-\alpha)\gamma} \norm{E}_{2,\beta,\gamma-2,\gamma';M},
	\end{align*}
	where in the last inequality we use \eqref{eq:deltaprime} to compare $\tau$ and $\delta'$.
	
	Applying now \ref{lem:equivop}\ref{item:equivopK} with $u=\tilde{u}$ and $\hat{\gamma}=\gamma$ and using the definition of $\widecheck{\psi}$ in \ref{def:cutoff} and the definition of $E_{1,II}$ in \eqref{eq:EII}, by \eqref{eq:utildeest} we conclude that
	\begin{equation*}
		\norm{E_{1,II}}_{0,\beta,\gamma-2,\gamma';M}\lesssim \delta'\norm{\tilde{u}}_{2,\beta,\gamma,\gamma';\K_M}\lesssim_b \delta'\norm{E}_{0,\beta,\gamma-2,\gamma';M}\lesssim_b m^{-1-\alpha}\norm{E}_{0,\beta,\gamma-2,\gamma';M},
	\end{equation*}
	where in the last inequality we again use \eqref{eq:deltaprime}.
	
	Applying \ref{lem:equivop}\ref{item:equivopS} with $u=u'$ and $\hat{\gamma}=\gamma$, and using the definition of $E_{1,III}$ in \eqref{eq:EIII}, by \eqref{eq:uprimeest} we have
	\begin{equation*}
		\norm{E_{1,III}}_{0,\beta,\gamma-2,\gamma';M}\lesssim b^{-1}\norm{u'}_{2,\beta,\gamma,\gamma';\tilde{S'}}\lesssim b^{-1}\norm{E}_{0,\beta,\gamma-2,\gamma';M}.
	\end{equation*}
	
	Combining the above we conclude that by fixing $b$ big enough and then taking $m$ big in terms of $b$,
	\begin{equation*}
		\norm{E_{1}}_{0,\beta,\gamma-2,\gamma';M}\lesssim(C(b)m^{-(1-\alpha)}+ b^{-1})\norm{E}_{0,\beta,\gamma-2,\gamma';M}\leq \frac{1}{2}\norm{E}_{0,\beta,\gamma-2,\gamma';M}.
	\end{equation*}
	
	\emph{Step 5:} By assuming $b$ large enough and $m$ big enough in terms of $b$ we conclude using and induction that
	\begin{equation*}
		\norm{u_n}_{2,\beta,\gamma,\gamma';M}\leq 2^{-n} \norm{E}_{0,\beta,\gamma;M}.
	\end{equation*}
	The proof is then completed by using the earlier estimates.
\end{proof}

	  \section{Main results}
\label{S:main} 

\begin{theorem} 
\label{thm}
\label{lem:ubrev}
If $m\in\N$ is large enough,  
then there is $(\breve{\zeta},\breve{\bar{\zeta}})\in \BPcal$ as in \ref{def:zeta}, 
$\breve{\tau}:=\tau\bbracket{\breve{\zeta};m}$ and $\breve{\varphi}:=\varphi\bbracket{\breve{\zeta};m}$ as in \ref{def:varphi} satisfying \ref{lem:varphiest}, 
and $\breve{\textup{ф}}\in C^{2,\beta}_{\sym}({M})$ 
where ${M}:=M[\breve{\zeta},\breve{\bar{\zeta}}]$ is as in \ref{def:initial}, 
such that in the notation of \ref{def:norm} and exponents as in \ref{def:sol} and \ref{conv:gamma} 
\begin{equation} 
\label{eq:phibrev} 
\norm{\breve{\textup{ф}}}_{2,\beta,\gamma,\gamma';M} \, \lesssim \,  m^{-3-2\alpha+(1+\alpha)\gamma} \, \leq \, m^{-3+\gamma},  
\end{equation} 
and 
$\breve{M} = \breve{M}_m := {M}_{\breve{\textup{ф}}}$ is a $\Grp_{\Sph^4}[m]$-invariant embedded closed minimal hypersurface in $\Sph^4$ 
of topology $\#_{m^2-1}\Sph^2\times \Sph^1$ (recall \ref{def:grp} and \ref{lem:nonlinear}). 
Moreover each $\breve{M}_m$ is a doubling over $\Sigma = \Spheqt$ in the sense of \ref{D:dou} and the following are satisfied. 
	\begin{enumerate}[(i)]
\item \label{thm:item1} 
	$\breve{\Sigma} := \Pi_{\Spheqtres}(\breve{M}) \subset \Sigma$ is closed with smooth boundary in $\Sigma$ 
(recall \ref{def:fermi}(\ref{item:proj})). 
\item \label{thm:item2} 
$\breve{M}=\Graph_{\breve{\Sigma}}^{\Sph^4}(\breve{u})\bigcup \Graph_{\breve{\Sigma}}^{\Sph^4}(-\breve{u})$ for some 
$\ubreve \in C^0(\Sigmabreve) \bigcap C^\infty(\Sigmabreve \setminus\partial\Sigmabreve)$.      
\item \label{thm:item3} 
$D^\Sigma_L(\radius - m^{-3-\gamma} )  \subset \Sigma\setminus\Sigmabreve \subset D^\Sigma_L(\radius + m^{-3-\gamma} ) $ 
where $\radius = \pi F_0T_3  e^{\zeta} / m^2 $ with $\zeta$ as above satisfying \ref{eq:zeta}, $F_0$ as in \ref{eq:flux}, and $T_3$ as in \ref{eq:T3}. 
		\item\label{item:ubrevout} 
$\norm{\breve{u}:C^{2,\beta}(\Spheqtres\setminus D^{\Spheqtres}_L(\delta'))}\lesssim m^{-2}$.  
		\item\label{item:ubrevin} 
For each $p\in L$, $\norm{\breve{u}-\varphi_{cat}\circ\dist_p^{\Spheqtres}}_{2,\beta,\gamma,\gamma'; D^{\Spheqtres}_p(\delta')\setminus D^{\Spheqtres}_p(\radius^{7/8})}\lesssim m^{-3-2\alpha+(1+\alpha)\gamma}$.
	\end{enumerate}

Furthermore as $m \to\infty$ the hypersurfaces $\breve{M}_m$ converge in the sense of varifolds to $2\Spheqtres$ 
and their volume satisfies the asymptotic formula 
\begin{equation} 
\label{E:vol} 
|\breve{M}_m| \, = \, 2|\Spheqtres| \, -   \,    {(4\pi /3)   \,  T_3 m^{-4}}   \,  + \, O\left(m^{-40/7}\right) 
\, < \, 2|\Spheqtres| \, 
.
\end{equation} 
	  \end{theorem}
	  \begin{proof}
	  		As in \cite[Lemma 5.5]{gLD}, there exists a family of diffeomorphisms $\mathcal{F}_{\zeta,\bar{\zeta}}:M[0]\to M[\zeta,\bar{\zeta}]$ 
continuously depending on $\zeta,\bar{\zeta}$ such that for any $u\in C^{k,\beta}(M[\zeta,\bar{\zeta}])$, we have
	  		\begin{equation}\label{eq:Fzeta}
	  			\norm{u\circ \mathcal{F}_{\zeta,\bar{\zeta}}}_{k,\beta,\gamma,\gamma';M[0]}\sim_k 	\norm{u}_{k,\beta,\gamma,\gamma';M[\zeta,\bar{\zeta}]}.
	  		\end{equation}
	  		
	  		 We first define $B\subset C^{2,\beta}(M[0])\times \Pcal$ by (recall \ref{eq:zeta})
	  	\begin{equation}\label{eq:B}
	  		B:=\{v\in C^{2,\beta}(M[0]):\norm{v}_{2,\beta,\gamma,\gamma';M[0]}\leq m^{-3-2\alpha+\gamma(1+\alpha)}\}\times\BPcal.
	  	\end{equation}
	  	We next define a map $\Jcal:B\to C^{2,\beta}(M[0])\times \Pcal$ as follows.
	  	Let $(v,\zeta,\bar{\zeta})\in B$, define $(u,w_H,\bar{w}_H):=-\mathcal{R}_{M[\zeta,\bar{\zeta}]}(H-(w+\bar{w})\circ\Pi_{\Spheqtres})$ by \ref{prop:lineareq} (recall \eqref{eq:w}). Define $\textup{ф}:=v\circ \mathcal{F}_{\zeta,\bar{\zeta}}^{-1}+u$. Then by \ref{lem:H}, \ref{prop:lineareq}, \eqref{eq:Fzeta} and the size of $v$ in \eqref{eq:B}
	  	\begin{equation}\label{eq2}
	  		\norm{\textup{ф}}_{2,\beta,\gamma,\gamma';M[\zeta,\bar{\zeta}]} \lesssim  m^{-3-2\alpha+(1+\alpha)\gamma}, 
\quad \abs{\mu_H}\lesssim  m^{1-2\alpha+\alpha\gamma},\quad  \abs{\bar{\mu}_H}\lesssim  m^{2-2\alpha+\alpha\gamma},
	  	\end{equation}
	  	where $\mu_H, \bar{\mu}_H$ are defined such that $-\tau\mu_HW=w_H$, $-\tau\bar{\mu}_HW=\bar{w}_H$.
	  	By \ref{prop:lineareq} again we define $(u_Q,w_Q,\bar{w}_Q):=-\mathcal{R}_{M[\zeta,\bar{\zeta}]}(H_{\textup{ф}}-H-\mathcal{L}_M\textup{ф})$. Then by \ref{lem:nonlinear} and \ref{prop:lineareq}
	  	\begin{equation}\label{eq4g}
	  	\norm{u_Q}_{2,\beta,\gamma,\gamma';M[\zeta]}\lesssim m^{-6-4\alpha+2(1+\alpha)\gamma},\quad \abs{\mu_Q}\lesssim  m^{-2-4\alpha+(1+2\alpha)\gamma}, \quad \abs{\bar{\mu}_Q}\lesssim  m^{-1-4\alpha+(1+2\alpha)\gamma}
	  	\end{equation}
	  		where $\mu_Q, \bar{\mu}_Q$ are defined such that $-\tau\mu_QW=w_Q$, $-\tau\bar{\mu}_QW=\bar{w}_Q$.
	  	Combining the definitions we have
	  	\begin{equation}\label{eq5g}
	  		\Lcal_M(u_Q-v\circ\Fcal_{\zeta,\bar{\zeta}}^{-1})+H_{\textup{ф}}=-\tau( \mu_{sum} W+ \bar{\mu}_{sum} \bar{W})\circ\Pi_{\Spheqtres},
	  	\end{equation}
	  	where $\mu_{sum}:=\mu+\mu_H+\mu_Q$, $\bar{\mu}_{sum}:=\bar{\mu}+\bar{\mu}_H+\bar{\mu}_Q$. 
	  	
	  	Finally, we define $\Jcal$ by
	  	\begin{align*}
	  		\mathcal{J}(v,\zeta,\bar{\zeta}):=(u_Q\circ\Fcal_{\zeta,\bar{\zeta}},\zeta-(\mu_{sum}+\bar{\mu}_{sum})/\phi_1,\bar{\zeta}+\bar{\mu}_{sum}/\phi_1).
	  	\end{align*}
	  	
	  	$B$ is convex and the embedding $B\hookrightarrow C^{2,\beta'}(M[0])\times \Pcal$ is compact for $\beta'\in(0,\beta)$. By \eqref{eq4g} and \eqref{eq:Fzeta}, $\mathcal{J}$ maps the first factor into $B$ itself. By \eqref{eq:alphaprime}, \eqref{eq:zeta}, \eqref{eq:zetamu}, \eqref{eq2} and \eqref{eq4g} $\mathcal{J}$ maps the second factor into $B$ itself by choosing $m$ big enough. It is easy to check that $\mathcal{J}$ is a continuous map in the induced topology. By Schauder’s fixed point theorem \cite[Theorem 11.1]{gilbarg} then, there is a fixed point $(\breve{v}, \breve{\zeta},\breve{\bar{\zeta}})$ of $\mathcal{J}$, which therefore satisfies $\breve{v}=\breve{u}_Q\circ \mathcal{F}_{\breve{\zeta},\breve{\bar{\zeta}}}$ and $\breve{w} + \breve{w}_H + \breve{w}_Q = 0$, $\breve{\bar{w}} + \breve{\bar{w}}_H + \breve{\bar{w}}_Q = 0$, where we use “$\breve{\cdot}$” to denote the various quantities for $\zeta = \breve{\zeta}$, $\bar{\zeta} = \breve{\bar{\zeta}}$ and $v = \breve{v}$. By \eqref{eq5g} then we conclude the minimality of $\breve{M}_{\breve{\textup{ф}}}$. The smoothness follows from standard regularity theory and the embeddedness follows from \ref{lem:phigl}\ref{item:phiglembed} and \eqref{eq2}. The topology then follows because we are connecting two three-spheres with $m^2$ bridges and the symmetry follows by construction. Finally the limit of the volume as $m \to \infty$ follows from the available estimates for $\varphi_{gl}\bbracket{\breve{\zeta},\breve{\bar{\zeta}}}$ in \ref{lem:phigl}\ref{item:phiglbound} and the bound on the norm of $\breve{\textup{ф}}$ in \eqref{eq2}.

(\ref{thm:item1}) and (\ref{thm:item2}) follow easily from the above.  
(\ref{thm:item3}) follows from \ref{eq:phibrev}, \ref{eq:tau} and $\radius=\sqrt{\tau}$ in \ref{def:initial}. 
(\ref{item:ubrevout}) follows by the estimate of $\varphi_{gl}$ in \ref{lem:phigl}\ref{item:phiglbound}, 
the estimate of $\breve{\textup{ф}}$ in \eqref{eq:phibrev} and \cite[Lemma B.9]{gLD}.  
(\ref{item:ubrevin}) also follows by \cite[Lemma B.9]{gLD} and \eqref{eq:phibrev} with the definition of the norms in \ref{def:norm} 
and the definition of $\varphi_{gl}=\varphi_{cat}\circ\dist_p^{\Spheqtres}$ in \ref{def:initial}.
To prove now \eqref{E:vol} we have the following.

\begin{lemma}[Volume on hypersurface bridges] 
\label{L:Mcatvol}
	For any $p\in L$ and $\bar{r}\in (\radius,\delta')$ we have
	\begin{multline*}
		|\breve{M}\cap \Pi_{\Spheqtres}^{-1} (D^{\Spheqtres}_p(\bar{r}))| \, = \, 
2|D^{\Spheqtres}_p(\bar{r})| 
\\ 
\, - \, T_3\frac{4\pi\radius^3}{3} \, + \, \oint_{\partial D^{\Spheqtres}_p(\bar{r})} \varphi_{cat }\circ\dist_p^{\Spheqtres} 
\frac{\partial \varphi_{cat}\circ\dist_p^{\Spheqtres} }{\partial \eta} \, + \, O\left(\frac{\radius^8}{\bar{r}^5}+\bar{r}^5\right).
	\end{multline*}
\end{lemma}	
\begin{proof}
	There is a domain $\Kcech_{\bar{r}}\subset \Kcech[p,\radius]\subset M[\breve{\zeta},\breve{\bar{\zeta}}]$ defined by requesting that 
$\breve{M}\cap \Pi_{\Spheqtres}^{-1}(D^{\Spheqtres}_p(\bar{r}))=\Graph_{\Kcech_{\bar{r}}}^{\Sph^4}(\breve{\textup{ф}})$. 
Using \cite[Lemma A.8]{gLD}, it follows that
	\begin{multline*}
		|\breve{M}\cap \Pi_{\Spheqtres}^{-1} (D^{\Spheqtres}_p(\bar{r}))| \, = \, |\Kcech_{\bar{r}}| 
\\ 
+ \, \frac{1}{2}\int_{\Kcech_{\bar{r}}}\abs{\nabla\breve{\textup{ф}}}^2 \, - \, 2\breve{\textup{ф}}H \, - \, \breve{\textup{ф}}^2\left(\abs{A}^2+3 \, - \, H^2\right) 
\, + \, O(m^{-9+3\gamma}\bar{r}^{3+3\gamma}),
	\end{multline*}
	where $A$ and $H$ are the second fundamental form and mean curvature of $\Kcech$ and we have used that $\abs{\breve{\textup{ф}}}\leq m^{-3-2\alpha+(1+\alpha)\gamma}\bar{r}^{\gamma}$ from \ref{thm} to estimate the error term. From this last estimate, the estimate for $H$ in \ref{lem:H}, the support of $w$ and $\bar{w}$ in \ref{lem:obstr}\ref{item:wsupp}, the form of $A$ in \ref{eq:AKcech}	and the definition of the norm in \ref{def:norm}, it follows that (recall that $\gamma\in (0,1)$ and $\bar{r}\geq\radius\sim m^{-2}$)
	\begin{align*}
\int_{\Kcech_{\bar{r}}}\abs{\nabla\breve{\textup{ф}}}^2+\abs{\breve{\textup{ф}}H}\lesssim m^{-6+2\gamma}\bar{r}^{2\gamma+1}\lesssim \bar{r}^5, 
\qquad 
\int_{\Kcech_{\bar{r}}} \breve{\textup{ф}}^2\abs{A}^2 \lesssim m^{-14+2\gamma}\bar{r}^{2\gamma-3}\lesssim \bar{r}^5.
	\end{align*}
	The conclusion now follows from these estimates, \ref{lem:voltrun} and the closeness of $\Kcech(\bar{r})$ to $\Kcech_{\bar{r}}$.
\end{proof}

\begin{lemma}[Volume on the complement of the bridges] 
\label{lem:volsurf}
	The following estimate holds for any $\bar{r}\in (\radius,\delta')$.
	\begin{multline*}
		|\breve{M}\cap \Pi_{\Spheqtres}^{-1} (\Spheqtres\setminus D^{\Spheqtres}_L(\bar{r}))| 
\, = \, 2|\Spheqtres| \, - \, 2\sum_{p\in L}|D^{\Spheqtres}_p(\bar{r})| \, 
\\
- \, \sum_{p\in L}\oint_{\partial D^{\Spheqtres}_p(\bar{r})} \varphi_{cat }\circ\dist_p^{\Spheqtres} \frac{\partial \varphi_{cat}\circ\dist_p^{\Spheqtres} }{\partial \eta}
\, + \, O\left(\frac{\radius^6}{\bar{r}^4}+\radius^3+\radius^{\frac{3-\gamma}{2}}\bar{r}^{1+\gamma}\right).
	\end{multline*}
\end{lemma}
\begin{proof}
	From \cite[Lemma A.10]{gLD}, the minimality of $\Spheqtres$ and \ref{lem:ubrev}\ref{item:ubrevout} for the estimates the error terms, we have
\begin{multline} 
\label{eq:volex}
|\breve{M}\cap \Pi_{\Spheqtres}^{-1} (\Spheqtres\setminus D^{\Spheqtres}_L(\delta'))| \, = \, 2|\Spheqtres\setminus D^{\Spheqtres}_p(\delta')| 
\\ 
- \, \int_{\Spheqtres\setminus D^{\Spheqtres}_L(\delta')}\breve{u}\Lcal_{\Spheqtres}\breve{u} \, 
- \, \sum_{p\in L}\oint_{\partial D^{\Spheqtres}_p(\delta')} \breve{u} \frac{\partial \breve{u} }{\partial \eta}+O\left(m^{-6}\right).
\end{multline}
Similarly, from \cite[Lemma A.10]{gLD}, the minimality of $\Spheqtres$, \ref{lem:ubrev}\ref{item:ubrevin} and \eqref{eq:varphicat} for the estimates the error terms, we have

	  	 \begin{multline} 
\label{eq:volin} 
|\breve{M}\cap \Pi_{\Spheqtres}^{-1} ( D^{\Spheqtres}_L(\delta')\setminus D^{\Spheqtres}_L(\bar{r}))| \, = \, 
2|D^{\Spheqtres}_L(\delta')\setminus D^{\Spheqtres}_L(\bar{r}))| 
\\
- \, \int_{D^{\Spheqtres}_L(\delta')\setminus D^{\Spheqtres}_L(\bar{r})}\breve{u}\Lcal_{\Spheqtres}\breve{u}          
		  \, - \, \sum_{p\in L}\oint_{\partial D^{\Spheqtres}_p(\bar{r})} \breve{u} \frac{\partial \breve{u} }{\partial \eta} \, + \, 
\sum_{p\in L}\oint_{\partial D^{\Spheqtres}_p(\delta')} \breve{u} \frac{\partial \breve{u} }{\partial \eta} \, + \, O\left(\frac{\radius^{4}}{\bar{r}}\right).
\end{multline}
	
From the minimality of $\breve{M}$ and \eqref{eq:SphereH}, it follows that on $\Spheqtres\setminus D^{\Spheqtres}_L(\bar{r})$
	\begin{equation*}
		\abs{\breve{u}\Lcal_{\Spheqtres}\breve{u}}=\abs{\breve{u}Q_{\breve{u}}}\lesssim \abs{\nabla^2\breve{u}}\abs{\nabla\breve{u}}^2\abs{\breve{u}}.
	\end{equation*}
	\ref{lem:ubrev} then gives,
	\begin{equation} 
\label{eq:vollinear}
		\abs{\int_{\Spheqtres\setminus D^{\Spheqtres}_L(\delta')}\breve{u}\Lcal_{\Spheqtres}\breve{u}}\lesssim m^{-8},  
\qquad 
\abs{\int_{D^{\Spheqtres}_L(\delta')\setminus D^{\Spheqtres}_L(\bar{r})}\breve{u}\Lcal_{\Spheqtres}\breve{u}} \lesssim \frac{\radius^6}{\bar{r}^4}.
	\end{equation}
	Next, using \ref{lem:ubrev}\ref{item:ubrevin} and \eqref{eq:varphicat}, it follows for any $p\in L$ that
	\begin{equation}\label{eq:bderror}
		\abs{\oint_{\partial D^{\Spheqtres}_p(\bar{r})} \breve{u} \frac{\partial \breve{u} }{\partial \eta}-\oint_{\partial D^{\Spheqtres}_p(\bar{r})} \varphi_{cat}\circ\dist_p^{\Spheqtres} \frac{\partial \varphi_{cat}\circ\dist_p^{\Spheqtres} }{\partial \eta}}\lesssim m^{-5+\gamma}\bar{r}^{1+\gamma}.
	\end{equation}
	
	By combining \eqref{eq:volex}, \eqref{eq:volin}, \eqref{eq:vollinear}, \eqref{eq:bderror} and recalling $\bar{r}\geq\radius\sim m^{-2}$, the result follows.
\end{proof}

\eqref{E:vol} follows now by combining \ref{lem:voltrun} and \ref{lem:volsurf} and by choosing $\bar{r}\sim \radius^{11/14}\sim m^{-11/7}$ to minimize the error. 
This completes the proof of the Theorem. 
\end{proof}

	    	 \section{Doubling the Spherical Shrinker $\Sphst$ in $\R^4$}   
\label{S:shr} 
	    \subsection*{Spherical Shrinker $\Sphst$ in $\R^{n+1}$}
	    \begin{definition}
	    	We define the metric $\gshr  $ on $\R^{n+1}$ by
	    	\begin{equation*}
	    		\gshr  :=e^{-\frac{\abs{x}^2}{2n}}g_{\Euc}
	    	\end{equation*}
	    	
	    	We call the minimal hypersurfaces in $(\R^{n+1},\gshr  )$ self-shrinkers.
	    \end{definition}
	    
	    By \cite[(11.17)]{kapouleas:kleene:moller} $\Sphst :=\Sph^n(\sqrt{2n})$ is totally geodesic in $(\R^{n+1},\gshr  )$, and in particular, a self-shrinker.
	    \begin{lemma}\label{lem:jacobishr}
	    	The Jacobi operator of $\Sph^n_{\shr}$ in $(\R^{n+1},\gshr  )$ is
	    	\begin{equation}\label{eq:jacobishr}
	    		\Lcal_{\Sph^n_{\shr}}:=e(\Delta_{\Sph^n_{\shr}}+1)=\frac{e}{2n}(\Delta_{\Sphsn}+2n)
	    	\end{equation}
	    \end{lemma}
	    \begin{proof}
	    	In this proof, we denote $h_0:=g_{\Euc}$, $g_0=h_0|_{\Sph^n(\sqrt{2n})}=g_{\Sph^n(\sqrt{2n})}$, $\omega:=-\frac{\abs{x}^2}{4n}$, $g=e^{2\omega}g_{\Euc}|_{\Sph^n(\sqrt{2n})}$. Then by definition the Jacobi operator of $\Sph^n_{\shr}$ in $(\R^{n+1},e^{-\frac{\abs{x}^2}{2n}}\delta_{ij})$ is
	    	\begin{equation*}
	    		\Lcal_{\Sph^n_{\shr}}=\Delta_g+\abs{A_h}_g^2+\Ric^h(\nu^h,\nu^h).
	    	\end{equation*}
	    	
	    	As mentioned, $A_h=0$, and $\Delta_g=e^{-2\omega}\Delta_{g_0}$. On the other hand, like in the proof of \cite[Lemma 11.3]{kapouleas:kleene:moller}, the Ricci curvature changes under conformal change by
	    	\begin{equation*}
	    		\Ric^h=\Ric^{h_0}+(n-1)(-\Hess^{h_0}\omega+d\omega\otimes d\omega)-(\Delta_{h_0}\omega+(n-1)\abs{\nabla\omega}_{h_0}^2)h_0.
	    	\end{equation*}
	    	As $\omega=-\frac{\abs{x}^2}{4n}$,
	    	\begin{align*}
	    		\nabla^{h_0}\omega=-\frac{x}{2n},\quad \abs{\nabla\omega}_{h_0}^2=\frac{\abs{x}^2}{4n^2},\quad \Hess^{h_0}\omega=-\frac{1}{2n}h_0,\quad \Delta_{h_0}\omega=-\frac{n+1}{2n}.
	    	\end{align*}
	    	Along with $\Ric^{h_0}=0$,
	    	\begin{align*}
	    		\Ric^h&=(n-1)\left(\frac{1}{2n}h_0+\frac{(n-1)x\otimes x}{4n^2}\right)-\left(-\frac{n+1}{2n}+(n-1)\frac{\abs{x}^2}{4n^2}\right)h_0\\
	    		&=\left(1-\frac{(n-1)\abs{x}^2}{4n^2}\right)h_0+\frac{(n-1)x\otimes x}{4n^2}.
	    	\end{align*}
	    	Thus 
	    	\begin{align*}
	    		\Ric^h(\nu^h,\nu^h)=e^{-2\omega}\left(1-\frac{(n-1)\abs{x}^2}{4n^2}+\frac{(n-1)(x\cdot\nu)^2}{4n^2}\right).
	    	\end{align*}
	    	And on $\Sph^n(\sqrt{2n})$, 
	    	\begin{align*}
	    		\Ric^h(\nu^h,\nu^h)=e^{-2\omega}\left(1-\frac{(n-1)2n}{4n^2}+\frac{(n-1)2n}{4n^2}\right)=e^{-2\omega}.
	    	\end{align*}
	    	The result then follows.
	    \end{proof}
	    \begin{remark}
	    	Since the eigenvalues of $\Delta_{\Sph^n}$ on the unit sphere $\Sph^n$ are
	    	\begin{equation*}
	    		\lambda^n_k=k(k+n-1), 
	    	\end{equation*}
	    	$\Lcal_{\Sph^n_{\shr}}$ has trivial kernel.
	    \end{remark}
	    
\subsection*{LD solutions on Spherical shrinker $\Sphshrtres$ in $\R^4$}
$\vspace{.162cm}$ 
\nopagebreak

	    Now we focus on the case $n=3$. We define the surjective map $\Theta_{\shr}:\R^2\times[-\frac{\pi}{4},\frac{\pi}{4}]\to\Sphshrtres=\Sph^3(\sqrt{6})$ by (recall \eqref{eq:thetaeqmetric}) $\Theta_{\shr}(\xx,\yy,\zz):=\sqrt{6}\Thetaeq(\xx,\yy,\zz)$. We still use $\T_0$ to denote the Clifford torus in $\Sphshrtres=\Sph^3(\sqrt{6})$, i.e. the intersection of $\Sph^3(\sqrt{6})$ with the cone $x_1^2+x_2^2=x_3^2+x_4^2$.

	    \begin{definition}
	    	By considering $\Sph^3_{\shr}$ as a subset of $\R^4$, we denote by $\Grp_{\R^4}[m]$ the subgroup of $\mathfrak{O}(4)$ which fixes $\Lmer[m]\subset\Spheqtres\subset\R^4$ as a set (recall \eqref{eq:Lmer}).
	    \end{definition}
	    
	    \begin{remark}
	    	$\Grp_{\R^4}[m]$ is isomorphic to $\Grp_{\Spheqtres}[m]$ (recall \ref{def:grp}).
	    \end{remark}
	    
	    As \eqref{eq:jacobirotinvar}, when the solution $\phi$ is rotationally invariant, the linearized equation $\Lcal_{\Sph^3_{\shr}}\phi=0$ from \eqref{eq:jacobishr} amounts to the ODE:
	    \begin{equation}\label{eq:jacobirotinvarshr}
	    	\frac{d^2\phi}{d\zz^2}-2\tan{2\zz}\frac{d\phi}{d\zz}+6\phi=0.
	    \end{equation}
\begin{lemma}[cf. {\ref{lem:phiC}}] 
\label{lem:phiCshr}
	    	The space of solutions of the ODE \eqref{eq:jacobirotinvarshr} in $\zz$ on $(-\pi/4,\pi/4)$ is spanned by the two functions
\begin{equation} 
\begin{aligned}
	    		\phi_C:= \frac{1}{C_0} {}_2F_1\left(\frac{1-\sqrt{7}}{2},\frac{1+\sqrt{7}}{2};1;\frac{1}{2}(\cos\zz-\sin\zz)^2\right),\\
	    		\phi_{C^{\perp}}:= \frac{1}{C_0} {}_2F_1\left(\frac{1-\sqrt{7}}{2},\frac{1+\sqrt{7}}{2};1;\frac{1}{2}(\cos\zz+\sin\zz)^2\right),           
\end{aligned}
\end{equation} 
	    	where $C_0={}_2F_1\left(-\frac{1-\sqrt{7}}{2},\frac{1+\sqrt{7}}{2};1;\frac{1}{2}\right)>0$. Moreover, the following hold.
	    	\begin{enumerate}[(i)]
	    		\item $\phi_C$ has singularity at $\uC=\{z=-\pi/4\}$ and is smooth at $\uC^{\perp}=\{z=\pi/4\}$, $\phi_{\uC^{\perp}}$ has singularity at $\uC^{\perp}=\{z=\pi/4\}$ and is smooth at $\uC=\{z=-\pi/4\}$.
	    		\item $\phi_{\uC}$ is strictly increasing in $\zz$ on $(-\pi/4,\pi/4)$, $\phi_{\uC^{\perp}}$ is strictly decreasing in $\zz$ on $(-\pi/4,\pi/4)$.
	    		\item\label{item:Fshr} $\phi_{\uC}(0)=\phi_{\uC^{\perp}}(0)=1$, $\phi'_{\uC}(0)=-\phi'_{\uC^{\perp}}(0)=F_{0shr}>0$, moreover $\phi_{\uC^{\perp}}=\phi_{\uC}\circ\SSS$.
	    	\end{enumerate}
	    \end{lemma}
	    
	    \begin{definition}\label{def:varphishr}
	    	We define the space of parameters $\Pcal:=\R^3$. The continuous parameters of the LD solutions are $\zetabold:=(\zeta,\bar{\zeta}^+,\bar{\zeta}^-)\in\BPcal$, where 
	    	\begin{equation}\label{eq:zetashr}
	    		\BPcal:=\{\zetabold=(\zeta,\bar{\zeta}^+,\bar{\zeta}^-)\in\Pcal:\abs{\zeta}\leq 1/m^{\alpha'}, \abs{\bar{\zeta}^-}\leq 1/m^{\alpha'},\abs{\bar{\zeta}^-}\leq 1/m^{\alpha'}\},
	    	\end{equation}
	    	where $\alpha'\in(0,1)$ is a small constant as in \ref{def:zeta}.
	    	
	    	Given $\zetabold=(\zeta,\bar{\zeta}^+,\bar{\zeta}^-)\in \BPcal$, we define the family of LD solutions $\varphi\bbracket{\zeta;m}$, $\underline{v}$ and $\underline{v}^{\pm}:=\underline{v}[\bar{\zeta}^{\pm}]$ as in \ref{def:varphi} by using $F_{0shr}$ in \ref{lem:phiCshr}\ref{item:Fshr} instead of $F_0$. We also define correspondingly $\mu$ and $\bar{\mu}^{\pm}$ as in \ref{lem:miss}.
	    \end{definition}
	    
	    \subsection*{Catenoidal bridges}
	    
\begin{definition}\label{def:fermiconf}
Let $g$ denote $\gshr  $.
Given $p\in\Sph^3_{\shr}$, we first define the function $\varsigma$ on $\R^4$ by requesting $\varsigma=(\exp_p^{\Sph^3_{\shr},\R^4,g})^{-1}(\zzcheck)$ (recall \ref{def:fermi}). 
We then extend $\varsigma$ to cylindrical coordinates $(\varrho, \vartheta, \varsigma):\R^4\setminus \mathrm{span}\{\nu_{\Sph^3_{\shr}}(p)\}\to \R_+\times\Sph^2\times\R$. 
We also define $U:=D^{\Sph^3_{\shr},\R^4}_p(\inj^{\Sph^3_{\shr},\R^4}_p/2)\subset \R^4$ and a Riemannian metric $\mathring{g}$ and a tensor $h$ on $U$ by
\begin{equation}\label{eq:mathringgshr}
\mathring{g}:=(\exp_p^{\Sph^3_{\shr},\R^4})_*g|_p, 
\qquad 
h:=g-\mathring{g}.
\end{equation}
\end{definition}

\begin{remark} 
\label{rk:fermishr}
	    	We have for $p=(0,0,0,\sqrt{6})\in \Sph^3_{\shr}\subset \R^4$
\begin{equation}\label{eq:fermiexpshr}
\exp_p^{\Sph^3_{\shr},\R^4}(\varrho, \vartheta, \varsigma)=r(\varsigma)\left(\vartheta\sin\left(\sqrt{\frac{e}{6}}\varrho\right),\cos\left(\sqrt{\frac{e}{6}}\varrho\right)\right),
\end{equation}
where the function $r(\varsigma)$ satisfies the ODE 
\begin{equation}\label{eq:fermir}
\frac{d r(\varsigma)}{d\varsigma}=e^{\frac{r^2(\varsigma)}{12}},\quad r(0)=\sqrt{6},
\end{equation}
and thus
\begin{equation} 
\label{eq:fermimetricshr}
(\exp_p^{\Sph^3_{\shr},\R^4})^{*}g=e^{-\frac{r^2(\varsigma)}{6}}r(\varsigma)^2\left(\frac{e}{6}d\varrho^2+\sin^2\left(\sqrt{\frac{e}{6}}\varrho\right) g_{\Sph^2}\right)+d\varsigma^2.
\end{equation}
\end{remark}
	    
	    \begin{definition}
	    	Given $p \in \Sph^3_{\shr}$ and $\radius \in \R_+$, $\kappa\in [-\radius,
	    	,\radius]$, we define a three-catenoid $\K[p, \radius, \kappa] \subset T_p\R_4 \cong \R^4$ and its parametrization $X_\K=X_\K[p, \radius,\kappa] : \Cylinder \to \K[p,\radius,\kappa]$ by (recall \eqref{eq:XK})
	    	\begin{align*}
	    		&X_\K[p, \radius,\kappa]:=X_\K[p, \radius]+\kappa\nu_{\Sph^3_{\shr}}(p),\quad \K[p,\radius,\kappa]:=X_\K[p, \radius,\kappa](\Cylinder).
	    	\end{align*}
	    	We then define a corresponding \emph{catenoidal bridge} $\Kcech[p,\radius,\kappa]$ in $\R^4$ and its \emph{core} in $\R^4$ as in \ref{def:catebridge}. We also define $\varphi_{cat}^{\pm}[\radius,\kappa]:[\radius,\infty)\to\R$ by (recall \eqref{eq:varphicat})
	    	\begin{equation*}
	    		\varphi_{cat}^{\pm}[\radius,\kappa]:=\varphi_{cat}[\radius]\pm\kappa.
	    	\end{equation*}
	    \end{definition}
	    
	    \begin{lemma}
\label{L:alpha:shr} 
	    	The following hold, where $R(s):=r(\sigma(s)+\kappa)$ (recall \eqref{eq:fermir}),  $\rho(s)$, $\sigma(s)$ as in \eqref{eq:rhot}, and as in \cite[Lemma C.4]{gLD}
	    	the symmetric two-tensor fields $\alpha$ and $\tilde{\alpha}$, differential one-form $\beta$, and
	    	function $\digamma$, are defined by requesting that for $q \in \Kcech$, $X, Y \in T_q\Kcech$ (recall \eqref{eq:mathringgshr})
	    	\begin{align*}
	    		\alpha(X,Y):=h(X,Y),\quad \beta(X):=h(X,\mathring{\nu}_q),\quad \digamma(q):=h(\mathring{\nu}_q,\mathring{\nu}_q),\quad \tilde{\alpha}(X,Y):=\nabla_{\mathring{\nu}_q}h(X,Y),
	    	\end{align*}
	    	where $\mathring{\nu}$ is the normal vector field on $\Kcech$ induced by $\mathring{g}$.
	    	\begin{enumerate}[(i)]
	    		\item $h=\left(e^{1-\frac{r^2(\varsigma)}{6}}\frac{r^2(\varsigma)}{6}-1\right)d\varrho^2+\left(e^{-\frac{r^2(\varsigma)}{6}}r^2(\varsigma)\sin^2\left(\sqrt{\frac{e}{6}}\varrho\right)-\varrho^2 \right) g_{\Sph^2}$.
\item 
$\alpha=\left(e^{1-\frac{R^2}{6}}\frac{R^2}{6}-1\right)\rho^2\tanh^2(2s)ds^2+\left(e^{-\frac{R^2}{6}}R^2\sin^2\left(\sqrt{\frac{e}{6}}\rho\right)-\rho^2 \right) g_{\Sph^2}$ 
\\ $\phantom{kk}$ \hfill and \qquad $\norm{\alpha :C^k(\Kcech,\chi,\rho^4)}\lesssim_k 1$.
	    		\item $\beta=-\left(e^{1-\frac{R^2}{6}}\frac{R^2}{6}-1\right)\rho\tanh(2s)\sech(2s)ds$ and $\norm{\beta :C^k(\Kcech,\chi,\radius^2\rho)}\lesssim_k 1$.
	    		\item $\digamma=\left(e^{1-\frac{R^2}{6}}\frac{R^2}{6}-1\right)\sech^2(2s)$ and $\norm{\digamma :C^k(\Kcech,\chi,\radius^2)}\lesssim_k 1$.
	    		\item 
$\tilde{\alpha}=-\frac{R}{3}e^{1-\frac{R^2}{6}}
\left(\frac{R^2}{6}-1\right)\rho^2\tanh^3(2s)ds^2-2\Bigg(Re^{-\frac{R^2}{6}}\left(\frac{R^2}{6}-1\right)\sin^2\left(\sqrt{\frac{e}{6}}\rho\right)\tanh(2s) 
\\
\phantom{kk} \hfill 
+ \, \left(R^2e^{-\frac{R^2}{6}}\left(\frac{\sin^2\left(\sqrt{\frac{e}{6}}\rho\right)}{\rho} \, + \, 
\sqrt{\frac{e}{6}}\sin\left(\sqrt{\frac{e}{6}}\rho\right)\cos\left(\sqrt{\frac{e}{6}}\rho\right)\right)-2\rho \right)\sech(2s)\Bigg) g_{\Sph^2}$  
\\
$\phantom{kk}$ \hfill 
and \qquad $\norm{\tilde{\alpha} :C^k(\Kcech,\chi,\rho^3)}\lesssim_k 1$.
	    	\end{enumerate}	
	    \end{lemma}	

	    \begin{proof}
	    	The expressions follow by the definitions and a direct calculation by \eqref{eq:rhot} and \ref{rk:fermishr}. The estimates then follow by \ref{lem:rhozh}.
	    \end{proof}
	    
	    \begin{lemma} 
\label{lem:HKshr}
	    	Let $g_{\Kcech}$, $\nu_{\Kcech}$ and $A_{\Kcech}$ be the metric, the unit normal and the second fundamental form on $\Kcech$ induced by $g$, then
	    	\begin{enumerate}[(i)]
	    		\item $\norm{\rho^2(\Delta_{g_{\Kcech}}-\Delta_{\mathring{g}_{\Kcech}})u:C^{0,\beta}(\Kcech,\chi,\rho^2)}
	    		\lesssim  \norm{u:C^{2,\beta}(\Kcech,\chi)}$.
	    		\item $\norm{\rho^{2}(\abs{A_{\Kcech}}^2_{g_{\Kcech}}-\abs{\mathring{A}_{\Kcech}}^2_{\mathring{g}_{\Kcech}}):C^{0,\beta}(\Kcech,\chi,\rho)}\lesssim_k 1$.
	    		\item 	$\norm{\rho^2H:C^k(\Kcech,\chi,\rho^3)}\lesssim_k 1$, where $H$ is the mean curvature on $\Kcech$ induced by $g$.
	    	\end{enumerate} 	
	    \end{lemma}
	    \begin{proof}
	    	\cite[Lemma C.10(iv), Lemma C.11, Corollary C.9]{gLD} along with \ref{lem:rhozh}.
	    \end{proof}

	    \subsection*{Initial surfaces}
	    \begin{lemma}
	    	Given $\varphi=\varphi\bbracket{\zeta}$, $\underline{v}$, $\bar{\underline{v}}^{\pm}:=\bar{\underline{v}}[\bar{\zeta}^{\pm}]$ as in \ref{def:varphishr}, there exists $\varphi_{nl}^{\pm}=\varphi_{nl}\bbracket{\zeta,\bar{\zeta}^{\pm},\kappa}\in C^{\infty}_{\sym}(\Sphshrtres\setminus L)$ such that the following hold.
	    	\begin{enumerate}[(i)]
	    		\item $H_{\varphi_{nl}^{\pm}}=0$ on $\Sphshrtres\setminus  D_{\T_0}^{\Sphshrtres}(4/m)$
	    		and $H_{\varphi_{nl}^{\pm}}$ is the mean curvature of the graph
	    		$\Graph_{\Sphshrtres}^{\R^4,\gshr  }(\pm\varphi_{nl}^{\pm})$ pushed forward to $\Sphshrtres$ by the projection $\Pi_{\Sphshrtres}$ (recall \ref{not:manifold}\ref{item:proj} and \ref{not:manifold}\ref{item:graph}).
	    		\item $\varphi_{nl}^{\pm}-\varphi-\underline{v}-\bar{\underline{v}}^{\pm}\mp\kappa V$ can be extended to a smooth function supported on $\Sphshrtres\setminus  D_{\T_0}^{\Sphshrtres}(3/m)$ which satisfies
	    		\begin{equation*}
	    			\norm{\varphi_{nl}^{\pm}-\varphi-\underline{v}-\bar{\underline{v}}^{\pm}\mp\kappa V:C^{2,\beta}(\Sphshrtres,\gshr  |_{\Sphshrtres})}\lesssim m^{-5+4\beta}.
	    		\end{equation*}
	    	\end{enumerate}
	    \end{lemma}
	    \begin{proof}
	    	Under Euclidean metric $g_{\Euc}$ in $\R^4$ and an orthonormal basis on $\Sph^3(1)$ we calculate ($\varphi_i$ denotes $\partial_{i}\varphi$ and similarly for second order derivatives, recall \eqref{eq:fermir})
	    	\begin{align*}
	    		&X_{\varphi}=r(\varphi)X,\quad
	    		(g_{\varphi})_{ij}=r^2(\varphi)\delta_{ij}+e^{\frac{r^2(\varphi)}{6}}\varphi_i\varphi_j,\\
	    		&\nu_{\varphi}=X-e^{\frac{r^2(\varphi)}{12}}\sum_i \varphi_i \partial_i X/ r(\varphi)/\sqrt{1+e^{\frac{r^2(\varphi)}{6}}\abs{\nabla\varphi}^2/r^2(\varphi)},\\
	    		&(A_{\varphi})_{ij}=\left(-r(\varphi)\delta_{ij}+(\frac{r(\varphi)}{6}-\frac{2}{r(\varphi)})e^{\frac{r^2(\varphi)}{6}}\varphi_i\varphi_j+e^{\frac{r^2(\varphi)}{12}}\varphi_{ij}\right)/\sqrt{1+e^{\frac{r^2(\varphi)}{6}}\abs{\nabla\varphi}^2/r^2(\varphi)}.\\
	    	\end{align*}
	    	From the expression of $(g_{\varphi})_{ij}$ and Sherman-Morrison formula, we can calculate the inverse matrix
	    	\begin{equation*}
	    		(g_{\varphi})^{ij}=\delta_{ij}/r^2(\varphi)-\varphi_i\varphi_j/(r^4(\varphi) e^{-\frac{r^2(\varphi)}{6}}+r^2(\varphi)\abs{\nabla\varphi}^2).
	    	\end{equation*}
	    	
	    	Now we denote the mean curvature of the graph $\Graph_{\Sphshrtres}^{\R^4,g_{\Euc}}(\varphi)$ by $H_{\varphi}$ and similarly the mean curvature of the graph $\Graph_{\Sphshrtres}^{\R^4,\gshr  }(\varphi)$ by $\bar{H}_{\varphi}$, and also their push-forward to $\Sphshrtres$ by the projection $\Pi_{\Sphshrtres}$ (notice that it is the same under $g_{\Euc}$ and $\gshr  $). From the formula for mean curvature under conformal metric $\bar{H}_{\varphi}=e^{\frac{r^2}{12}}(H_{\varphi}-3\nu_{\varphi}\cdot\nabla^{\gshr  }(-\frac{r^2}{12}))$, thus
\begin{multline}   
\label{eq:SphereHshr}
\sqrt{1+e^{\frac{r^2(\varphi)}{6}}\abs{\nabla\varphi}^2/r^2(\varphi)}\bar{H}_{\varphi} \, = \, 
\sqrt{1+e^{\frac{r^2(\varphi)}{6}}\abs{\nabla\varphi}^2/r^2(\varphi)}(g_{\varphi})^{ij}(A_{\varphi})_{ij}  
\\  
= \, \frac{e^\frac{r^2}{6}}{r^2}\Delta \varphi \, + \, e^\frac{r^2}{6}\left(\frac{r}{2} \, - \, \frac{3}{r}\right) \, + \, \left(\frac{1}{6r} \, - \, \frac{2}{r^3} \, + \, \frac{1}{r^3 \, 
+ \, e^\frac{r^2}{6}r\abs{\nabla\varphi}^2}\right)e^\frac{r^2}{4}\abs{\nabla\varphi}^2 \, + \, O(\varphi^3).
\end{multline}   
	    	We then have the estimate as in \ref{lem:SphereH} by \eqref{eq:jacobishr} and \eqref{eq:fermir}. The results then follow as in \ref{lem:varphinl}.
	    \end{proof}
	    \begin{definition}[Families of initial surfaces]
	    	Given $\zetabold\in\BPcal$ as in \eqref{eq:zetashr}, $\kappa\in [-\radius,\radius]$
	    	we write
	    	\begin{equation*}
	    		\zetaboldu:=(\zetabold,\kappa)=(\zeta,\bar{\zeta}^+, \bar{\zeta}^-,\kappa).
	    	\end{equation*}
	    	And given $\zetaboldu\in \BPcal\times [-\radius,\radius]$, $\varphi=\varphi\bbracket{\zeta}$, $\underline{v}$, $\bar{\underline{v}}^{\pm}=\bar{\underline{v}}[\underline{\zeta}^{\pm}]$ as in \ref{def:varphishr},, we define the smooth initial surface (recall \ref{not:manifold}\ref{item:graph} and \ref{def:catebridge})
	    	\begin{equation*}
	    		M=M[\zetaboldu]=M[\zeta,\bar{\zeta}^+, \bar{\zeta}^-,\kappa]:=\Graph_{\Omega}^{\R^4}(\varphi^{gl}_+)\cup \Graph_{\Omega}^{\R^4}(-\varphi^{gl}_-)\cup\sqcup_{p\in L}\Kcech[p,\radius,\kappa], 
	    	\end{equation*}
	    	where $\radius:=\sqrt{\tau}$, $\Omega:=\Sphshrtres\setminus D_L^{\Sphshrtres}(9\radius)$ and the functions $\varphi^{gl}_{\pm}:=\varphi^{gl}\bbracket{\zeta,\bar{\zeta}^{\pm},\kappa}:\Omega\to\R$ are defined as the following.
	    	\begin{equation}
	    		\varphi^{gl}_{\pm}:=
	    		\begin{cases} 
	    			\varphi_{nl}^{\pm} & \text{on } \Sphshrtres\setminus  D_L^{\Sphshrtres}(3\delta'), \\
	    			\Psi[2\delta',3\delta';\dist^{\Sphshrtres}_p](\varphi_{cat}^{\pm}[\radius,\kappa]\circ \dist^{\Sphshrtres}_p,\varphi_{nl}^{\pm}) & \text{on }  \sqcup_{p\in L} D_p^{\Sphshrtres}(3\delta')\setminus D_p^{\Sphshrtres}(\radius).
	    		\end{cases}
	    	\end{equation}
	    \end{definition}
	    
	    \begin{notation}
	    	If $f^+$ and $f^-$ are functions supported on $\tilde{S'}$ (recall \ref{def:region}), we define $J_M(f^+, f^-)$ to be the function on $M$ supported on $(\Pi_{\Sphshrtres}|M)^{-1}\tilde{S'}$ defined by $f^+\circ\Pi_{\Sphshrtres}$ on the graph of $\varphi^{gl}_+$ and by $f^-\circ\Pi_{\Sphshrtres}$ on the graph of $\varphi^{gl}_-$.
	    \end{notation}
	    
	    Similar to \ref{prop:lineareq} we have the following linear theory.
	    \begin{proposition}\label{prop:linearshr}
	    	There is a linear map $\mathcal{R}_M:C^{0,\beta}_{\sym}(M)\to C^{2,\beta}_{\sym}(M) \oplus \skernel[L]\oplus \skernel[L]\oplus \skernele[L]\oplus \skernele[L]$, $E\mapsto (u,w_E^+,w_E^-,\bar{w}_E^+,\bar{w}_E^-)$ such that the following hold:
	    	\renewcommand{\theenumi}{\roman{enumi}}
	    	\begin{enumerate}
	    		\item $\Lcal_M u=E+J_M(w_E^+,w_E^-)+J_M(\bar{w}_{E}^+,\bar{w}_{E}^-)$.
	    		\item $\norm{u}_{2,\beta,\gamma,\gamma';M}\lesssim \norm{E}_{0,\beta,\gamma-2,\gamma';M}$ and $\abs {\mu_{E}^{\pm}}+m^{-1}\abs{\bar{\mu}_{E}^{\pm}}\lesssim m^{4-\gamma}\norm{E}_{0,\beta,\gamma-2,\gamma';M}$, where $\mu_E^{\pm},\bar{\mu}_E^{\pm}$ are defined such that $ -\tau\mu_E^{\pm} W=w_E$, $ -\tau\bar{\mu}_E^{\pm} \bar{W}=\bar{w}_E$ (recall \ref{def:VW}), the norms are defined as \ref{def:norm}.
	    	\end{enumerate}
	    \end{proposition}

	    \subsection*{The main Theorem for self-shrinkers} 
$\vspace{.162cm}$ 
\nopagebreak
	    
	    By repeating the other part of the paper, we have the following theorem.
	    \begin{theorem}\label{thmshr}
	    	If $m$ is large enough in terms of an absolute constant,
	    	then there is a $\Grp_{\R^4}[m]$-invariant embedded closed self-shrinker ${\breve{M}_{\mathrm{shr},m}}$ in $\R^4$ of topology $\#_{m^2-1}\Sph^2\times \Sph^1$   
doubling $\mathbb{S}_{\mathrm{shr}}^3\subset\R^4$.    
Moreover ${\breve{M}_{\mathrm{shr},m}}$ converges in the sense of varifolds as $m \to\infty$ to $2\Sphshrtres$.
	    \end{theorem}

	    \begin{proof}
	    	We first define $B\subset C^{2,\beta}(M[0])\times \Pcal\times\R$ by (recall \ref{eq:zetashr})
	    	\begin{equation*}
	    		B:=\{v\in C^{2,\beta}(M[0]):\norm{v}_{2,\beta,\gamma,\gamma';M[0]}\leq m^{-3-2\alpha+\gamma(1+\alpha)}\}\times\BPcal\times [-\radius,\radius].
	    	\end{equation*}
	    	
	    	With a family of diffeomorphism $\Fcal_{\zetaboldu}:M[0]\to M[\zetaboldu]$ as in \cite[Lemma 5.5]{gLD}, we define $\Jcal:B\to C^{2,\beta}(M[0])\times \Pcal\times\R$ as follows.
	    	Let $(v,\zetaboldu)\in B$, define $(u,w_H^+,w_H^-,\bar{w}_H^+,\bar{w}_H^-):=-\mathcal{R}_{M[\zetabold]}(H-J_M(w^++\bar{w}^+,w^-+\bar{w}^-))$ by \ref{prop:linearshr}, where 
	    	$w^{\pm}:=\Lcal \underline{v}\pm\kappa V$ and $\bar{w}^{\pm}:=\Lcal \underline{\bar{v}}^{\pm}$ (recall \ref{def:varphishr}). 
Define $\textup{ф}:=v\circ \mathcal{F}_{\zeta,\bar{\zeta}}^{-1}+u$ 
	    	and define $\mu_H^{\pm}, \bar{\mu}_H^{\pm}$ such that $-\tau\mu_H^{\pm}W=w_H^{\pm}$, $-\tau\bar{\mu}_H^{\pm}W=\bar{w}_H^{\pm}$.
	    	By \ref{prop:linearshr} again we define $(u_Q,w_Q^+,w_Q^-,\bar{w}_Q^+,\bar{w}_Q^-):=-\mathcal{R}_{M[\zetaboldu]}(H_{\textup{ф}}-H-\mathcal{L}_M\textup{ф})$, 
	    	and define $\mu_Q^{\pm}, \bar{\mu}_Q^{\pm}$ such that $-\tau\mu_Q^{\pm}W=w_Q^{\pm}$, $-\tau\bar{\mu}_Q^{\pm}W=\bar{w}_Q^{\pm}$.
	    	Combining the definitions we have
	    	\begin{equation*}
	    		\Lcal_M(u_Q-v\circ\Fcal_{\zetabold}^{-1})+H_{\textup{ф}}=-J_M(\tau(\mu_{sum}^+ W+ \bar{\mu}_{sum}^+ \bar{W})+\kappa,\tau( \mu_{sum}^- W+ \bar{\mu}_{sum}^- \bar{W})-\kappa),
	    	\end{equation*}
	    	where $\mu_{sum}^{\pm}:=\mu+\mu_H^{\pm}+\mu_Q^{\pm}$, $\bar{\mu}_{sum}^{\pm}:=\bar{\mu}^{\pm}+\bar{\mu}_H^{\pm}+\bar{\mu}_Q^{\pm}$. 
	    	
	        Now we define $\mathcal{J}$ by
	    	\begin{multline*}
	    		\mathcal{J}(v,\zetaboldu):=
	    		\\
	    		\Big(u_Q\circ\Fcal_{\zetaboldu},\zeta-\frac{\mu_{sum,\sym}+\bar{\mu}_{sum,\sym}}{\phi_1},
	    		\bar{\zeta}^++\frac{\bar{\mu}_{sum}^+}{\phi_1},\bar{\zeta}^-+\frac{\bar{\mu}_{sum}^-}{\phi_1},\kappa-\tau(\mu_{sum,\asym}+\bar{\mu}_{sum,\asym})\Big),
	    	\end{multline*}
	    	where $\mu_{sum}^{\sym}:=(\mu_{sum}^++\mu_{sum}^-)/2$, $\mu_{sum}^{\asym}:=(\mu_{sum}^+-\mu_{sum}^-)/2$, $\bar{\mu}_{sum}^{\sym}:=(\bar{\mu}_{sum}^++\bar{\mu}_{sum}^-)/2$, $\bar{\mu}_{sum}^{\asym}:=(\bar{\mu}_{sum}^+-\bar{\mu}_{sum}^-)/2$. 
	    	
	    	As in the proof of \ref{thm}, $B$ is convex and the embedding $B\hookrightarrow C^{2,\beta'}(M[0])\times \Pcal\times\R$ is compact for $\beta'\in(0,\beta)$. Moreover, $\Jcal$ maps $B$ to itself. The results then follows by Schauder's fixed point theorem as before.
	    \end{proof}
	

	     \bibliographystyle{plain}
	     \bibliography{paper}
\end{document}